\documentclass{amsart}
\usepackage[right=3cm, left=3cm, top=2cm]{geometry}
\usepackage{amsmath,amsthm, amsfonts} % ams stuff, unnecessary if amsart
\usepackage{amssymb}
\usepackage{enumerate}

\newtheorem{theorem}{Theorem}[section]
\newtheorem{prop}[theorem]{Proposition}
\newtheorem{cor}[theorem]{Corollary}
\newtheorem{lemma}[theorem]{Lemma}
\newtheorem{remark}[theorem]{Remark}

\theoremstyle{definition}
\newtheorem{definition}[theorem]{Definition}

\newcommand{\rr}{\mathbb{R}}
\newcommand{\nn}{\mathbb{N}}

\newcommand{\cc}{\mathbb{C}}
\newcommand{\fa}{\mathfrak{A}}
\newcommand{\C}{\mathbb{C}}
\newcommand{\hatG}{\widehat{\Gamma}}
\renewcommand{\epsilon}{\varepsilon}

\DeclareMathOperator{\range}{Range}

\DeclareMathOperator{\Span}{span}

\newcommand{\<}{\left\langle}
\renewcommand{\>}{\right\rangle}
\newcommand{\Rmnum}[1]{\expandafter\@\romancap\romannumeral #1@}
\newcommand{\rl}{\mathfrak{R}_L}
\newcommand{\fh}{\mathfrak{H}}
\newcommand{\fk}{\mathfrak{K}}
\newcommand{\fb}{\mathfrak{B}}
\newcommand{\bh}{\fb(\fh)}
\newcommand{\bk}{\fb(\fk)}
\newcommand{\ah}{\fa(\fh)}
\newcommand{\bromega}{\breve{\omega}}
\newcommand{\breta}{\breve{\eta}}
\newcommand{\brTheta}{\breve{\Theta}}
\newcommand{\bub}[1]{\mathring{#1}}
\newcommand{\iq}{internal\ }
\newcommand{\cem}{\circledast}
\newcommand{\ultralim}{\mathop{\sigma\text{-weak-lim}}}

\begin{document}

\title{E$_0$-semigroups and $q$-purity}
\author{Christopher Jankowski}
\address{Christopher Jankowski, Department of Mathematics,
Ben-Gurion University of the Negev,
P.O.B. 653, Beersheva 84105,
Israel.}
\email{cjankows@math.bgu.ac.il}
\thanks{C.J. was partially supported by the
Skirball Foundation via
the Center for Advanced Studies in Mathematics at Ben-Gurion
University of the Negev. 
}

\author{Daniel Markiewicz}
\address{Daniel Markiewicz, Department of Mathematics,
Ben-Gurion University of the Negev,
P.O.B. 653, Beersheva 84105,
Israel.}
\email{danielm@math.bgu.ac.il}

\thanks{D.M. and R.T.P. were partially supported by 
grant 2008295 from the U.S.-Israel Binational Science Foundation.
}

\author{Robert T. Powers}
\address{Robert T. Powers, Department of Mathematics, 
David Rittenhouse Lab.,
209 South 33rd St., 
Philadelphia, PA 19104-6395, U.S.A.}
\email{rpowers@math.upenn.edu}

\date{June 12, 2011}
\begin{abstract}
An E$_0$-semigroup is called $q$-pure if it is a CP-flow and its set of flow subordinates is totally ordered by subordination. The range rank of a positive boundary weight map is the
 dimension of the range of its dual map. Let $\fk$ be a separable Hilbert space. We describe all  
 $q$-pure E$_0$-semigroups of type II$_0$ which arise from boundary weight maps with range rank
 one over $\fk$. We also prove that no $q$-pure E$_0$-semigroups of type II$_0$ arise from boundary weight maps with range rank two over $\fk$.   In the case when $\fk$ is finite-dimensional, we provide a criterion to determine if two boundary weight maps of range rank one over $\fk$ give rise to cocycle conjugate $q$-pure E$_0$-semigroups.  
\end{abstract}

\maketitle

\section{Introduction}

Let $\bh$ be algebra of all bounded operators on a separable Hilbert space $\fh$. A 
CP-semigroup $\alpha=\{ \alpha_t\}$ acting on $\fb(\fh)$ is a continuous 
one-parameter semigroup of contractive completely positive maps which is continuous
 in the
 point-strong operator topology. When $\alpha_t$ is in addition an endomorphism 
for every $t>0$, then $\alpha$ is called an E-semigroup, and if furthermore
 $\alpha_t$ unital, then $\alpha$ is an E$_0$-semigroup. We recommend the monograph
 by Arveson~\cite{arv-monograph} as an excellent introduction to the theory of
 E$_0$-semigroups, and we will make use of its terminology in the remainder.

Given $\alpha$ and $\beta$ two CP-semigroups acting on $\bh$, we say that 
$\beta$ is a subordinate of $\alpha$ if $\alpha_t - \beta_t$ is a completely
 positive map for all $t>0$. Given an E$_0$-semigroup $\alpha$, let  
$\mathfrak{S}(\alpha)$ be the
 set of all its CP-semigroup subordinates, endowed with the
 partial order given by subordination. Let also $\mathfrak{E}(\alpha)$ be the 
subset of all E-semigroup subordinates of $\alpha$. Both partially ordered 
sets are easily seen to be cocycle conjugacy invariants of $\alpha$. The set
 $\mathfrak{S}(\alpha)$  was first studied by Bhat~\cite{bhat-cocycles}, 
who characterized it completely in the type I case, with the help of the quantum
 stochastic calculus in the sense of Hudson-Parthasaraty~\cite{parthasarathy1}. 
 Liebscher~\cite{liebscher} has further described an alternative presentation for
 $\mathfrak{E}(\alpha)$ in terms of subproduct systems arising from certain measure
 types with partial order given by absolute continuity.

We would like to improve our understanding of the class of E$_0$-semigroups $\alpha$
 such that $\mathfrak{S}(\alpha)$ is as small as possible.  
 When $\alpha$ is an E$_0$-semigroup and
 additionally it is a CP-flow, then we will call it  $q$-pure if and only if
 its set of CP-flow subordinates is totally
 ordered by subordination.  See Section~\ref{sec-q-purity} for a detailed
 discussion. It seems to us that $q$-pure E$_0$-semigroups will become
 important objects in the classification theory of E$_0$-semigroups. 

The class of  $q$-pure E$_0$-semigroups was first studied by Powers~\cite{powers-holyoke}.  Recall that Powers~\cite{powers-CPflows} has  proven that any spatial E$_0$-semigroup  can be constructed (up to cocycle conjugacy) by choosing an 
appropriate $q$-weight map over a separable Hilbert space $\fk$. A $q$-weight map over $\fk$ is a boundary 
weight map from the predual of $\bk$ to certain weights on a (non-closed) *-subalgebra 
$\ah\subseteq \fb(\fk \otimes L^2(0,\infty))$ satisfying the $q$-positivity conditions.
 Powers~\cite{powers-holyoke} completely described the $q$-weight maps over $\fk=\cc$ which give rise to $q$-pure E$_0$-semigroups. Later Jankowski~\cite{jankowski1} analyzed a class of 
E$_0$-semigroups of type II$_0$ arising 
from boundary weight doubles $(\phi,\nu)$ where 
where $\phi:M_n(\cc) \to M_n(\cc)$ is a $q$-positive linear map and $\nu$ is an unbounded boundary weight over 
$L^2(0,\infty)$. Jankowski characterized those boundary weight doubles giving rise to  $q$-pure E$_0$-semigroups 
when the map $\phi$ either is invertible or has rank one and in addition $\nu$ has the particular form 
$\nu(T) = (f,(I-\Lambda)^{-1/2}T(I-\Lambda)^{-1/2} f)$ where $\Lambda \in B(L^2(0,\infty))$ is the multiplication 
operator by $e^{-x}$,  $T \in \ah$ and $f\in L^2(0,\infty)$.

In this paper we generalize the results of Powers~\cite{powers-holyoke} and Jankowski~\cite{jankowski1} 
to the class $\mathfrak{C}$ of all E$_0$-semigroups of type II$_0$ arising from $q$-weight maps 
with range rank one over a separable Hilbert space $\fk$. The range rank of a  $q$-weight map 
$\omega:\bk_*\to\ah_*$ is defined to be the dimension of the range of the dual map 
(see Definition~\ref{bromega-finite-rank} and the discussion preceding it for the details). 
In particular, the class $\mathfrak{C}$ contains all $q$-weights considered earlier by Powers~\cite{powers-holyoke}  
as well as those considered by Jankowski~\cite{jankowski1} with $\phi$ rank one. 
In fact, we obtain an effective description of all $q$-weight maps in the class $\mathfrak{C}$ 
which are $q$-pure (see Theorem~\ref{rank-one-q-pure}).   
We also prove that if $\fk$ is finite-dimensional, 
then a $q$-corner between unbounded range rank one $q$-weight maps must have range rank one 
(see Theorem~\ref{corners-are-range-rank-one}). And we describe one criterion to determine if 
a $q$-corner exists (see Theorem~\ref{rank-one-corner-form}) between $q$-pure boundary weight 
maps of range rank one over $\fk$ finite-dimensional. The latter results form an important 
first step for the classification theory of $q$-pure E$_0$-semigroups. See Section~\ref{sec-rank-one}. 

We also show that if a $q$-weight map over a separable Hilbert space $\fk$ has range rank 
two, then it cannot give rise to an E$_0$-semigroup of type II$_0$ that is $q$-pure 
(Corollary~\ref{sec6-main}). In fact, when $\dim \fk=2$, we prove further that a $q$-weight 
map can only give rise to a $q$-pure E$_0$-semigroup of type II$_0$ if it has range rank 1 or 4. 
This generalizes a result of Jankowski~\cite{jankowski3}. This comprises 
Section~\ref{sec-rank-two}, which is the last section of the article.

We describe the remaining sections of the paper. 
In Section 2 we review the basic terminology and results 
necessary for the paper. In Section~\ref{sec-expectations} we introduce the concept of \emph{boundary expectations}, 
which has proved useful and convenient for the analysis of certain properties of $q$-weight maps with range rank one or two.

\section{Preliminaries}\label{sec-preliminaries}

In this article, we will consider only Hilbert spaces that are separable unless stated otherwise. We will also denote the inner product of the Hilbert space by $(\cdot, \cdot)$ and it will be taken to be conjugate linear in the \emph{first} entry.

\subsection{E$_0$-semigroups and CP-flows}

\begin{definition}  Let $\mathfrak{H}$ be a separable Hilbert space.
We say a family $\alpha = \{\alpha_t\}_{t \geq 0}$ of
normal completely positive contractions of $\mathfrak{B}(\mathfrak{H})$ into itself is a \emph{CP-semigroup} acting on
$\mathfrak{B}(\mathfrak{H})$ if:

(i)  $\alpha_s \circ \alpha_t = \alpha_{s+t}$ for all $s,t \geq 0$
and $\alpha_0 (A)=A$ for all $A \in \mathfrak{B}(\mathfrak{H})$;

(ii) For each $f,g \in \fh$ and $A \in \mathfrak{B}(\mathfrak{H})$, the inner product
$(f, \alpha_t(A)g)$ is continuous in $t$;

If $\alpha_t(I) =I$ for all $t \geq 0$, then $\alpha$ is called a \emph{unital} CP-semigroup. When
$\alpha$ is a unital CP-semigroup and in addition the map $\alpha_t$ is an endomorphism for every
$t\geq 0$, then $\alpha$ is called an \emph{E$_0$-semigroup}.
\end{definition}

We have two notions of equivalence for $E_0$-semigroups:

\begin{definition}
An $E_0$-semigroup $\alpha$ acting on $\mathfrak{B}(\mathfrak{H}_1)$ is \emph{conjugate} to
an E$_0$-semigroup $\beta$ acting on $\mathfrak{B}(\mathfrak{H}_2)$  if there exists
a $*$-isomorphism $\theta$ from $\mathfrak{B}(\mathfrak{H}_1)$ onto $\mathfrak{B}(\mathfrak{H}_2)$ such that
$\theta \circ \alpha_t = \beta_t \circ \theta$ for all $t \geq 0$.

A strongly continuous family of contractions $\mathcal{W}=\{W_t\}_{t \geq 0}$ acting on $\mathfrak{H}_2$ is
called a contractive $\beta$-cocycle if $W_t \beta_{t}(W_s)=W_{t+s}$ for all $t,s \geq 0$. A
contractive $\beta$-cocycle $W_t$ is said to be a \emph{local cocycle} if for all $A\in \mathfrak{B}(\mathfrak{H}_2)$ and
$t\geq 0$, $W_t\beta_t(A) = \beta_t(A)W_t$.

We say $\alpha$ and $\beta$ are \emph{cocycle conjugate} if there exists
a unitary $\beta$-cocycle $\{W_t\}_{t \geq 0}$ such that the E$_0$-semigroup
acting on $\mathfrak{B}(\mathfrak{H}_2)$ given by $\beta'_t(A) = W_t \beta_t(A)W_t^*$ for all $A \in \mathfrak{B}(\mathfrak{H}_2)$ and $t\geq 0$
is conjugate to $\alpha$.
\end{definition}

Let $\mathfrak{K}$ be a separable Hilbert space. We will always denote by
$\{S_t\}_{t \geq 0}$ the right shift semigroup on  $\mathfrak{K} \otimes L^2(0,
\infty)$ (which we identify with the space of
$\mathfrak{K}$-valued measurable functions on $(0, \infty)$ which are square
integrable):
\begin{equation*}
(S_t f)(x) = \left\{ \begin{matrix}
f(x-t), & x> t; \\
0, & x\leq t.
\end{matrix}
\right.
\end{equation*}
We will also denote by $E(t, \infty) = S_tS_t^*$ for all $t\geq 0$, and $E(t,s) = E(t,\infty) - E(s,\infty)$ for all $0\leq t<s<\infty$.

\begin{definition}
A CP-semigroup $\alpha$ acting on $\fb(\fk \otimes L^2(0,\infty))$
is called a \textit{CP-flow} over $\fk$ if $\alpha_t(A)S_t = S_t A$
for all $A \in \fb(\fk \otimes L^2(0,\infty))$ and $t \geq 0$.
\end{definition}

A dilation of a unital CP-semigroup $\alpha$ acting on $\mathfrak{B}(\mathfrak{K})$ is a pair $(\alpha^d, W)$, where
$\alpha^d$ is an E$_0$-semigroup acting on $\mathfrak{B}(\mathfrak{H})$ and $W:\fk \to \fh$ is an isometry such that
$\alpha^d_t(WW^*) \geq WW^*$ for $t > 0$ and furthermore
$$
\alpha_t(A) = W^*\alpha_t^d(WAW^*)W
$$
for all $A \in \mathfrak{B}(\mathfrak{K})$ and $t\geq 0$. The dilation is said to be \emph{minimal} if the
span of the vectors
$$
\alpha_{t_1}^d(WA_1W^*)\alpha_{t_2}^d(WA_2W^*)\cdots\alpha_{t_n}^d(WA_nW^*)Wf
$$
for  $f\in K, A_i \in \mathfrak{B}(\mathfrak{K}), i=1,\dots n, n\in \mathbb{N}$ is dense in $\mathfrak{H}$.
This definition of minimality is due to Arveson (see \cite{arv-monograph} for a detailed discussion
regarding dilations of CP-semigroups). We will often suppress the isometry $W$, and refer to a
minimal dilation $\alpha^d$ instead of $(\alpha^d, W)$.

\begin{theorem}[Bhat's dilation theorem]\label{theorem:Bhat's-dilation-theorem}
Every unital CP-semigroup has a minimal dilation which is unique up to conjugacy.
\end{theorem}

The following addendum by Powers (Lemma 4.50 of \cite{powers-CPflows}) further clarifies the situation for CP-flows.

\begin{theorem}
Every unital CP-flow $\alpha$ has a minimal dilation $\alpha^d$ which is also a CP-flow. We call $\alpha^d$ the minimal flow dilation of the unital CP-flow.
\end{theorem}

Given two CP-flows $\alpha$ and $\beta$ over $\mathfrak{K}$, we will say that $\alpha$ \emph{dominates} $\beta$
or that $\beta$ is a \emph{subordinate} of $\alpha$ if for all $t\geq 0$, the map $\alpha_t
-\beta_t$ is completely positive. We will often denote this relationship by $\alpha \geq \beta$.
Powers~\cite{powers-CPflows} has described a useful criterion for determining whether two CP-flows
have minimal dilations that are cocycle conjugate in terms of the next definition.

\begin{definition}\label{def-corners}
Let $\alpha$ and $\beta$ be CP-flows over $\mathfrak{K}_1$ and $\mathfrak{K}_2$, respectively. For $j=1,2$, let
$\mathfrak{H}_j=\mathfrak{K}_j \otimes L^2(0,\infty)$ and let $S_t^{(j)}$ denote the right shift on $\mathfrak{H}_j$. Let $\gamma=\{\gamma_t: t\geq 0\}$ be a family of maps from $\mathfrak{B}(\mathfrak{H}_2, \mathfrak{H}_1)$ into itself and define for each $t>0$, $\gamma_t^*: \mathfrak{B}(\mathfrak{H}_1, \mathfrak{H}_2) \to \mathfrak{B}(\mathfrak{H}_1, \mathfrak{H}_2)$ by $\gamma_t^*(C)=[\gamma_t(C^*)]^*$ for all $C\in \mathfrak{B}(\mathfrak{H}_1, \mathfrak{H}_2)$. We say that $\gamma$ is a \emph{flow corner} from
$\alpha$ to $\beta$ if the maps
$$
\Theta_t \begin{bmatrix}
A & B \\
C & D
\end{bmatrix} =
\begin{bmatrix}
\alpha_t(A) & \gamma_t(B) \\
\gamma_t^*(C) & \beta_t(D)
\end{bmatrix}
$$
define a CP-flow $\Theta=\{ \Theta_t : t\geq 0\}$ over $\mathfrak{K}_1 \oplus \mathfrak{K}_2$ with respect to the shift
$S_t^{(1)}\oplus S_t^{(2)}$. Note that $\gamma$ is a flow corner from $\alpha$ to $\beta$ if and only if $\gamma^*$ is a flow corner from $\beta$ to $\alpha$.

A flow corner $\gamma$ is called a \emph{hyper-maximal
flow corner} if every subordinate CP-flow $\Theta'$ of $\Theta$ of the form
$$
\Theta_t' \begin{bmatrix}
A & B \\
C & D
\end{bmatrix} =
\begin{bmatrix}
\alpha_t'(A) & \gamma_t(B) \\
\gamma_t^*(C) & \beta_t'(D)
\end{bmatrix}
$$
for $t\geq 0$ must satisfy $\alpha_t'=\alpha_t$ and $\beta_t'=\beta_t$ for all $t\geq 0$.
\end{definition}

\begin{theorem}\label{hyperflowcorn}  Suppose $\alpha$ and $\beta$ are unital
$CP$-flows over
$\mathfrak{K}_1$ and $\mathfrak{K}_2$ and $\alpha^d$ and $\beta^d$ are their minimal dilations
to $E_0$-semigroups.  Suppose $\gamma$ is a hyper maximal flow corner from $\alpha$
to $\beta$.  Then $\alpha^d$ and $\beta^d$ are cocycle conjugate.  Conversely,
if $\alpha^d$ is a type II$_0$ and $\alpha^d$ and $\beta^d$ are cocycle conjugate,
then there is a hyper maximal flow corner from $\alpha$ to $\beta$.
\end{theorem}

\subsection{Boundary weight maps} For the remainder of this section, let $\mathfrak{K}$ be a fixed separable Hilbert space
(not necessarily infinite-dimensional) and let $\mathfrak{H}=\mathfrak{K} \otimes L^2(0,\infty)$.

Define $\Lambda: \mathfrak{B}(\mathfrak{K}) \rightarrow \mathfrak{B}(\mathfrak{H})$ by
\begin{equation*}
(\Lambda(A)f)(x)=e^{-x} Af(x)
\end{equation*}
and let $\mathfrak{A}(\mathfrak{H})$ be the algebra
\begin{equation*}
\mathfrak{A}(\mathfrak{H}) = [I - \Lambda(I_\fk)]^\frac{1}{2} \mathfrak{B}(\mathfrak{H})  [I -
\Lambda(I_\fk)]^\frac{1}{2}.
\end{equation*}

\begin{definition}
We say that a linear functional $\mu:\mathfrak{A}(\mathfrak{H}) \to \C$
is a \textit{boundary weight}, denoted $\mu \in \mathfrak{A}(\mathfrak{H})_*$, if the functional
$\ell$ defined on $\mathfrak{B}(\mathfrak{H})$ by
\begin{equation*} \ell(A)= \mu\Big(
[I - \Lambda(I_\fk)]^\frac{1}{2} A [I - \Lambda(I_\fk)]^\frac{1}{2} \Big)
\end{equation*}
is a normal bounded linear functional.  The boundary weight $\mu$ is called \emph{bounded} if there exists $C>0$ such that $|\mu(T)| \leq C \|T\| $ for all $T \in \mathfrak{A}(\mathfrak{H})$. Otherwise, $\mu$ is called \emph{unbounded}.

A linear map from $\mathfrak{B}(\mathfrak{K})_*$ to $\mathfrak{A}(\mathfrak{H})_*$ will be called a \emph{boundary weight map}.
\end{definition}
Boundary weights were first defined in \cite{powers-CPflows} (Definition 4.16), where
their relationship to CP-flows was explored in depth.  For an additional discussion of boundary
weights and their properties, we refer
the reader to Definition 1.10 of \cite{markiewicz-powers} and its subsequent remarks.

Given a normal map $\phi: \mathfrak{B}(\mathfrak{H}) \to \mathfrak{B}(\mathfrak{K})$, we will denote by $\hat{\phi}:\mathfrak{B}(\mathfrak{K})_* \to \mathfrak{B}(\mathfrak{H})_*$ the
predual map satisfying $\rho(\phi(A)) = (\hat{\phi}(\rho))(A)$ for all $A \in \mathfrak{B}(\mathfrak{H})$ and $\rho \in \mathfrak{B}(\mathfrak{K})_*$.

Define $\Gamma : \mathfrak{B}(\mathfrak{H}) \to \mathfrak{B}(\mathfrak{H})$ by the weak* integral
\begin{equation}\label{gamma}
\Gamma(A) = \int_0^\infty e^{-t} S_t A S_t^* dt.
\end{equation}
The following records facts that are implicit in the proof of 
Theorem~4.17 in Powers~\cite{powers-CPflows} (for a proof, see Proposition~2.11 of \cite{jankowski-markiewicz}):

\begin{prop}\label{get-omega}
Let $\mu \in\fa(\fh)_*$ be a boundary weight. We have that for all $T\in\fa(\fh)$,
$$
\mu(T) = \lim_{x\to 0+} \mu\big(E(x,\infty) T E(x,\infty)\big).
$$
In particular $\mu=\mu'$ if and only if for all $x>0$ and $T \in E(x,\infty)\mathfrak{B}(\mathfrak{H})E(x,\infty)$, we have that $\mu(T) = \mu'(T)$. Furthermore, given $x>0$ and $T \in E(x,\infty)\mathfrak{B}(\mathfrak{H})E(x,\infty)$,
\begin{equation}\label{mut}
\mu(T) = \lim_{y\to x+} \frac{1}{y-x} \hatG(\mu) \Big(T - e^{x-y} S_{y-x} T S_{y-x}^*\Big).
\end{equation}
\end{prop}

If $\alpha$ is a CP-flow over $\mathfrak{K}$, we define its resolvent by the weak* integral
\begin{equation} \label{resolvent-pure}
R_\alpha(A) = \int_0^\infty e^{-t} \alpha_t(A) dt
\end{equation}
defined for $A\in \mathfrak{B}(\mathfrak{H})$.  Powers~\cite{powers-CPflows} proved that there exists a completely positive boundary weight map $\omega:  \mathfrak{B}(\mathfrak{K})_* \to \mathfrak{A}(\mathfrak{H})_*$ such that
\begin{equation} \label{resolvent}
\hat{R}_{\alpha}(\eta) = \hat{\Gamma} (\omega(\hat{\Lambda}\eta) + \eta)
\end{equation}
and $\omega(\rho)(I-\Lambda(I_\fk)) \leq \rho(I_\fk)$ for all $\rho \in \mathfrak{B}(\mathfrak{K})_*$ positive. Such a boundary weight map is uniquely determined by \eqref{resolvent} in combination with Proposition~\ref{get-omega}, and in fact for all $\rho \in \mathfrak{B}(\mathfrak{K})_*$, $x>0$ and $T \in E(x,\infty)\mathfrak{B}(\mathfrak{H})E(x,\infty)$,
\begin{equation}\label{got-omega}
\omega(\rho)(T) = \lim_{y \to x+} \frac{1}{y-x}
(\widehat{R}_\alpha - \widehat{\Gamma})(\eta) ( T - e^{x-y} S_{y-x}TS_{y-x}^*),
\end{equation}
where $\eta \in \mathfrak{B}(\mathfrak{H})_*$ is any normal functional such that $\rho=\widehat{\Lambda}(\eta)$. Such a functional exists since $\Lambda$ is isometric hence $\widehat{\Lambda}$ is onto.

The map $\omega$ is called \emph{the boundary weight map associated to $\alpha$}.  

The following result, which is a compilation of Theorems 4.17, 4.23, and 4.27 of
\cite{powers-CPflows}, describes the converse relationship between boundary weight maps and
CP-flows.

\begin{theorem}\label{powerstheorem}
Let $\omega: \mathfrak{B}(\mathfrak{K})_* \to \mathfrak{A}(\mathfrak{H})_*$ be a completely positive map satisfying
$\omega(\rho)(I-\Lambda(I_\fk)) \leq \rho(I_\fk)$ for all positive
$\rho$.  Let $\{S_t\}_{t \geq 0}$ be the right shift semigroup acting on $\mathfrak{H}$.
For each $t>0$, define the truncated boundary weight map
$\omega \vert_t: \mathfrak{B}(\mathfrak{K})_* \to \mathfrak{B}(\mathfrak{H})_*$ by
\begin{equation*}
\omega \vert_t(\rho)(A)= \omega(\rho)\big(E(t,\infty) A E(t,\infty)\big)
\end{equation*}
If for every $t>0$, the map $(I +\hat{\Lambda}\omega\vert_t)$ is invertible  and furthermore the map
\begin{equation}
\label{dmb1} \hat{\pi}_t : = \omega\vert_t(I + \hat{\Lambda}\omega\vert_t)^{-1}
\end{equation}
is a completely positive contraction from $\mathfrak{B}(\mathfrak{K})_*$ into $\mathfrak{B}(\mathfrak{H})_*$, then $\omega$ is the boundary weight map associated to a CP-flow over $\mathfrak{K}$.
The CP-flow is unital if and only if $\omega(\rho)(I-\Lambda(I_\fk))
= \rho(I_\fk)$ for all $\rho \in \mathfrak{B}(\mathfrak{K})_*$.
\end{theorem}

\begin{definition}\label{def-breve}
Let $\omega: \mathfrak{B}(\mathfrak{K})_* \to \mathfrak{A}(\mathfrak{H})_*$  be a completely positive boundary weight map
satisfying
$\omega(\rho)(I-\Lambda(I_\fk)) \leq \rho(I_\fk)$ for all positive
$\rho$. If for every $t>0$ the map $\hat{\pi}_t$ as defined in the statement of Theorem~\ref{powerstheorem} exists and it is a completely positive contraction, then $\omega$ is called a \emph{$q$-weight map
over $\fk$}. In that case, the
family $\pi_t^{\#}$ (for $t>0$) of completely positive normal contractions from $\mathfrak{B}(\mathfrak{H})$ to $\mathfrak{B}(\mathfrak{K})$ is called the \emph{generalized boundary representation} associated to $\omega$, or alternatively to the CP-flow
associated to $\omega$.  For every $t>0$, we have
\begin{equation} \label{dmb2} \omega\vert_t = \hat{\pi}_t(I - \hat{\Lambda} \hat{\pi}_t)^{-1}. \end{equation}
We say a $q$-weight map $\omega$ is \emph{unital} if it induces a unital CP-flow.  By
Theorem~\ref{powerstheorem}, $\omega$ is unital if and only
if $\omega(\rho)(I - \Lambda(I_\fk))= \rho(I_\fk)$ for all $\rho \in \mathfrak{B}(\mathfrak{K})_*$. 

 We will say that a $q$-weight $\omega:\fb(\fk)_* \to \fa(\fh)_*$ is \emph{bounded} if for every $\rho \in \fb(\fk)_*$ the linear functional $\omega(\rho)$ extends to a $\sigma$-weakly continuous linear functional on $\fb(\fh)$.
\end{definition}

We note that, in order to check that $\pi^\#_t$ is a completely positive contraction for all $t>0$, it suffices to check for small $t$ in the sense that if $\pi^\#_s$ is a completely positive contraction for some $s>0$, then for all $t>s$,  $\pi^\#_t$ is automatically a completely positive contraction.

 If $\omega: \fb(\fk)_* \rightarrow \fa(\fh)_*$ is a completely positive boundary weight map, 
then we have a well-defined completely positive map $\bub{\omega}:\fb(\fk)_*\to \fb(\fh)_*$ given by
$$
\bub{\omega}(\rho)(A) = \omega(\rho)\big((I-\Lambda)^\frac{1}{2} A (I-\Lambda)^\frac{1}{2})\big) 
\qquad \forall \rho \in\fb(\fk)_*, \forall A \in \fb(\fh).
$$ 
By an argument analogous to the proof of continuity of positive linear functionals C$^*$algebras, the positivity of $\bub{\omega}$ implies that it is bounded.

Since $\bub{\omega}:\fb(\fk)_*\to \fb(\fh)_*$ is a bounded linear map, it induces
a normal dual map $\bub{\omega}':\fb(\fh) \to \fb(\fk)$ satisfying
$$
\rho(\bub{\omega}'(A))= \bub{\omega}(\rho)(A) \qquad \forall A \in \fb(\fh).
$$  Observe
that $\bub{\omega}'$ is completely positive since $\bub{\omega}$ has that property.

Since the map $\fb(\fh) \to \fa(\fh)$ given by $A\mapsto (I - \Lambda)^\frac{1}{2} A (I - \Lambda)^\frac{1}{2}$ 
is one-to-one and onto, there is a unique 
linear map $\bromega: \fa(\fh) \rightarrow \fb(\fk)$ satisfying
\begin{equation}\label{dualized}
\bromega\big((I - \Lambda)^\frac{1}{2} A (I - \Lambda)^\frac{1}{2}\big) = \bub{\omega}'(A) \qquad \forall A \in \fb(\fh).
\end{equation}

\begin{definition} \label{bromega-finite-rank}
Let $\omega: \fb(\fk)_* \to \fa(\fh)_*$ be a completely positive boundary weight map.
We define the \emph{dualized boundary weight map} $\bromega:\fa(\fh) \to \fb(\fk)$ to be the unique map satisfying equation~\eqref{dualized} or, alternatively, 
$$\rho(\bromega(B))= \omega(\rho)(B) \qquad \forall \rho \in \fb(\fk)_*, \forall B \in \fa(\fh).$$
Similarly, for every $t>0$, there exists a
unique normal map $\breve{\omega}\vert_t: \mathfrak{B}(\mathfrak{H}) \rightarrow \fb(\fk)$ such that 
$$
\rho(\breve{\omega}\vert_t(A))=\omega\vert_t(\rho)(A)
$$ 
for every $\rho \in \fb(\fk)_*$ and $A \in \fb(\fh)$ or, alternatively,
$\breve{\omega}\vert_t(A)= \breve{\omega}\big(E(t,\infty) A E(t,\infty)\big)$ for all $t>0$ and $A \in \fb(\fh)$. 

We will say that a $q$-weight map $\omega:\fb(\fk)_* \to \fa(\fh)_*$ has \emph{finite  range rank}  if $\range(\bromega) \subseteq \fb(\fk)$ is finite-dimensional. In this case, we will say that the \emph{range rank} of $\omega$ is the dimension of $\range(\bromega)$. Of course, if $\dim \fk < \infty$, then $\omega$ automatically has finite range rank.
\end{definition}

In the next result proven by
Powers~\cite{powers-CPflows} we recall the criterion for subordination in terms of the generalized
boundary representation.

\begin{theorem}\label{boundary-representation-subordinates}
Let $\alpha$ and $\beta$ be CP-flows acting on $\mathfrak{B}(\mathfrak{H})$ with generalized 
boundary representations $\pi_t^{\#}$ and $\xi_t^{\#}$, respectively. Then $\alpha \geq \beta$
if and only if $\pi_t^{\#}-\xi_t^{\#}$ is completely positive for all $t>0$.  Also if $\pi_s^\# \geq \xi_s^\#$ for some $s>0$, then $\pi_t^\# \geq \xi_t^\#$ for all $t\geq s$, so one only has to check for a sequence $(t_n)_{n\in\nn}$ tending to zero.
\end{theorem}

We can deduce from Theorem~\ref{boundary-representation-subordinates} that there is a bijective correspondence
between CP-flows and $q$-weight maps:  Let $\alpha$ and $\beta$ be CP-flows,
with associated $q$-weight maps $\omega$ and $\eta$ and generalized boundary representations 
$\{\pi_t ^{\#}\}_{t>0}$ and $\{\xi_t ^{\#}\}_{t>0}$, respectively.
By Theorem~\ref{boundary-representation-subordinates}, $\alpha = \beta$ if and only if
$\pi_t^{\#} = \xi_t^{\#}$ for every $t>0$.  By equations \eqref{dmb1} and \eqref{dmb2},
this holds if and only if $\omega\vert_t = \eta\vert_t$ for all $t>0$.  By Proposition~\ref{get-omega},
we have $\omega\vert_t = \eta\vert_t$ for all $t>0$ if and only if $\omega = \eta$.  Therefore, 
$\alpha = \beta$ if and only if $\omega = \eta$.

The index of the E$_0$-semigroup induced by a CP-flow turns out to be the rank of an associated map which is called its normal spine:

\begin{theorem}\label{normal-spine}
Let $\alpha$ be a CP-flow over $\fk$ with generalized boundary representation $\pi_t^\#$. Then for every $A \in \cup_{t >0} S_t \fb(\fh) S_t^*$, we have that $\pi_t^\#(A)$ converges $\sigma$-strongly to an element denoted by $\pi_0^\#(A)$. Furthermore, this map extends uniquely to a map $\pi_0^\#: \fb(\fh) \to \fb(\fk)$, called the normal spine of $\alpha$, which is a $\sigma$-weakly continuous completely positive contraction. The index of $\alpha^d$ is equal to the rank of $\pi_0^\#$ as a completely positive map.
\end{theorem}

We note that in the particular case when $\alpha$ is a CP-flow which induces an E$_0$-semigroup type II$_0$, then it follows that its normal spine $\pi_0^\#=0$.

\subsection{Generalized Schur maps} \label{generalized-Schur}

Recall that a map $\phi: M_n(\cc) \to M_n(\cc)$ is said to be a \emph{Schur map} if there exists a matrix $Q=(q_{ij}) \in M_n(\cc)$ such that
$$
\phi \big( (x_{ij}) \big) = ( q_{ij} x_{ij} )
$$
In this section we review the concept and notation associated with generalized Schur maps introduced in \cite{jankowski-markiewicz},  and which will be used in the remainder of the paper.

For each $i=1, 2, \dots, n$, let $\mathfrak{K}_i$ and $\mathfrak{H}_i$ be Hilbert spaces, and let $\mathfrak{K}=\bigoplus_{i=1}^n \mathfrak{K}_i$ and $\mathfrak{H}=\bigoplus_{i=1}^n \mathfrak{H}_i$. Let for $i=1,\dots, n$, $V_i:\mathfrak{K}_i \to \fk$ and $W_i:\mathfrak{H}_i \to \fh$ be the canonical isometries. Given operators $A \in \mathfrak{B}(\mathfrak{K})$ and $B \in \mathfrak{B}(\mathfrak{H})$, and  for $i,j=1,2, \dots, n$ given operators  $X \in \mathfrak{B}(\mathfrak{K}_j, \mathfrak{K}_i), Z \in \mathfrak{B}(\mathfrak{H}_j, \mathfrak{H}_i)$, we define
\begin{align*}
A_{ij} & = V_i^*AV_j \in \mathfrak{B}(\mathfrak{K}_j, \mathfrak{K}_i) &  X^{ij} & = V_iXV_j^* \in \mathfrak{B}(\mathfrak{K}) \\
B_{ij} & = W_i^*BW_j \in \mathfrak{B}(\mathfrak{H}_j, \mathfrak{H}_i) & Z^{ij} & = W_iZW_j^* \in \mathfrak{B}(\mathfrak{H})
\end{align*}
In particular,
$$
(X^{ij})_{rs} = \delta_{ir} \delta_{js} X.
$$

Given a subalgebra $\fa$ of $\mathfrak{B}(\mathfrak{H})$, and for each $i,j=1,2, \dots, n$, let $\fa_{ij} = W_i^*\fa W_j$. Suppose that for all $i,j=1,2, \dots, n$,
\begin{equation}\label{contains-corners}
W_i\fa_{ij}W_j^* \subseteq \fa.
\end{equation}
Given a linear map $\phi: \fa \to \mathfrak{B}(\mathfrak{K})$, for each $i,j=1,2, \dots, n$ we define the linear map
$\phi_{ij} : \fa _{ij} \to \mathfrak{B}(\mathfrak{K}_j, \mathfrak{K}_i)$ given by
$$
\phi_{ij}(X) = [\phi(X^{ij})]_{ij}
$$
We say that $\phi$ is a \emph{generalized Schur map} with respect to the decompositions $\bigoplus_{i=1}^n \mathfrak{K}_i$ and $\bigoplus_{i=1}^n \mathfrak{H}_i$ if for all $A \in \fa$,
$$
[\phi (A)]_{ij} = \phi_{ij}( A_{ij}) .
$$
In particular, if $\phi$ is a generalized Schur map and if $X \in \mathfrak{B}(\mathfrak{K}_j, \mathfrak{K}_i)$, then
$$
\phi(X^{ij}) = [\phi_{ij}(X)]^{ij}.
$$

A similar definition applies to maps from $\mathfrak{B}(\mathfrak{K})_*$ to the algebraic dual $\fa'$. If $\rho \in \mathfrak{B}(\mathfrak{K})_*$ and $\eta \in \fa'$, we define for each $i,j=1,2, \dots, n$ the linear functionals $\rho_{ij} \in \mathfrak{B}(\mathfrak{K}_j, \mathfrak{K}_i)'$ and $\eta_{ij} \in \fa_{ij}'$ given by
$$
\rho_{ij} (X) = \rho(X^{ij}), \qquad \eta_{ij}(Z) = \eta(Z^{ij}),
$$
for all $X \in \mathfrak{B}(\mathfrak{K}_j, \mathfrak{K}_i)$ and $Z \in \fa_{ij}$. For each $\mu \in \mathfrak{B}(\mathfrak{K}_j, \mathfrak{K}_i)'$, we define $\mu^{ij} \in \mathfrak{B}(\mathfrak{K})'$ given by
$$
\mu^{ij}(A) = \mu(A_{ij}).
$$
Given a map $\Psi: \mathfrak{B}(\mathfrak{K})_* \to \fa'$ and $i,j=1,2, \dots, n$, we define $\Psi_{ij} : \mathfrak{B}(\mathfrak{K}_j, \mathfrak{K}_i)' \to \fa_{ij}'$ by
$$
\Psi_{ij}(\mu) = [\Psi(\mu^{ij})]_{ij}.
$$
We say that $\Psi : \mathfrak{B}(\mathfrak{K})_* \to \fa'$ is a \emph{generalized Schur map with respect to the decompositions} $\bigoplus_{i=1}^n \mathfrak{K}_i$ and  $\bigoplus_{i=1}^n \mathfrak{H}_i$ if
$$
[\Psi(\rho)]_{ij} = \Psi_{ij}(\rho_{ij}).
$$
We observe that if $\Psi$ is a generalized Schur map and $\rho \in \mathfrak{B}(\mathfrak{K})_*$, then
$\Psi([\rho_{ij}]^{ij})=[\Psi_{ij}(\rho_{ij})]^{ij}$.

\subsection{Powers weights and boundary weight doubles}
For a moment, let us examine the 
case when $\omega$ is a $q$-weight map over $\C$.
Then $\omega$
is determined by its value $\omega_1:=\omega(1)$, and it induces a CP-flow $\alpha$ over $\cc$ if and only  
$\omega_1$ is a positive boundary weight and $\omega_1(I-\Lambda)\leq1$. In that case,
the CP-flow $\alpha$ is unital if and only if $\omega_1(I-\Lambda)=1$, and
therefore dilates to an $E_0$-semigroup $\alpha^d$.

\emph{Since all the key properties of $\omega$ are determined by the single boundary weight $\omega_1$ 
in the special case $\mathfrak{K}=\cc$, we will write $\omega$ instead of $\omega_1$.}

Results from
\cite{powers-CPflows} show that $\alpha^d$ is of type I if $\omega_1$
is bounded and of type II$_0$ if $\omega_1$ is unbounded.  This leads to the following definition.

\begin{definition} \label{powers-weight} A boundary weight $\nu \in \mathfrak{A}(L^2(0, \infty))_*$ is called
a \emph{Powers
weight} if $\nu$ is positive and $\nu(I-\Lambda)=1$. We say that a Powers weight $\nu$ is
\emph{type I}
if it is bounded and \emph{type II} if it is unbounded.
\end{definition}

If $\nu$ is a Powers weight, then it has the form:
\begin{equation*}
\nu
\Big((I - \Lambda)^\frac{1}{2}A
(I - \Lambda)^\frac{1}{2}\Big) = \sum_{i=1}^k (f_i, Af_i) \end{equation*} for some mutually
orthogonal nonzero $L^2$-functions $\{f_i\}_{i=1}^k$ ($k \in \mathbb{N} \cup
\{\infty\}$) with $\sum_{i=1}^k \|f_i\|^2=1$.  We note that if $\nu$ is a
type II Powers weight, then for the weights $\nu_t$ defined by
$\nu_t(A)=\nu(E(t,\infty)AE(t,\infty))$ for $A \in B(L^2(0, \infty))$
and $t>0$, both $\nu_t(I)$ and $\nu_t(\Lambda)$
approach infinity as $t \rightarrow 0+$.

In \cite{powers-holyoke}, Powers defined $q$-corners and hyper-maximal $q$-corners, and  determined necessary and sufficient conditions for cocycle conjugacy between E$_0$-semigroups arising 
from type II Powers weights. In the following, we will generalize several of the definitions and results obtained in \cite{powers-holyoke}, such as Definition 3.11 and Theorems 3.9, 3.10 and 3.14.

\subsection{Comparison theory for $q$-weight maps}
Suppose that $\omega:\bk_* \to \ah_*$ is a $q$-weight map which induces a CP-flow $\alpha$ with generalized boundary representation $\pi_t^\#$. It will important to describe the subordinates of $\alpha$ in terms of $\omega$. 

\begin{definition}
A $q$-weight map $\eta$ with associated generalized boundary representation $\xi^\#_t$ is called a \emph{$q$-subordinate} of $\omega$ if for all $t>0$ we have $\xi_t^\# \leq \pi_t^\#$. We will denote this relation by $\eta \leq_q \omega$.
\end{definition}
In view of Theorem~\ref{boundary-representation-subordinates}, it is clear that $q$-subordination of $q$-weight maps is equivalent to subordination  for the associated CP-flows. We will also make use of the following fact.

\begin{prop}\label{q-subordinate-at-1}
Let $\omega:\bk_* \to \ah_*$ be a $q$-weight map and let $\eta$ be a $q$-subordinate of $\omega$. Then for all positive $\rho \in \bk_*$ and positive $T \in \ah$,
\begin{equation}\label{positive-diff}
\omega(\rho)(T) \geq \eta(\rho)(T).
\end{equation}
Furthermore, if $\omega(\rho)(I-\Lambda) = \eta(\rho)(I-\Lambda)$ for all $\rho \in \bk_*$, then $\omega = \eta$.
\end{prop}
\begin{proof} Let $\pi_t^\#$ and $\xi_t^\#$ be the generalized boundary representations for $\omega$ and $\eta$ respectively. By Theorem~\ref{boundary-representation-subordinates}, we have that for all $t>0$, $\pi_t^\# - \xi_t^\#$ is completely positive. Therefore,
$$
\omega\vert_t = \hat{\pi}_t \sum_{n=0}^\infty (\hat{\Lambda} \hat{\pi}_t)^n \geq \hat{\xi}_t  \sum_{n=0}^\infty (\hat{\Lambda} \hat{\xi}_t)^n = \eta\vert_t
$$
for all $t>0$ (the series converge, at every positive functional $\rho$ in the sense of weights). The inequality is in the completely positive sense. Therefore, we have that for all $T\in\ah$ positive, and for all $\rho \in \bk_*$, by Proposition~\ref{get-omega},
$$
\omega(\rho)(T) - \eta(\rho)(T) = \lim_{t\to 0+} \omega\vert_t(\rho)(T) - \eta\vert_t(\rho)(T)  \geq 0.
$$
Thus we have that $\omega \geq \eta$ (in the positive sense). Therefore, we have that $\bub{\omega} - \bub{\eta}$ is positive as a map from $\bk_*$ to $\bh_*$, or alternatively,
$\bub{\omega}(\rho) - \bub{\eta}(\rho)$ is a positive normal functional for all positive $\rho \in \bk_*$. Now notice that for all $\rho\in \bk_*$, 
$$
\bub{\omega}(\rho)(I) = \omega(\rho)(I-\Lambda) = \eta(\rho)(I-\Lambda) = \bub{\eta}(\rho)(I).
$$
Therefore, it follows that for every positive $\rho \in \bk_*$, $\bub{\omega}(\rho) = \bub{\eta}(\rho)$. Now by considering linear combinations, $\bub{\omega}=\bub{\eta}$, hence $\omega=\eta$.
\end{proof}

Now suppose that $\omega_i$ is a unital $q$-weight map for $i=1,2$, so $\omega_i$ induces a unital CP-flow $\alpha_i$.
By Theorem~\ref{theorem:Bhat's-dilation-theorem}, $\alpha_i$ has a minimal dilation to an $E_0$-semigroup
$\alpha_i ^d$ which is unique up to conjugacy.  A fundamental question to ask is whether $\alpha_1 ^d$
and $\alpha_2 ^d$ are cocycle conjugate.  Theorem~\ref{hyperflowcorn} gives us a partial answer in terms of
flow corners from $\alpha_1$ to $\alpha_2$.  In this section, we translate this description
in terms of boundary weight maps.
We will use the following notation.  If $\mathfrak{K}_i$ is a separable Hilbert space and $\mathfrak{H}_i = \mathfrak{K}_i \otimes L^2(0, \infty)$
for $i=1,2$, then we denote by $\fa(\mathfrak{H}_i, \mathfrak{H}_j)$ the vector space
$$\fa(\mathfrak{H}_i, \mathfrak{H}_j) = [I - \Lambda(I_{\mathfrak{K}_j})]^{1/2} \mathfrak{B}(\mathfrak{H}_i, \mathfrak{H}_j) [I - \Lambda(I_{\mathfrak{K}_i})]^{1/2}.$$
Let $\mathfrak{K}=\mathfrak{K}_1 \oplus \mathfrak{K}_2$ and $\mathfrak{H}= \mathfrak{H}_1 \oplus \mathfrak{H}_2$.  We define $\mathfrak{B}(\mathfrak{K}_j, \mathfrak{K}_i)_*$ to be the closed
subspace of $\mathfrak{B}(\mathfrak{K}_j, \mathfrak{K}_i)'$ given by
$$
\mathfrak{B}(\mathfrak{K}_j, \mathfrak{K}_i)_* = \{ \rho \in \mathfrak{B}(\mathfrak{K}_j, \mathfrak{K}_i)' : \rho^{ij} \in \mathfrak{B}(\mathfrak{K})_*\},$$
and similarly we define the vector space
$$
\fa(\mathfrak{H}_j, \mathfrak{H}_i)_* = \{ \nu \in \fa(\mathfrak{H}_j, \mathfrak{H}_i)' : \nu^{ij} \in \fa(\fh)_*\}.
$$

For the sake of clarity, we will frequently write generalized Schur maps in matrix form.  For example,
if $\rho \in \mathfrak{B}(\mathfrak{K}_1 \oplus \mathfrak{K}_2)'$, so $\rho = \sum_{i,j=1}^2 (\rho_{ij})^{ij}$, we write it as
$$\rho=
\begin{pmatrix}
\rho_{11} & \rho_{12} \\ \rho_{21} & \rho_{22}
\end{pmatrix},
$$
and if $\omega$ is a boundary weight map over $\mathfrak{K}_1 \oplus \mathfrak{K}_2$ which is also a generalized Schur map,
we denote it by
$$\omega(\rho)=
\begin{pmatrix}
\omega_{11}(\rho_{11}) & \omega_{12}(\rho_{12}) \\ \omega_{21}(\rho_{21}) & \omega_{22}(\rho_{22})
\end{pmatrix},
$$
or write $\omega$ in the more abbreviated form
$$\omega=
\begin{pmatrix}
\omega_{11} & \omega_{12} \\ \omega_{21} & \omega_{22}
\end{pmatrix}.
$$

\begin{definition} \label{definition-q-corner}
For $i=1,2$, let $\mathfrak{K}_i$ be a separable Hilbert space, and let $\mathfrak{H}_i = \mathfrak{K}_i \otimes L^2(0, \infty)$.
Suppose $\mu: \mathfrak{B}(\mathfrak{K}_1)_* \to \fa(\mathfrak{H}_1)_*$ and $\eta: \mathfrak{B}(\mathfrak{K}_2)_* \to \fa(\mathfrak{H}_2)_*$ are $q$-weight maps.
We say that a map $\ell$ from $\mathfrak{B}(\mathfrak{K}_2, \mathfrak{K}_1)_*: \rightarrow \fa(\mathfrak{H}_2, \mathfrak{H}_1)_*$ is a \emph{$q$-corner} from $\mu$ to $\eta$ if
$\omega: \mathfrak{B}(\mathfrak{K}_1 \oplus \mathfrak{K}_2)_* \rightarrow \mathfrak{A}(\mathfrak{H}_1 \oplus \mathfrak{H}_2)_*$ defined by 
$$
\omega(\rho)= \begin{pmatrix} \mu(\rho_{11}) & \ell(\rho_{12}) \\ \ell^*(\rho_{21})
& \eta(\rho_{22}) \end{pmatrix}  \qquad \forall \rho \in \mathfrak{B}(\mathfrak{K}_1 \oplus \mathfrak{K}_2)_*
$$
is a $q$-weight map over $\fk_1 \oplus \fk_2$.  We say $\ell$ is a \emph{hyper-maximal $q$-corner} from $\mu$ to $\eta$ if, whenever 
$$ \omega \geq_q \begin{pmatrix} \mu' & \ell\\ \ell^*
& \eta' \end{pmatrix} \geq_q 0,
$$
we have $\mu=\mu'$ and $\eta=\eta'$.
\end{definition}

The following result has a straightforward proof using the techniques of generalized Schur maps introduced in \cite{jankowski-markiewicz}, which we omit.

\begin{theorem}
Suppose $\alpha$ and $\beta$ are unital CP-flows over $\mathfrak{K}_1$ and $\mathfrak{K}_2$ with boundary weight maps
$\mu$ and $\eta$, respectively.  If there is a hyper-maximal $q$-corner from 
$\mu$ to $\eta$, then $\alpha^d$ and $\beta^d$ are cocycle conjugate.  Conversely,
if $\alpha^d$ is a type II$_0$ E$_0$-semigroup and $\alpha^d$ and $\beta^d$ are cocycle conjugate,
then there is a hypermaximal $q$-corner from $\mu$ to $\eta$.
\end{theorem}

\section{Boundary expectations}\label{sec-expectations}

\begin{definition} \label{boundary-expectation-definition}
Let $\omega$ be a $q$-weight map over a separable Hilbert space $\fk$. We will say that a map 
$L:\bk \to \bk$ is a \emph{boundary expectation} corresponding to $\omega$ 
(or alternatively the CP-flow it induces) if it satisfies the following properties:
\begin{enumerate}[(i)]
\item $L$ is completely positive;
\item $L \circ \breve{\omega} = \breve{\omega}$;
\item $\range(L) = \range(\breve{\omega})$;
\item $L^2=L$ and $\|L\|=1$.
\end{enumerate}
\end{definition}

We note that we do not know if a boundary expectation always exists, 
but in general even when one exists it need not be unique. 
However, as we establish in the following, a boundary expectation always exists when 
$\omega$ is a $q$-weight map with finite range rank whose normal spine is zero.

Let $\fk$ be a separable Hilbert space. Recall that the relative BW-topology on the set $\mathfrak{C}(\bk) = \{ \phi: \bk \to \bk: \| \phi \| \leq 1 \}$ is determined by requiring that a net $\phi_\lambda$ converges to $\phi$ if and only if for all $x\in \bk$ and $\rho \in \bk_*$, we have $\rho(\phi_\lambda(x)) \to \rho(\phi(x))$. Furthermore, $\mathfrak{C}(\bk)$ is compact in the relative BW-topology (see \cite{arv-subalgs1} for details). 

\begin{theorem} \label{twoz} 
Let $\fk$ be a separable Hilbert space. Let $\omega$ be a non-zero $q$-weight map over $\fk$ with finite range rank, and let $\{\Pi^\# _t\}_{t>0}$ be the corresponding generalized boundary representation. Let us consider $(\Pi_t^\# \circ \Lambda)_{t>0}$ as a net with respect to the directed $J=(0,\infty)$ directed by $t \precsim s$ if and only if $t \geq s$. If 
$\Pi_0 ^\# \equiv 0$, then the net $(\Pi_t^\# \circ \Lambda)_{t \in J}$ in $\mathfrak{C}(\bk)$ has at least one cluster point $L$ in the relative BW-topology, and every such cluster point is a boundary expectation for $\omega$.
\end{theorem}

\begin{proof}  Since $\Pi^\#_t \circ \Lambda$ is a completely positive contraction of $\bk$ for each $t>0$, we have that the net is inside the set $\frak{C}(\bk)$ which is compact in the relative BW-topology. Therefore, there exists a subnet $\Pi^\#_{t_\mu} \circ \Lambda$ which converges to a map $L$ in the BW-topology.  We now show
that $L$ has properties $(i)$ through $(iv)$ of Definition~\ref{boundary-expectation-definition}.
Property $(i)$
follows trivially from the fact that the space of completely positive maps in $\frak{C}(\bk)$
is closed in the BW-topology.

Let $s>0$.  We claim that
\begin{equation}\label{zero}
\lim_{t \rightarrow 0^+} \|(I + \breve{\omega}\vert_t \circ \Lambda)^{-1}\circ \breve{\omega}\vert_s\| =0.
\end{equation}  For this, we first note that if $t<s$ and $A \in \mathfrak{B}(\mathfrak{H})$, then
\begin{equation} \label{omegas} 
\breve{\omega}\vert_s(A) = \breve{\omega}\big(E(s,\infty) A E(s,\infty)\big) = \breve{\omega}\vert_t \big(E(s,\infty)AE(s,\infty)\big).
\end{equation}  Since $\Pi_0^\#$ is the zero map and 
$E(s,\infty) \in U_s \mathfrak{B}(\mathfrak{H}) U_s^*$, it follows that $\Pi^\# _t(E(s,\infty)) \rightarrow 0$ $\sigma$-strongly as $t \rightarrow 0^+$.  Now note that $\Pi_t^\# = (I + \bromega\vert_t\Lambda)^{-1}\bromega\vert_t$. Since the range of $\bromega\vert|t$ is invariant under $I + \bromega\vert_t\Lambda$, the same holds for the inverse. Hence, the range of $\Pi_t^\#$ is contained in the range of $\bromega\vert|t$ and the latter is contained in the range of $\bromega$ which is finite dimensional. Therefore, $\| \Pi^\# _t(E(s,\infty)) \| \rightarrow 0$  as $t \rightarrow 0^+$. 

The maps $\phi_t(A):= \Pi^\# _t(E(s,\infty)A E(s,\infty))$ are completely positive for all $t>0$
and thus satisfy
$$
\|\phi_t\|=\|\phi_t(I)\|=\|\Pi^\# _t(E(s,\infty) I E(s,\infty)) \|=\|\Pi^\# _t(E(s,\infty))\|,
$$ 
so for every contraction $A \in \mathfrak{B}(\mathfrak{H})$, we have
\begin{equation} \label{pies1} 
\left\|\Pi^\#_t \big(E(s,\infty) A E(s,\infty) \big) \right\| \leq \|\Pi^\#_t(E(s,\infty))\|.
\end{equation}
Putting together equations \eqref{omegas} and \eqref{pies1},
we observe that if $t<s$, then
\begin{align*} \|(I + \breve{\omega}\vert_t \circ \Lambda)^{-1} \circ \breve{\omega}\vert_s \| & =
\sup_{\|A\| \leq 1, A \in \mathfrak{B}(\mathfrak{H})} \|(I + \breve{\omega}\vert_t \circ \Lambda)^{-1} \breve{\omega}\vert_s(A)\|
\\ & = \sup_{\|A\| \leq 1, A \in \mathfrak{B}(\mathfrak{H})} \|(I + \breve{\omega}\vert_t \circ \Lambda)^{-1} \breve{\omega}\vert_t(E(s,\infty)AE(s,\infty))\|
\\ & = \sup_{\|A\| \leq 1, A \in \mathfrak{B}(\mathfrak{H})} \|\Pi ^\#_t(E(s,\infty)AE(s,\infty))\| \\ &= \|\Pi^\# _t(E(s,\infty))\| \rightarrow 0 \textrm{ as } t \rightarrow 0,
\end{align*}
establishing equation \eqref{zero}.  Thus, for all $x\in \bk$ and $\rho \in \bk_*$,
\begin{align}
\rho(L \circ \breve{\omega}\vert_s(x) )\nonumber 
& = \lim_{\mu} 
\rho \Big( \big[ (I + \breve{\omega}\vert_{t_\mu} \Lambda)^{-1} \breve{\omega}\vert_{t_\mu}  \Lambda\big] (\breve{\omega}\vert_s(x)) \Big)
= \lim_{\mu} \rho \Big( \big[I - (I + \breve{\omega}\vert_{t_\mu} \Lambda)^{-1}\big] (\bromega\vert_s(x))\Big)
\\ \label{dd} & = \rho(\bromega\vert_s(x)) - \lim_{\mu} \rho\Big((I + \bromega\vert_{t_\mu}  \Lambda)^{-1} (\bromega\vert_s(x)) \Big) = \rho(\bromega\vert_s(x)).
\end{align}
Therefore, $L$ fixes the range of $\breve{\omega}\vert_s$ for every $s>0$.  Let $R \in \mathfrak{A}(\mathfrak{H})$,
so $R= (I - \Lambda(I))^{\frac{1}{2}} B (I - \Lambda(I))^{\frac{1}{2}}$ for some $B \in \mathfrak{B}(\mathfrak{H})$.  
Let $\rho \in \bk_*$.  By Proposition \ref{get-omega}, we have 
$\lim_{s \rightarrow 0+} \omega(\rho)\big(E(s,\infty) R E(s,\infty)\big)=  \omega(\rho)(R)$,
hence $\lim_{s \rightarrow 0+} \rho(\breve{\omega}\vert_s(R)) = \rho(\breve{\omega}(R)).$
Since the range of $\bromega\vert_t$ is contained in the range of $\bromega$ for every $t>0$ and the latter is finite dimensional we actually have that 
\begin{equation} \label{ee} \lim_{s \rightarrow 0+} \| \bromega\vert_s(R) - \breve{\omega}(R)\| =0. 
\end{equation}
Since $L \in \mathfrak{C}(\bk)$,  it is norm-continuous, hence equations \eqref{dd} and \eqref{ee} imply
$$L(\breve{\omega}(R)) = \lim_{s \rightarrow 0^+}L(\breve{\omega}\vert_s(R))=
\lim_{s \rightarrow 0^+} \breve{\omega}\vert_s(R)=\breve{\omega}(R)$$
for all $R \in \mathfrak{B}(\mathfrak{H})$, thereby proving $(ii)$.

Since $L$ fixes the range of $\breve{\omega}$ we have $\range(L) \supseteq \range(\breve{\omega})$.  
From the algebraic fact that $\breve{\omega}\vert_t \circ \Lambda$ commutes with $(I + \breve{\omega}\vert_t \circ \Lambda)^{-1}$ for all $t>0$,
we have $x \in \bk$,
$$
L(x)= \ultralim_{\mu} (I + \breve{\omega}\vert_{t_\mu} \Lambda)^{-1} \breve{\omega}\vert_{t_\mu}  \Lambda (x)  = 
\ultralim_{\mu}  \bromega\vert_{t_\mu} \big[\Lambda (I + \bromega\vert_{t_\mu} \Lambda)^{-1}(x)\big]. 
$$ 
Therefore, every element of the range of $L$ is the $\sigma$-weak limit of elements in the 
range of $\bromega$, which is a finite dimensional subspace of
$\bk$ and is thus $\sigma$-weakly closed.  Therefore, $\range(L) \subseteq \range(\breve{\omega})$.  We conclude $\range(L) = \range(\breve{\omega})$,
proving $(iii)$, whereby
property $(ii)$ implies that $L$ fixes its range, hence $L^2=L$.  Note $\|L\| \leq 1$ since $L \in \mathfrak{C}(\bk)$, so since $L$ is idempotent we have $\|L\|=0$ or $\|L\|=1$.  By assumption, $\bromega$ is not the zero map, hence 
$\{0\} \subsetneq \range(\breve{\omega})=\range(L)$, so $\|L\|=1$,
proving $(iv)$.
\end{proof}

Let $L$ be a boundary expectation for a $q$-weight $\omega$ over $\fk$, and 
let $\rl=L(\bk)$. Since $L:\bk\to \bk$ is a completely positive and contractive idempotent, by the work of Choi-Effros~\cite{choi-effros-injectivity} (see also section 6.1 of \cite{effros-ruan-book}), we recall that
\begin{equation}\label{eq:choi-effros}
L( L(T) S) = L( L(T) L(S)) = L( T L(S)), \qquad \forall T,S \in \bk.
\end{equation}
Furthermore, we have that $\rl$ is a unital C*-algebra under its norm and involution as a subspace of $\bk$ but with multiplication $\cem$ given by
\begin{equation}\label{multiplication}
x \cem y = L(x y) ,  \qquad \forall x,y \in \rl.
\end{equation}
If the range of $L$ is $\sigma$-weakly closed (for example if it is finite-dimensional as in the previous theorem), then it is the dual of a Banach space hence  $\rl$ is  a W$^*$-algebra. We also note that its unit is $L(I)$ (thus we remark that while $\rl$ is unital, in general it does not share the unit with $\bk$). 

Finally, we note that although a boundary expectation $L$ need \emph{not} be a conditional expectation in the traditional sense, it satisfies the following property by a direct application of \eqref{eq:choi-effros}:
$$
L(xTy) = x \cem L(T) \cem y, \qquad \forall T \in \bk, \forall x,y \in \rl.
$$

The following lemma will be useful for the study of $q$-corners between $q$-weights with range rank one.

\begin{lemma}\label{L-rank-one}
Let $\mathfrak{H}_1$ and $\mathfrak{H}_2$ be orthogonal complementary subspaces of $\cc^n$. Suppose that $L:M_n(\cc) \to M_n(\cc)$ is a completely positive contractive idempotent map, 
which has block form
$$
L \begin{pmatrix} A & B \\ C & D \end{pmatrix} = 
\begin{pmatrix} L_{11}(A) & L_{12}(B) \\ L_{21}(C) & L_{22}(D) \end{pmatrix}
$$
where $L_{ij}: \mathfrak{B}(\mathfrak{H}_j, \mathfrak{H}_i) \to \mathfrak{B}(\mathfrak{H}_j, \mathfrak{H}_i)$ for $i,j=1,2$. If $\dim( \range L_{11}) = \dim (\range L_{22})=1$, then either $L_{12}\equiv 0$ or $\dim (\range L_{12})=1$. 
\end{lemma}
\begin{proof}
Suppose $L_{12} \not\equiv 0$, and let $B_0 \in \range(L_{12})$ be an element with $\| B_0 \| =1$. Notice that $B_{0} = L_{12}(B_0)$. 

We will show that the element of $\rl$ given by
$$
u= \begin{pmatrix} 0 & B_0 \\ 0 & 0 \end{pmatrix}  = \begin{pmatrix} 0 & L_{12}(B_0) \\ 0 & 0 \end{pmatrix} =  L\begin{pmatrix} 0 & B_0 \\ 0 & 0 \end{pmatrix}
$$
is a partial isometry such that $u^*\cem u+u\cem u^*=L(I)=I_{\rl}$ and $u\cem u^*$ is a minimal projection in
the W$^*$-algebra $\rl$. 
Notice that $\|u\| =1$ since the norm in $\rl$ is the same as the norm as a subspace of $M_n(\cc)$.  We note that
$$
u\cem u^* = L \left( \begin{pmatrix} 0 & B_0 \\ 0 & 0 \end{pmatrix} \cdot \begin{pmatrix} 0 & 0 \\ B_0^* & 0 \end{pmatrix} \right) = L \begin{pmatrix} B_0B_0^* & 0 \\ 0 & 0 \end{pmatrix} =
\begin{pmatrix} L_{11}(B_0B_0^*) & 0 \\ 0 & 0 \end{pmatrix} 
$$ 
$$
u^*\cem u = L \left( \begin{pmatrix} 0 & 0 \\ B_0^* & 0 \end{pmatrix} \cdot \begin{pmatrix} 0 & B_0 \\ 0 & 0 \end{pmatrix} \right) = L \begin{pmatrix} 0 & 0 \\ 0 & B_0^*B_0 \end{pmatrix} =
\begin{pmatrix} 0 & 0 \\ 0 &  L_{22}(B_0^*B_0) \end{pmatrix} 
$$ 
Since $\|u\| =1$, we have that $\| L_{11}(B_0B_0^*) \| =\| L_{22}(B_0^*B_0)  \|= 1$. Let $T_1 =  L_{11}(B_0B_0^*)$ and $T_2 =  L_{22}(B_0^*B_0)$, and note that these are positive operators since $L$ is completely positive. Notice that since 
$\dim(\range(L_{11}))=\dim(\range(L_{22}))=1$, and for each $i=1,2$, $T_i \in \range(L_{ii})$, hence it follows that there exists a linear functional $\rho_i$ on $\mathfrak{B}(\mathfrak{H}_i)$ such that $L_{ii}(A) = \rho_i(A) T_i$ for all $A \in \mathfrak{B}(\mathfrak{H}_i)$. Since $L$ is a completely positive contractive idempotent, so is $L_{ii}$, hence $\rho_i$ is a state. Therefore, we have that
$$
(u \cem u^*) \cem (u \cem u^*) = L\left( \begin{pmatrix}  T_1 & 0 \\ 0 &  0 \end{pmatrix} 
\begin{pmatrix}  T_1 & 0 \\ 0 &  0 \end{pmatrix} \right) = L \begin{pmatrix}  T_1^2 & 0 \\ 0 &  0 \end{pmatrix} =
\begin{pmatrix}  \rho_1(T_1^2)T_1 & 0 \\ 0 &  0 \end{pmatrix} = \rho_1(T_1^2) (u \cem u^*)
$$
On the other hand, we have that $\| u \cem u^* \|=1$, hence by taking norms on both sides
and using the C$^*$-norm identity, we obtain that $\rho_1(T_1^2)=1$ and $u \cem u^*$ is a projection. Hence $u$ is a partial isometry in $\rl$.

Furthermore, observe that
\begin{align*}
I_{\rl} & = L(I) = \begin{pmatrix} L_{11}(I_{\mathfrak{H}_1}) & L_{12}(0) \\ L_{21}(0) & L_{22}(I_{\mathfrak{H}_2}) \end{pmatrix} =
 \begin{pmatrix} \rho_1(I_{\mathfrak{H}_1}) T_1 & 0 \\ 0 & \rho_2(I_{\mathfrak{H}_2}) T_2 \end{pmatrix} = \begin{pmatrix}  T_1 & 0 \\ 0 &  T_2 \end{pmatrix} \\
&  = u^*\cem u+u\cem u^*.
\end{align*}

We prove that $u\cem u^*$ is minimal. Indeed, if $q$ is a projection in $\rl$ such that $q \leq u\cem u^*$, then $(u\cem u^*) \cem q \cem (u\cem u^*) = q$. Now note that
\begin{align*}
(u\cem u^*) \cem q \cem (u\cem u^*) & = L ((u\cem u^*) q (u\cem u^*)) =
L \left( \begin{pmatrix}  T_1 & 0 \\ 0 &  0 \end{pmatrix} \begin{pmatrix}  q_{11} & q_{12} \\ q_{21} &  q_{22} \end{pmatrix} \begin{pmatrix}  T_1 & 0 \\ 0 &  0 \end{pmatrix} \right) \\
& = L \begin{pmatrix}  T_1q_{11}T_1  & 0 \\ 0 &  0 \end{pmatrix} = \rho_1(T_1q_{11}T_1)\, u\cem u^*.
\end{align*}
Since $q\leq u\cem u^*$ is a projection, this implies that either $q=0$ or $q=u \cem u^*$.

Since $u$ is a partial isometry such that $u^*\cem u+u\cem u^*= 1_{\rl}$ and $u\cem u^*$ is a minimal projection, we obtain a system of matrix units $(e_{ij}: i,j=1,2)$ for $\rl$ by assigning $e_{12}=u$ and following the relations
$$
e_{ij}\cem e_{k\ell} = \delta_{jk} e_{i\ell}, \qquad e_{ij} = e_{ji}^*, \qquad e_{11}+e_{22}=1_{\rl}
$$
for all $i,j=1,2$. Therefore, we have that for every $x \in \rl$, if we denote by
$x_{ij}=e_{ii}xe_{jj}$, we have that $x = \sum_{i,j=1}^2 x_{ij}.$
Note however that $x_{ij}e_{ji} \in e_{ii}\rl e_{ii}$. Since $e_{11}$ is minimal, and $e_{22}$ is Murray-von Neumann equivalent to $e_{11}$, we have that $e_{22}$ is also minimal. Hence $e_{ii}\rl e_{ii} = \cc e_{ii}$ for $i=1,2$. Therefore, there exists $\lambda_{ij}\in\cc$ such that 
$$
x_{ij}e_{ji} = \lambda_{ij} e_{ii} \quad\implies\quad  x = x_{ij} e_{ji} e_{ij} =  \lambda_{ij} e_{ii} e_{ij} = \lambda_{ij} e_{ij}
$$
Therefore, $x=\sum_{ij} \lambda_{ij} e_{ij}$ for some $\lambda_{ij} \in \cc$ for $i=1,2$. In particular, for every $X \in M_n(\cc)$, $L_{12}(X) = \lambda_{12} e_{12}$ for some $\lambda_{12}\in\cc$, hence $\range(L_{12}) = \Span(e_{12})$ and it is one-dimensional.
\end{proof}

We will find in Theorem~\ref{1or4} that in the special case when $\fk=\cc^2$, Theorem~\ref{twoz} and Lemma~\ref{L-rank-one} 
can also be used to narrow down the possible range ranks of $q$-weights that are $q$-pure.

\section{CP-semigroups and $q$-purity}\label{sec-q-purity}

\begin{definition}\label{q-pure}
We will say that a CP-flow $\alpha$, or alternatively its $q$-weight map, is \emph{$q$-pure} if its set of flow subordinates is totally ordered by subordination. 
\end{definition}

We remark that a unital CP-flow is $q$-pure if and only if its minimal flow dilation is also $q$-pure by Theorem~3.5 and Lemma~4.50 of \cite{powers-CPflows}. We also note that a $q$-pure E$_0$-semigroup must have index 0 or 1. However, since automorphism groups are not CP-flows, therefore a $q$-pure E$_0$-semigroup cannot be of type I$_0$.

\begin{prop}
Let $\alpha$ be a unital CP-flow over $\fk$ which is $q$-pure, and let $\mathfrak{S}$ be its set of flow subordinates. Then $\mathfrak{S}$ is a complete totally ordered set which is order isomorphic (hence homeomorphic in the order topology) to a compact subset of the unit interval. Furthermore, the order topology on $\mathfrak{S}$ can also be described by the uniform convergence in the strong operator topology on compact sets.
\end{prop}
\begin{proof}
Let $\alpha^d$ be the minimal flow dilation of $\alpha$ on the Hilbert space $\fh$. Then the 
set of subordinates of $\alpha$ is order isomorphic to the set $\mathfrak{Q}$ of positive
 contractive flow cocycles of $\alpha^d$. Notice that it is clear that if $T^j=(T^j_t)$ 
is an  increasing net of positive contractive flow cocycles of $\alpha^d$, then for 
each $t>0$, the net $T^j_t$ converges in the strong topology to an operator $T_t$, and 
$T=(T_t)$ must be a local positive contractive cocycle. The same argument applies for 
decreasing nets, hence $\mathfrak{Q}$ is a complete totally ordered set. 
	Now let $(e_n: n\in\nn)$ be an orthonormal basis for $\fh$, and define the map 
$\phi: \mathfrak{Q} \to [0,1]$ by
$$
\phi(T) = \sum_{n\in\nn} 2^{-n} \int_0^1 \<T_t e_n, e_n \> dt
$$ 
It is clear that $\phi$ is a well-defined injective order-preserving map. Since $\mathfrak{Q}$ is complete, it follows that the range of $\phi$ must be compact. See also Theorem 12, page 242 of \cite{birkhoff}.

Finally, note that the map $\phi$ is also continuous when $\mathfrak{Q}$ is endowed with the convergence in the strong operator topology on compact sets. In this topology, which coincides with uniform $\sigma$-weak convergence on compact sets, the set $\mathfrak{Q}$ is also compact, hence we obtain the desired homeomorphism.
\end{proof}

We remark that in general we do not know whether $\mathfrak{S}$ is homeomorphic to an \emph{interval}. For that to hold, it suffices to show that for every two subordinates $\beta$ and $\gamma$ such that $\beta_t \leq \gamma_t$ for all $t$, there exists another subordinate $\sigma$ such that $\beta_t \leq \sigma_t \leq \gamma_t$ for all $t>0$ (see Theorem 14 of page 243 of \cite{birkhoff}).

Of course the situation with E-subordinates of an E$_0$-semigroup is somewhat simpler, in that they form a complete lattice (see Theorem~4.9 of \cite{powers-typeII}; see also Theorem 4 of \cite{liebscher}). We do not know if the analogous result holds for CP-flow subordinates of a unital CP-flow.

\subsection*{Concrete description of $q$-purity in the range rank one case}

Let $\fk$ be a separable Hilbert space and define $\fh=\fk \otimes L^2(0,\infty)$. If $\omega: \fb(\fk)_* \to \fa(\fh)_*$ is a $q$-weight map of range rank one (see Definition~\ref{bromega-finite-rank}), then there exists a positive
boundary weight $\mu \in \fa(\fh)_*$ and a positive $T \in \fb(\fk)$ such that
\begin{equation} \label{rankone} 
\omega(\rho )(A) = \rho (T)\mu (A)
\end{equation} 
for all $\rho \in \frak B (\frak K ) _*$ and $A
\in \frak A (\frak H )$. However not every boundary weight map of the form \eqref{rankone} is a $q$-weight.

\begin{theorem} \label{rank-one-q-weight} Suppose $T$ is a positive operator in $\frak B
(\frak K )$ of norm one and $\mu \in \frak A (\frak H )_*$ is positive
and $\mu (I - \Lambda (T)) \leq 1$.  Then the mapping $\omega (\rho )$ of
$\frak B (\frak K )_*$ into $\frak A (\frak H )_*$ given by
$$
\omega (\rho )(A) = \rho (T)\mu (A)
$$
$\rho \in \frak B (\frak K )_*$ and for $A \in \frak A (\frak H )$ is a
$q$-weight.  The $q$-weight map $\omega$ is unital if and only if $T = I$
and $\mu (I - \Lambda (T)) = 1$.  Conversely every $q$-weight of range rank
one is of the above form, and its generalized boundary representation $\pi^\#$ is given by
$$
\pi_t^\# (A) = \frac{\mu{\vert_t} (A)}{1 + \mu{\vert_t} (\Lambda (T))} \; T, \qquad\qquad
\forall t>0, \forall A \in \bh.
$$
\end{theorem}
\begin{proof}

Suppose $\omega$ is a boundary weight map of the form \eqref{rankone}, where
$\mu \in \fa(\fh)_*$ is positive and $T \in \fb(\fk)$ is positive with norm one.  
 We observe that
$$
\hat \Lambda\omega{\vert_t} (\rho )(A) = \rho (T)\cdot\mu{\vert_t}
(\Lambda (A))
$$
for all $t > 0$, $A \in \frak B (\frak K )$, and $\rho \in \frak B (\frak K )_*$.  Then we
see that the inverse of the
mapping
$$
A \rightarrow A + \mu{\vert_t} (\Lambda (A))\cdot T
$$
is given by
$$
A \rightarrow A - \big(1 + \mu{\vert_t} (\Lambda (T))\big)^{-1}\mu{\vert_t}
(\Lambda (A))\cdot T.
$$
Then the generalized boundary representation is
\begin{align*}
\pi_t^\# (A) &= \mu{\vert_t} (A)\cdot T - \big(1 + \mu{\vert_t}
(\Lambda (T))\big)^{-1}\mu{\vert_t} (T)\mu{\vert_t} (A)\cdot T
\\
&= \big(1 + \mu{\vert_t} (\Lambda (T))\big)^{-1} \mu{\vert_t} (A)\cdot T.
\end{align*}
We see that the generalized boundary representation is completely
positive for all $t > 0$.  Now we need to check that the norm $\pi_t^\# (I)$
is
not greater than one.  We have
$$
\pi_t^\# (I) = \big(1 + \mu{\vert_t} (\Lambda (T))\big)^{-1} \mu{\vert_t}
(I)\cdot T
$$
and since $T$ is of norm one we have
$$
\Vert\pi_t^\# (I)\Vert = \big(1 + \mu{\vert_t} (\Lambda (T))\big)^{-1}
\mu{\vert_t} (I)
$$
for all $t > 0$.  In order that this norm not exceed one we must have
$$
\mu{\vert_t} (I - \Lambda (T)) \leq 1
$$
and since the above function of $t$ is non increasing the above
inequality holds for all $t > 0$ if and only if
$$
\mu (I - \Lambda (T)) \leq 1.
$$
Since $\Lambda (T) \leq \Lambda (I) = \Lambda$ it follows that this
inequality implies $\mu (I - \Lambda ) \leq 1$ so even if we did not assume
$\mu \in \frak A ( \frak H )_*$ the conditions that $\omega$ be a
contractive $q$-positive boundary weight map would force this on us.
\end{proof}

\begin{theorem}
Let $\fk$ be a separable Hilbert space, define $\fh=\fk \otimes L^2(0,\infty)$ and 
let $\omega:\bk_* \to \ah_*$ be a non-zero $q$-weight map of range rank one. Then 
for all $\lambda \in [0,1]$, the map $\lambda \omega$ is a $q$-weight map subordinate to 
$\omega$. Furthermore, $\omega$ is $q$-pure if and only if every $q$-subordinate of $\omega$
 has the form  $\lambda \omega$ for some $\lambda \in [0,1]$. 
\end{theorem}
\begin{proof}
Let $\lambda \in [0,1]$ be given. Let us show that if $\omega$ is a $q$-weight of range rank one then $\lambda\omega$ is a $q$-subordinate of $\omega$.  Let $\mu \in \ah_*$ be positive and $T\in \bk$ positive with norm one be given by the previous theorem so that $\omega$ satisfies \eqref{rankone}. Then by the previous theorem $\lambda \omega$ is a $q$-weight map of range rank one. Let $\pi^\#$ and $\phi^\#$ be the generalized boundary representations of $\omega$ and
$\lambda\omega$, respectively. Then we have that for all $t > 0$ and $A \in \frak B (\frak H )$,
\begin{align*}
\pi_t^\# (A) & = (1 + \mu{\vert_t} (\Lambda (T)))^{-1} \mu{\vert_t}
(A)\cdot T,\\
\phi_t^\# (A) & = (1 + \lambda\mu{\vert_t} (\Lambda (T)))^{-1}
\lambda\mu {\vert_t} (A)\cdot T.
\end{align*}
Since for $b \geq 0$ the function $h(x) = x/(1+bx)$ is an increasing function of $x$ it 
follows that $\pi_t^\# \geq \phi_t^\#$ for all $t > 0$.  Thus we have that $\lambda \omega \leq_q \omega$.

Now suppose that $\omega$ is non-zero and $q$-pure and $\eta$ is a non-zero $q$-subordinate of $\omega$. Let $\bromega$ and $\breta$ be the dualized $q$-weight maps corresponding to $\eta$ and $\omega$, respectively.  

Observe that if $\omega_1 \leq_q \omega_2$ are $q$-weights, then 
\begin{equation}\label{omega-at-1}
\bromega_1(I-\Lambda) \leq \bromega_2(I-\Lambda).
\end{equation}
Thus we have that that $0 \leq \breta(I-\Lambda) \leq \bromega(I - \Lambda)$, and observe that $\bromega(I-\Lambda) \neq 0$ and $\breta(I-\Lambda) \neq 0$ since $\omega$ and $\eta$ are non-zero. Let
$$
\lambda = \frac{ \| \breta(I-\Lambda) \| }{ \| \bromega(I-\Lambda) \|}
$$
and notice that $0<\lambda\leq 1$. If $\lambda=1$, then by Proposition~\ref{q-subordinate-at-1}, we have that $\eta=\omega$, as desired. So let us consider the case $0<\lambda <1$. Let $\epsilon>0$ be small enough so that $0 <\lambda -\epsilon < \lambda + \epsilon<1$ and
\begin{equation}\label{sandwich-norm}
(\lambda - \epsilon) \;\|   \bromega(I-\Lambda) \|  < \| \breta(I-\Lambda) \| < (\lambda + \epsilon)\; \|  \bromega(I-\Lambda) \|.
\end{equation}
 Now observe that $(\lambda -\epsilon)\omega\leq_q(\lambda +\epsilon)\omega$. Furthermore we cannot have $\eta \leq_q (\lambda -\epsilon) \omega$ or $\eta \geq_q (\lambda +\epsilon) \omega$, because by 
\eqref{omega-at-1} either inequality implies
$$
\breta(I-\Lambda) \leq (\lambda -\epsilon) \bromega(I-\Lambda), \quad\text{or}\quad
\breta(I-\Lambda) \geq (\lambda +\epsilon) \bromega(I-\Lambda),
$$
in which case we have that
$$
\| \breta(I-\Lambda)\|  \leq (\lambda -\epsilon) \| \bromega(I-\Lambda) \|, \quad\text{or}\quad
\| \breta(I-\Lambda) \| \geq (\lambda +\epsilon) \|\bromega(I-\Lambda) \|,
$$
contradicting \eqref{sandwich-norm}. However, by assumption $\omega$ is $q$-pure, hence the set of $q$-subordinates of $\omega$ is totally ordered. Therefore we must have 
\begin{equation}\label{sandwich}
(\lambda -\epsilon)\omega\leq_q \eta \leq_q(\lambda +\epsilon)\omega.
\end{equation}
Now by Proposition~\ref{q-subordinate-at-1}, this implies that for all $\epsilon>0$ small enough, positive $\rho \in \bk_*$ and positive $T \in \ah$, 
$$
(\lambda -\epsilon)\omega(\rho)(T) \leq \eta(\rho)(T) \leq(\lambda +\epsilon)\omega(\rho)(T).
$$
Thus we have that for all positive $\rho \in \bk_*$ and positive $T \in \ah$, $\eta(\rho)(T) =\lambda \omega(\rho)(T)$ hence $\eta = \lambda \omega$.
\end{proof}

We note that when $\omega: \bk_* \to \ah_*$ is a $q$-weight map but with range rank not equal to 1, then the above description of $q$-purity does not hold, for instance $\lambda \omega$ can fail to be a subordinate of $\omega$ for all $\lambda \in (0,1)$ (see Section~4 of \cite{powers-holyoke} or the beginning of Section~4 of \cite{jankowski1}). Furthermore, if $\nu$ is a type II Powers weight such $A \mapsto \nu((I-\Lambda)^{1/2}A(I-\Lambda)^{1/2})$ is a pure normal state and $\phi: M_n(\cc) \to M_n(\cc)$ is such that $\phi \circ (I \otimes \nu)$ defines a $q$-weight map, then $\omega$ is $q$-pure if and only if every non-zero 
$q$-subordinate $\eta$ of $\omega$ has the form 
$$
\breve{\eta} = \phi(I+t\phi)^{-1}\circ (I \otimes \nu)
$$
for some $t\geq 0$ (see Definition 4.2 and Lemma 4.3 of \cite{jankowski1}).

\section{The range rank one case}\label{sec-rank-one}

In this section we will study E$_0$-semigroups arising from $q$-weights with range rank one and their subordinates. Before we begin, however, we need to analyze the subordination structure of boundary weights.

\subsection{Boundary weights and their subordinates}
Let $\fk$ be a separable Hilbert space and let $\fh= \fk \otimes L^2(0, \infty)$, the space of $\fk$-valued Lebesgue
measurable functions defined on $(0, \infty)$.  We have found the following description of the boundary weights acting on $\frak A (\frak H )$ to be useful.  Let $q:(0,\infty)\to \rr)$ be given by $q(x)=1-e^{-x}$, and let
$\mathfrak{H}_q = \fk\otimes L^2(0,\infty; q(x)dx)$  be the linear space of Lebesgue measurable $\frak K$-valued functions which are square integrable with respect to the measure $(1-e^{-x})dx$. Notice that the operator $M_q$ of multiplication by $q(x)$ is bounded on $\fh_q$.
We define a sesquilinear form $\< \cdot, \cdot \>$ on $\fh_q \times M_q\fh_q$ as follows:
given $f \in \fh_q$ and $g \in M_q \fh_q$, then
$$
\< f, g \> = \int_0^\infty \overline{f(x)}g(x) dx.
$$
Now observe that if $A \in \fa(\fh)$, and $g \in \fh_q$, then we have that $Ag$ is well-defined in a natural way and furthermore $Ag \in M_q \fh_q$.  Now if $\omega \in
\frak A (\frak H )_*$ we have the
functional $\rho (A) = \omega ((I - \Lambda )^{1/2} A(I - \Lambda )^{1/2}
)$ for $A \in \frak B (\frak H )$ is normal so
there are two orthonormal sets of vectors $\{ f_i,g_i,i = 1,2,\cdots \}$ and a sequence of positive
real numbers $(\lambda_i)_{i=1}^r$ ($r=\infty$ is allowed) such that
$$
\sum_{i=1}^r \lambda_i < \infty\qquad \text{and}\qquad \rho (A) = \sum_{i=1}^r \lambda_i(f_i,Ag
_i)
$$
for $A \in \frak B (\frak H )$.  We can then think of the functions $h_i =
(I - \Lambda )^{-{\tfrac{1}{2}}}f_i$ and $k_i = (I - \Lambda )^{-{\tfrac{1}{2}}}g_i$ as
two sets of orthonormal vectors in $\frak H _q$ (with respect to the inner product of $\fh_q$ and not the sesquilinear form) and in terms of these
vectors we have
$$
\omega (A) = \sum_{i=1}^r \lambda_i \<h_i,Ak_i\>
$$
for $A \in \frak A (\frak K )$.  Note if $\omega$ is positive then we can
arrange it so $f_i = g_i$ or $h_i = k_i$.

We now define a useful ordering on positive boundary weights.

\begin{definition}  Suppose $\omega$ and $\eta$ are positive
boundary weights on
$\frak A (\frak H )$.  We say $\omega$ $q$-dominates $\eta$ or $\eta$ is a
$q$-subordinate of $\omega$, written $\omega \geq_q \eta$ if for all $t >
0$ we have
$$
\frac {\omega\vert_t(A)} {1+\omega\vert_t(\Lambda )}  \geq \frac {\eta\vert
_t(A)} {1+\eta\vert_t(\Lambda )}
$$
for all positive $A \in \frak B (\frak H )$.
\end{definition}

Let $\omega \in \fa(\fh)_*$ be a positive boundary weight.  If $A \in \fb(\fh)$ 
is positive and $\omega \vert_t(A) \rightarrow \infty$
as $t \to 0+$, we will write $\omega(A)= \infty$.  This notation will be
useful when dealing with unbounded boundary weights which are unbounded linear
functionals (i.e. $\omega(I)= \infty$)
or weights for which $\omega \vert_t (\Lambda(E)) \rightarrow \infty$ as $t \to 0+$
for a projection $E$.

One checks that if $\omega \geq_q \eta$ and $\eta \geq_q \mu$ then $\omega
\geq_q \mu$.  Also we see
that the $q$-ordering is stronger than the normal ordering in that if
$\omega \geq_q \eta$ then $\omega \geq \eta$ (i.e. $\omega (A) \geq \eta (A)$
for all positive $A \in \frak A (\frak H ))$.
Note that if $\omega$ is a positive boundary weight then $\omega \geq_q \lambda
\omega$ for all $\lambda \in [0,1]$.

\begin{definition}[\cite{powers-holyoke}]\label{old-q-pure} We say a positive boundary weight $\omega$ is \emph{$q$-pure} if $\omega
\geq_q \eta \geq_q 0$
if and only if $\eta = \lambda\omega$ with $\lambda \in [0,1]$.  
\end{definition}

We should remark that in the special case when $\fk=\cc$, a $q$-weight map $\omega$ over $\fk$
 is $q$-pure (Definition~\ref{q-pure}) if and only if $\omega(1)$ is $q$-pure as a boundary weight.

Note the
$q$-ordering is quite different from the normal ordering.  For example the
sum $\omega + \eta$ of two positive weights can be $q$-pure.  In Theorem
3.9 of \cite{powers-holyoke} this order relation was characterized in the
case where $\frak K$ is one dimensional so $\frak H = L^2(0,\infty )$.  It
turns out the same characterization applies to when $\frak K$ is any
separable Hilbert space.

\begin{theorem} \label{sub-decomposition} Suppose $\omega$ and $\eta$ are positive
boundary weights on $\frak A (\frak H )$.  Suppose $\rho \in \frak A (\frak
H )_*$ is positive and $\omega \geq \rho$ and $\rho (I) < \infty$ (so $\rho
\in \frak B (\frak H )_*)$ and $\eta = \lambda (1 +\rho (\Lambda
))^{-1}(\omega - \rho )$ with $0 \leq \lambda \leq 1$.  Then $\omega \geq_q
\eta$.  Conversely suppose $\omega \geq_q \eta$ and $\eta \neq 0$.  Then
there is a positive element $\rho \in \frak B (\frak H )_*$ and and a real
number $\lambda \in (0,1]$ so that $\eta = \lambda (1 + \rho (\Lambda
))^{-1}(\omega - \rho )$.  Furthermore, if $\omega (I) = \infty$ then
$\rho$ and $\lambda$ are unique.  It follows then that a positive boundary
weight on $\frak A (\frak H )$ is $q$-pure if and only if every rank
one positive functional $\rho \in \frak B (\frak H )_*$ subordinate to
$\omega$ (so if $\omega \geq \rho )$ is a multiple of $\omega$.   (Note in
the case when $\omega (\Lambda ) = \infty$ this means there are no bounded
positive functionals subordinate to $\rho$)
\end{theorem}
\begin{proof}  Assume the first two sentences in the statement of the theorem
are satisfied.  Then we have
\begin{align*}
\frac {\eta\vert_t} {1+\eta\vert_t(\Lambda )}  &= \frac {\omega\vert_t-\rho
\vert_t} {\lambda^{-1}+\lambda^{-1}\rho (\Lambda )+\omega\vert_t(\Lambda )-
\rho\vert_t(\Lambda )}
\\
&\leq \frac {\omega\vert_t-\rho\vert_t} {1+\rho (\Lambda
)+\omega\vert_t(\Lambda )-\rho\vert_t(\Lambda )} \leq \frac
{\omega\vert_t-\rho\vert_t} {1+\omega\vert _t(\Lambda )}
\\
&\leq \frac {\omega\vert_t} {1+\omega\vert_t(\Lambda )}
\end{align*}
Hence, we have
$$
\frac{\omega\vert_t(A)}{1+\omega\vert_t(\Lambda )}  \geq \frac {\eta\vert
_t(A)} {1+\eta\vert_t(\Lambda )}
$$
for all $t > 0$ so $\omega \geq_q \eta$.

Next assume $\omega$ and $\eta$ are as stated in the first sentence of the
statement of the theorem and $\omega \geq_q \eta \geq_q 0$.  Let
$$
h(t) = (1 + \eta{\vert_t} (\Lambda ))/(1 + \omega{\vert_t}(\Lambda ))
$$
for $t > 0$.  Since $\omega \geq_q \eta$ we have for $0 < t < s$

\begin{align*}
\frac {\eta\vert_t(\Lambda )-\eta\vert_s(\Lambda )}
{1+\eta\vert_t(\Lambda )}  &
= \frac {\eta\vert_t(E(t,s)\Lambda )} {1+\eta\vert_t(\Lambda )}
\\
&\leq \frac {\omega\vert_t(E(t,s)\Lambda )} {1+\omega\vert_t(\Lambda )}
= \frac {\omega\vert_t(\Lambda )-\omega\vert_s(\Lambda )} {1+\omega
\vert_t(\Lambda )}.
\end{align*}
Multiplying by the common denominator we have
$$
\eta{\vert_t} (\Lambda ) - \eta{\vert_s} (\Lambda ) -
\omega{\vert_t} (\Lambda )\eta{\vert_s} (\Lambda ) \leq
\omega{\vert_t}(\Lambda ) - \omega{\vert_s} (\Lambda ) -
\eta{\vert_t} (\Lambda )\omega{\vert_s} (\Lambda ).
$$
Rearranging \textsl{}this inequality gives
$$
\eta{\vert_t} (\Lambda ) + \omega{\vert_s} (\Lambda ) +
\eta{\vert_t} (\Lambda )\omega{\vert_s} (\Lambda ) \leq \omega{\vert_t}
(\Lambda ) + \eta{\vert_s} (\Lambda ) + \omega{\vert_t} (\Lambda
)\eta{\vert_s} (\Lambda ),
$$
and adding one gives
$$
(1 + \eta{\vert_t} (\Lambda ))(1 + \omega{\vert_s} (\Lambda ))
\leq (1 + \omega{\vert_t} (\Lambda ))(1 + \eta{\vert_s} (\Lambda
)),
$$
which yields $h(t) \leq h(s)$.  Hence, $h$ is non decreasing and since
$h(t) \leq 1$ for all $t > 0$ the $h(t)$ has a limit as $t \rightarrow 0+$.
We denote this limit by $\kappa$ so $h(t) \rightarrow \kappa$ as $t
\rightarrow 0+$.  Since $\omega \geq_q \eta$ we have $h(t) \omega{\vert_t}
\geq \eta{\vert_t}$  for all $t > 0$ so $\kappa \omega \geq \eta$.  Now if
$\kappa = 0$ then $\eta = 0$ and $\eta$ is trivially a subordinate of
$\omega$.  Since we are dealing with the case when $\eta \neq 0$ we have
$\kappa > 0$.  Let $\rho = \omega - \kappa^{-1}\eta $.  Since $\kappa\omega
\geq \eta$ we have $\rho$ is positive.  Note $\eta = \kappa (\omega - \rho
)$ and $\omega \geq \rho$.   Since $h$ is non decreasing we have $h(t)
\geq \kappa$ for all $t > 0$ and, hence,

$$
\frac {1+\kappa (\omega\vert_t(\Lambda )-\rho\vert_t(\Lambda ))} {1+\omega\vert
_t(\Lambda )} \geq \kappa
$$
for all $t > 0$.  Hence, $\rho{\vert_t} (\Lambda ) \leq \kappa^{-1} -
1$ for all $t > 0$ which yields $\rho (\Lambda ) \leq \kappa^{-1} - 1$.
Since $\rho \leq \omega$ we have $\rho (I - \Lambda ) \leq \omega (I -
\Lambda )$ and since $\omega$ is boundary weight we have $\omega (I -
\Lambda ) < \infty$ so $\rho (I) = \rho (I - \Lambda ) + \rho (\Lambda )
\leq$ $\kappa^{-1}+ \omega (I - \Lambda ) - 1$.  Hence, $\rho$ is bounded
so $\rho$ is a positive element of $\frak B (\frak H )_*$.  Since $\kappa
\leq (1 + \rho (\Lambda ))^{-1}$ we have $\kappa = \lambda (1 + \rho
(\Lambda ))^{-1}$ with $\lambda \in (0,1]$ and $\eta = \lambda (1 + \rho
(\Lambda ))^{-1}(\omega - \rho )$.  Hence, $\eta$ is of the form stated.

Finally, we show that if $\omega (I) = \infty$ then $\lambda$ and $\rho$
are unique.  Suppose then that $\lambda ,\lambda ^{\prime} \in (0,1]$ and
$\rho$ and $\rho ^{\prime}$ are positive elements of $\frak B (\frak H )_*$
so that $\omega \geq \rho$ and $\omega \geq \rho ^{\prime}$ and $\eta =
\lambda (1 + \rho (\Lambda ))^{-1}(\omega - \rho )$ and $\eta = \lambda
^{\prime}(1 + \rho ^{\prime}(\Lambda ))^{-1}(\omega - \rho ^{\prime}).  $
Then we have

$$
(\lambda (1 + \rho (\Lambda ))^{-1} - \lambda ^{\prime}(1 + \rho ^{\prime}(
\Lambda ))^{-1})\omega = \lambda (1 + \rho (\Lambda ))^{-1}\rho - \lambda
^{\prime}(1 + \rho ^{\prime}(\Lambda ))^{-1}\rho ^{\prime}
$$
Note if the functional on the left is non-zero it is unbounded and the
functional on the right is bounded so it follows both sides of the above
equality are zero.  Since $\omega \neq 0$ we have $\lambda (1 + \rho
(\Lambda ))^{-1} =$ $\lambda ^{\prime}(1 + \rho ^{\prime}(\Lambda ))^{-1}$
which when inserted in the right hand side of the above equality yields
$\rho = \rho ^{\prime}$.  Since $\rho (\Lambda ) = \rho ^{\prime}( \Lambda
)$ the fact that the right hand side is zero yields $\lambda = \lambda
^{\prime}$.  Hence, in the case at hand $\lambda$ and $\rho$ are unique.

We prove the last statement of the theorem.  From what we have proved we
see that $\omega$ is $q$-pure if and only if every positive $\rho \in \frak
B (\frak H )_*$ with $\omega \geq \rho$ is a multiple of $\omega $.  Note
every positive $\rho \in \frak B (\frak H )_*$ can be written as a possibly
infinite sum of positive multiples of orthogonal pure states of $\frak B
(\frak H )$ so for every positive $\rho \in \frak B (\frak H )_*$ there is
a pure state $\vartheta \in \frak B (\frak H )_*$ (so $\vartheta (A) =
(f,Af)$ for $A \in \frak B (\frak H )$ where $f \in \frak H$ is a unit
vector) and a number $s > 0$ so that $\rho \geq s\nu$.  It follows there
are non zero positive $\rho \in \frak B (\frak H )_*$ with $\omega \geq
\rho$ if and only if there are non zero positive multiples of pure states
$\vartheta \in \frak B (\frak H )_*$ so that $\omega \geq s\vartheta \geq
0.  $ Then it follows that if $\omega (\Lambda ) = \infty$ then $\omega$ is
$q$-pure if and only if there for every pure state $\vartheta \in \frak B
(\frak H )_*$ so that $\omega \geq s\vartheta \geq 0$ then $s = 0$ and if
$\omega (\Lambda ) < \infty$ then $\omega$ is $q$-pure if and only if
$\omega$ is pure in the ordinary sense of pure.  \end{proof}

In light of Theorem~\ref{sub-decomposition}, the proof of Theorem 3.10
of \cite{powers-holyoke} can easily be adapted to
give another 
characterization of $q$-pure boundary weights.  Suppose
$$
\omega (A) = \sum_{i=1}^r \lambda_i(f_i,Af_i)
$$
for $A \in \frak A (\frak H )$ where the $\{ f_i\}$ are orthonormal in
$\frak H _q$ and the $\lambda_i$ are positive numbers whose sum is finite.
Then $\omega$ is $q$-pure if and only if either $r = 1$ or for every set of
complex numbers $c_i$ for $0 < i < r+1$ so that
$$
0 < \sum_{i=1}^r \vert c_i\vert^2 < \infty\qquad \text{the vector} \qquad \sum
_{i=1}^r c_if_i \notin \frak H .
$$
In other words $\omega$ is $q$-pure if and only if it is pure in the ordinary
sense or there is no linear combination of the $f_i$ that lies in $\frak H
$.
Now one easily sees the reason for the strange fact that the sum of
two positive boundary weights can be $q$-pure.  The boundary weight
$$
\omega (A) = (f,Af) + (g,Ag)
$$
for $A \in \frak A (\frak H )$ and $f,g \in \frak H_q$ fails to be $q$-pure if and
only if $zf + g \in \frak H$ for
some complex number $z$.  With a little thought it is easy to construct
lots of examples of functions in $\frak H_q$ so that no linear combination
of
them is in $\frak H$.

\subsection{$q$-purity} 
We proceed to characterize the $q$-weight maps of range rank one which are $q$-pure, now that we are appropriately equipped with a concrete definition of $q$-purity in this case.

\begin{theorem} \label{rank-one-subordinates} Let $\fk$ be a separable Hilbert space. Suppose $\omega$
is a $q$-weight of range rank one over $\fk$ so $\omega$ can be
expressed in the form $\omega (\rho )(A) = \rho (T)\mu (A)$ for $\rho \in
\frak B (\frak K )_*$ and $A \in \frak A (\frak H )$ where $T$ is a
positive operator of norm one and $\mu$ is a positive element of $\frak A
(\frak H )_*$ and $\mu (I - \Lambda (T)) \leq 1$.  Suppose $T _1 \in \frak
B (\frak K )$ is a positive norm one operator and $T_1 \leq T$.  Let $\eta
(\rho ) = \lambda\rho (T_1)\mu (A)$ for $\rho \in \frak B (\frak K )_*$ for
$\lambda \geq 0$.  Then $\eta$ is a $q$-weight map which is $q$-subordinate to
$\omega$ (i.e., $\omega \geq_q \eta )$ if and only if $\lambda \leq (1 +
\mu (\Lambda (T - T_1)))^{-1}$ where if $\mu (\Lambda (T - T_1)) = \infty$
then $\omega \geq_q \eta$ only for $\lambda = 0$.
\end{theorem}
\begin{proof}  Assume $\omega$ and $\eta$ are as stated in the theorem.
Let $\pi ^\#$ and $\phi^\#$ be the generalized boundary representations of
$\omega$ and $\eta$, respectively.  Then we have
$$
\pi_t^\# (A) = (1 + \mu{\vert_t} (\Lambda (T)))^{-1} \mu{\vert_t}\vert
(A)\cdot T
$$
and
$$
\phi_t^\# (A) = (1 + \lambda\mu{\vert_t} (\Lambda (T_1)))^{-1} \lambda
\mu{\vert_t} (A)\cdot T_1
$$
for $A \in \frak B (\frak H )$ and $t > 0$.  Since $T_1$ is positive and $\|T_1\|=1$, there
is a sequence of vector states $\rho_n$ such that 
$\lim_{n \to \infty} \rho_n(T_1) = 1$, and since $T_1 \leq T \leq I$ we also have
$\lim_{n \to \infty} \rho_n(T) = 1$.  Therefore, to
determine whether $\pi_t^\# \geq \phi_t^\#$ need apply $\rho_o$ to
$\pi_t^\#$ and $\phi _t^\#$.  The $\omega \geq_q \eta$ if and only if
$$
(1 + \mu{\vert_t} (\Lambda (T)))^{-1} \leq \lambda (1 +
\lambda\mu{\vert_t} (\Lambda (T_1)))^{-1}
$$
for all $t > 0$.  Multiplying by the product of the denominators we
find the above inequality is equivalent to
$$
\lambda \leq (1 + \mu{\vert_t} (\Lambda (T - T_1)))^{-1}
$$
and since $\mu{\vert_t}(\Lambda(T - T_1)$ increases as $t$
decreases this inequality is valid for all $t$ if and only if $\lambda \leq
(1 + \mu (\Lambda (T - T_1))) ^{-1}.$
\end{proof}

\begin{cor} \label{q-pure-projection} Let $\fk$ be a separable Hilbert space. Suppose $\omega$ is a $q$-weight of range
rank one over $\fk$ so $\omega$ can be expressed in the form $\omega (\rho )(A) = \rho
(T)\mu (A)$ for $\rho \in \frak B (\frak K )_*$ and $A \in \frak A (\frak H
)$ where $T$ is a positive operator of norm one and $\mu$ is a positive
element of $\frak A (\frak H )_*$ and $\mu (I - \Lambda (T)) \leq 1$.
Suppose $T$ is not a projection and let
$$
T = \int s dF(s)
$$
be the spectral resolution of $T$.  Since $T$ is not a projection for some
$s_o \in (0,1)$ we have $F([s_o,1]) \neq T$.  Let $T_1 = F([s_o,1])T$ and
then let $\eta (\rho )(A) = \lambda\rho (T_1)\mu (A)$ for $\rho \in \frak B
(\frak K )_*$ and $A \in \frak A (\frak H )$ where $\lambda \geq 0$.  Then
$\eta$ is a $q$-weight with $\omega \geq_q \eta$ if and only if $0 \leq
\lambda \leq$ $1/(1 + \mu (\Lambda (T - T_1)))$.  There are always some
positive $\lambda$ satisfying this inequality since
$$
\mu (\Lambda (T - T_1)) \leq \kappa\mu (I - \Lambda (T)) \leq \kappa
$$
where $\kappa = s_o/(1 - s_o)$.  It follows that if the boundary weight
$\omega$ is $q$-pure then $T$ is a projection.
\end{cor}
\begin{proof}  Except for the estimate of $\mu (\Lambda (T - T_1))$ the
theorem follows immediately from Theorem~\ref{rank-one-subordinates}.  All that remains is
proving the estimate for $\mu (\Lambda (T - T_1))$.  Let $\kappa = s_o/(1 -
s_o)$.  We prove $\Lambda (T - T_1) \leq \kappa (I - \Lambda (T))$.

We note all the operators in the inequality we want to prove are
multiplication operators.  For example for $A \in \frak B (\frak K )$ we
have $( \Lambda (A)f) =$ $e^{-x}Af(x)$ for $x \in [0,\infty )$ and $f \in
\frak H$.  It follow that our inequality is valid if and only if it is
valid for all $x$.  Let $P = F([0,s_o)) =$ $I - F([s_o,1])$ so $T - T_1
\leq s_oP$ and $T_1P = 0$ so $TP \leq s_oP$.  Also we have
$$
T = T(I - P) + TP \leq I - P + TP \leq I - P + s_oP
$$
so
$$
(1 - s_o)P \leq I - T
$$
Now we have
\begin{align*}
\Lambda (T - T_1) &\leq s_o\Lambda (P) = \kappa (1 - s_o)\Lambda (P) =
\kappa (1 - s_o)e^{-x} P
\\
&\leq \kappa e^{-x} (I - T) = \kappa (e^{-x} - e^{-x} T) \leq \kappa (I - e
^{-x}T)
\\
&= \kappa (I - \Lambda (T))
\end{align*}
for all $x > 0$.  And so we have $\Lambda (T - T_1) \leq \kappa (I - \Lambda
(T))$ which yields
$$
\mu (\Lambda (T - T_1)) \leq \kappa\mu (I - \Lambda (T)) \leq \kappa .
$$
If $T$ is not a projection there is a $T_1$ satisfying the conditions of
the theorem so that $T_1$ is not a multiple of $T$ so $\omega$ is not
$q$-pure.  \end{proof}

\begin{theorem}  \label{Thm6} Let $\fk$ be a separable Hilbert space.
Suppose $\omega$ is a $q$-weight of range rank
one so $\omega$ can be expressed in the form $\omega (\rho )(A) = \rho
(T)\mu (A)$ for $\rho \in \frak B (\frak K )_*$ and $A \in \frak A (\frak H
)$ where $T$ is a positive operator of norm one and $\mu$ is a positive
element of $\frak A (\frak H )_*$ and $\mu (I - \Lambda (T)) \leq 1$.
Suppose $\vartheta \in \frak B (\frak H )_*$ is positive and $\mu \geq
\vartheta \geq 0$.  Let $\nu (A) = \lambda (\mu (A) - \vartheta (A))$ and
let $\eta (\rho )(A) = \rho (T)\nu (A)$ for $\rho \in \frak B (\frak K )_*$
and $A \in \frak A (\frak H )$.  Then $\eta$ is a $q$-weight subordinate to
$\omega$ if and only if $0 \leq \lambda \leq 1/(1 + \vartheta (\Lambda
(T)))$.  Conversely, if we maintain the assumptions made on $T$ and $\mu$ 
and we have $\eta (\rho ) = \rho (T)\nu (A)$ for $\rho \in \frak B (\frak K
)_*$ and $\omega \geq_q \eta \geq_q 0$, then there is a positive $\vartheta \in \frak B (\frak H )_*$
so that $\mu \geq \vartheta \geq 0$ and a number $\lambda \in [0,1]$ so
that $\nu = \lambda (1 + \vartheta (\Lambda (T)))^{-1}(\mu - \vartheta )$.
Thus, if $\omega$ is $q$-pure then $\mu$ is $q$-pure.
\end{theorem}
\begin{proof}  Assume the hypothesis and notation of the theorem is
satisfied.  Then if $\pi^\#$ and $\phi^\#$ are the generalized boundary
representations of $\omega$ and $\eta$, respectively we have
$$
\pi_t^\# (A) = (1 + \mu{\vert_t} (\Lambda (T)))^{-1} \mu{\vert_t}
(A)\cdot T
$$
and
$$
\phi_t^\# (A) = (1 + \nu{\vert_t} (\Lambda (T)))^{-1} \nu{\vert_t}
(A)\cdot T
$$
for $t > 0$.  Hence, $\omega \geq_q \eta$ if and only if
\begin{equation} \label{6.1}
(1 + \nu{\vert_t} (\Lambda (T)))^{-1} \nu{\vert_t}  \leq (1 + \mu
_{\vert_t} (\Lambda (T)))^{-1} \mu{\vert_t}
\end{equation}
for all $t > 0$.  If we replace $\Lambda (T)$ by $\Lambda (I)$ we have the
above conditions says $\mu \geq_q \nu$.  The analysis of the above order
relations is almost identical to the analysis in Theorem \ref{sub-decomposition} so rather
than repeat that argument we leave it to the reader to check the
conclusions are the same if one replaces $\Lambda$ with $\Lambda (T)$.  One
important point to remember is that $\mu$ satisfies the condition $\mu (I -
\Lambda (T)) \leq 1$.  All we need is to assume $\mu (I - \Lambda (T)) <
\infty$.  We mention this because without this assumption the proof fails
since $\mu (I - \Lambda (T))$ can be infinite for a boundary weight $\mu$.
(We have the condition $\mu (I - \Lambda (T)) \geq \mu (I - \Lambda )$ for
positive boundary weights so a bound on $\mu (I - \Lambda )$ does not give
us a bound on $\mu (I - \Lambda (T)).)$  With this said the conclusion of
the theorem follows.  Note that for this new ordering given above a
positive boundary weight satisfying $\mu (I - \Lambda (T)) < \infty$ is
pure with respect to this new ordering if and only if $\mu$ is $q$-pure.
So it follows that if $\omega$ is $q$-pure then $\mu$ is $q$-pure.
\end{proof}

Now we have all the pieces to give necessary and sufficient
conditions that a range rank one $q$-weight is $q$-pure.

We will make use of the following notation. If $A \in \fb(\fh)$ and $t>0$, then we define the operator of $\fb(\fh)$ given by
$$
A\vert_t = E(t,\infty) A E(t,\infty).
$$
We emphasize that in fact, for all $t>0$, we have that $A\vert_t \in \fa(\fh)$.

\begin{theorem}  \label{rank-one-q-pure}  
Let $\fk$ be a separable Hilbert space.
Suppose $\omega$ is a $q$-weight of range rank
one over $\fk$ so $\omega$ can be expressed in the form $\omega (\rho )(A) = \rho
(T)\mu (A)$ for $\rho \in \frak B (\frak K )_*$ and $A \in \frak A (\frak H
)$ where $T$ is a positive operator of norm one and $\mu$ is a positive
element of $\frak A (\frak H )_*$ and $\mu (I - \Lambda (T)) \leq 1$.  Then
$\omega$ is $q$-pure if and only if the following three conditions are met.
\begin{description}
\item[(i)] $T$ is a projection.

\item[(ii)] $\mu$ is $q$-pure.

\item[(iii)] If rank(T) $> 1$ and $e \in \frak B (\frak K )$ is a
rank one
projection with $T \geq e$ then $\mu (\Lambda (e)) = \infty$.
\end{description}
\end{theorem}
\begin{proof}  Assume $\omega$ is of the form given in the statement of the
theorem and suppose that $\omega$ is $q$-pure.  Then from Corollary \ref{q-pure-projection} we
know that $T$ is a projection hence condition (i) is satisfied. 
 And from Theorem~\ref{Thm6} we know that $\mu$ is
$q$-pure, hence condition (ii) is satisfied.  Suppose that rank(T) $> 1$ however condition $(iii)$ is not satisfied. Then there exists a rank one
projection $e \in \frak B ( \frak K )$ with $T \geq e$ and $\mu (\Lambda
(e)) < \infty$.  Let $T_1 = T - e$.  Then $\Lambda (T - T_1) = \Lambda (e)$
and we have $\mu (\Lambda (T - T_1)) < \infty$.   Thus by 
Theorem~\ref{rank-one-subordinates}   
there is a range rank one $q$-weight map $\eta$ with $\omega \geq_q
\eta$ such that $\eta$ is not a multiple of $\omega$.  Hence, $\omega$ is not
$q$-pure.  Since this contradicts our assumption, we conclude that if $\omega$ is
$q$-pure, then it satisfies conditions (i), (ii) and (iii).

Now we assume $\omega$ is of the form given above and three conditions
given above are satisfied and $\eta$ is a $q$-weight map so that
$\omega \geq_q \eta$.  The proof will be complete when we show $\eta =
\lambda \omega$ with $\lambda \in [0,1]$.  Let $\pi^\#$ and $\phi^\#$ be
the generalized boundary representations of $\omega$ and $\eta$,
respectively.  Consider the mapping $\psi$ of $\frak B (\frak K )_*$ into
itself given by $(\psi(\rho) )(A) = \rho (TAT)$ for $A \in \frak B (\frak K
).  $ Note $\psi$ is completely positive and $\psi^2 = \psi$.  We will show
that $\eta (\rho ) = \eta (\psi (\rho ))$ for all $\rho \in \frak B (\frak
K )_*$.  Assume $\rho \in \frak B (\frak K )$ and $A \in \frak A (\frak H
)$ and both $\rho$ and $A$ are positive.  Since $\omega \geq_q \eta$ we
have $\omega \geq \eta$ so $\omega (\rho )(A) = \rho (T)\mu (A) \geq$ $\eta
(\rho )(A)$.  Hence, we have for positive $\rho \in \frak B (\frak K )$ if
$\rho (T) = 0$ then $\eta (\rho ) = 0$.  We write $\frak K = \frak K_1
\oplus \frak K_2$ with $\frak K _1 = T\frak K$ and $\frak K_2 = (I -
T)\frak K$.  Now each $X \in \frak B (\frak K )$ can be uniquely expressed
in matrix form
$$
X = \begin{bmatrix}
X_{11}&X_{12}\\
X_{21}&X_{22}
\end{bmatrix}
$$
where $X_{ij}$ maps $\frak K_i$ into $\frak K_j$ for $i,j = 1,2$.  Let $\breve{\omega}:\fa(\fh) \to \fb(\fk)$ denote 
the dualized $q$-weight map for $\omega$. For $A\in \frak A (\frak H )$ we have that
$\breve{\omega} (A)$ can be expressed in matrix form as
$$
\breve{\omega} (A) =
\begin{bmatrix}
\mu (A)I&0
\\
0&0
\end{bmatrix}
$$

Let $\breve{\eta}:\fa(\fh) \to \fb(\fk)$ denote the dualized $q$-weight map for $\eta$.
Since $\omega \geq \eta$, we have that $\breve{\omega} \geq \breve{\eta}$ hence
$$
\breve{\omega} (A) =
\begin{bmatrix}
\mu (A)I&0
\\
0&0
\end{bmatrix}
\geq \breve{\eta} (A) =
\begin{bmatrix}
\breta(A)_{11}&\breta(A)_{12}
\\
\breta(A)_{21}&\breta(A)_{22}
\end{bmatrix}
\geq 0
$$
for positive $A \in \frak A (\frak H )$.  This show us that $\breta(A)_{22} =
0$ for positive $A$ and since $\breta(A)_{22} = 0$ the only way that the above matrix can be positive is for both $\breta(A)_{12} = 0$ and $\breta(A)_{21} = 0.  $ Hence, $\breta(A)_{12} =
\breta(A)_{21} = \breta(A)_{22} = 0$ for all positive $A \in \frak A (\frak H )$. 
Since every $A \in \frak A (\frak H )$ is the complex linear
combination of four positive elements of $\frak A (\frak H )$ we have
$$
\breve{\eta} (A) =
\begin{bmatrix}
\breta(A)_{11}& 0
\\
0& 0
\end{bmatrix}
, \qquad \forall A\in\fa(\fh).
$$
Since the mapping $A \rightarrow TAT$ in matrix form sets all the $A_{ij}$
equal to zero except $A_{11}$ it follows that $\eta (\rho ) = \eta (\psi
(\rho ))$ for all $\rho \in \frak B (\frak K )_*$.

Now let $\pi_t^\#$ and $\phi_t^\#$ be the generalized boundary representations
of
$\omega$ and $\eta$, respectively.  Since $\omega \geq_q \eta$ we have
$$
\pi_t^\# (\rho )(A) = \rho (T)(1 + \mu{\vert_t} (\Lambda (T)))^{-1}\mu
{\vert_t} (A) \geq \eta ((I + \hat \Lambda\eta{\vert_t} )^{-1}\rho
)(A{\vert_t} ) \geq 0
$$
for positive $\rho \in \frak B (\frak K )_*$, positive $A \in \frak A
(\frak H )$ and $t > 0$.  Now let us replace $\rho$ in the above inequality
by $\vartheta \in \frak B (\frak K ) _*$ given by $\vartheta (A) =$ $\rho
(A) + \eta (\rho )(\Lambda (A){\vert_t} )$.  Note for $\rho \geq 0$
we have $\vartheta \geq 0$.  With this replacement we have
$$
(\rho (T) + \eta (\rho ){\vert_t} (\Lambda (T)))(1 + \mu{\vert_t}
(\Lambda (T)))^{-1}\mu (A) \geq \eta (\rho ){\vert_t} (A)
$$
for positive $A \in \frak A (\frak H )$, positive $\rho \in \frak B (\frak
K )_*$ and $t > 0$.  Now assume further that $\rho$ is a state and $\rho
(T) = 1$.  Then we have
$$
(1 + \mu{\vert_t} (\Lambda (T)))^{-1}\mu{\vert_t}  \geq (1 + \eta
(\rho ){\vert_t} (\Lambda (T)))^{-1} \eta (\rho ){\vert_t}
$$
for all $t > 0$.  But this is exactly the situation we had with
inequality \eqref{6.1} in the previous theorem with $\nu$ replaced by $\eta (\rho
)$.
Hence, we conclude $\eta (\rho )$ is of the form given in the previous
theorem and since $\mu$ is $q$-pure we have $\eta (\rho ) = \lambda\mu$ for
some $\lambda \geq 0$.  Hence, we have shown that $\eta (\rho ) =
\lambda\mu$ for each positive $\rho \in \frak B (\frak K )_*$ so that $\rho
(T) =$ $\rho (I)$.  Now for arbitrary positive $\rho \in \frak B (\frak K
)_*$ we have $\eta (\rho ) = \eta (\psi (\rho ))$ and $\psi (\rho )$ is
positive and $\psi (\rho )(T) = \psi (\rho )(I)$ so $\eta (\rho ) =
\lambda\mu$ for all positive $\rho \in \frak B (\frak K )_*$.  Then by
linearity we have $\eta (\rho ) = \lambda\mu$ for all $\rho \in \frak B
(\frak K )_*$ where $\lambda$ depends linearly on $\rho$ and, therefore,
there is an operator $C \in \frak B (\frak K )$ so that $\eta (\rho ) = \rho
(C)\mu$ for all $\rho \in \frak B (\frak K )_*$.  From the positivity of
$\eta$ we see that $C \geq 0$ so we can write $C =  \lambda T _1$ where
$\lambda \geq 0$ and $T_1$ is a positive operator of norm one.  If $\lambda
= 0$ then $\eta = 0\omega = 0$ and the proof is complete so we assume
$\lambda > 0$.  Now we have shown that $\eta (\rho ) = \eta (\psi (\rho ))$
for $\rho \in \frak B (\frak K )_*$ from which it follows that $T_1 =
TT_1T$ and since $T$ is a projection and $T_1$ is of norm one we have $T
\geq T_1 \geq 0$.  If $T_1$ is a multiple of $T$ the proof is complete so
we assume $T_1$ is not a multiple of $T$.  In this case rank(T) $\geq 2$
since in the rank one case $T_1$ must be a multiple of $T$.  Now we can
apply Theorem \ref{rank-one-subordinates} which tells us that $\lambda \leq (1 + \mu (\Lambda
(T - T_1))) ^{-1}$ so if $\lambda > 0$ we have $\mu (\Lambda (T - T_1)) <
\infty$.  Since $T_1$ is not a multiple of $T$ we have $T - T_1$ is a
positive non zero operator so there is a rank one hermitian $e \in \frak B
(\frak K )$ so that $\Vert T - T_1\Vert e \leq T - T_1$.  Then we have
$\Vert T - T_1\Vert\mu (\Lambda (e)) \leq \mu (\Lambda (T - T_1)) < \infty$
so condition (iii) is violated.  Assuming $T_1$ was not a multiple of $T$
has led to a contradiction so we conclude $T_1 = T$ and, therefore, $\eta
= \lambda\omega.$
\end{proof}

\subsection{$q$-corners}
Next we consider the problem of identifying the $q$-corners
between range rank one $q$-weight maps.  The first step is the  observation that in the type II$_0$ case, all $q$-corners between range rank one $q$-weight maps must have range rank one in the natural sense.

\begin{theorem}\label{corners-are-range-rank-one}
Let $\omega_1, \omega_2$ be q-weight maps of range rank one over finite-dimensional Hilbert spaces $\fk_1$ and $\fk_2$ respectively, whose associated normal spines are trivial. If $\gamma$ is a non-zero $q$-corner from $\omega_1$ to $\omega_2$, then it has range  rank one, i.e. there exist $S \in \fb(\fk_2,\fk_1)$ and $\ell \in \fa(\mathfrak{H}_2,\mathfrak{H}_1)_*$ such that
\begin{equation}\label{form-rank-one-q-corner}
\gamma(\mu)(B) = \ell(B) \mu(S), \qquad \forall B \in \fa(\mathfrak{H}_2,\mathfrak{H}_1), \forall \mu \in \fb(\fk_2,\fk_1)_*.
\end{equation}
\end{theorem}
\begin{proof}
Suppose that $\gamma$ is a $q$-corner from $\omega_1$ to $\omega_2$ and let $\omega$ be the $q$-weight map over $\fk_1 \oplus \fk_2$ defined by
$$
\omega = \begin{pmatrix} \omega_1 & \gamma \\ \gamma^* & \omega_2 \end{pmatrix}.
$$
Notice that $\omega$ has finite range rank, since $\fk_1 \oplus \fk_2$ is finite-dimensional. Let $\Pi_t^\#$ be the generalized boundary representation associated to $\omega$. 
Notice that since the normal spines of the generalized boundary representations associated 
to $\omega_1$ and $\omega_2$ are zero, it follows that by positivity the same must hold for the
 normal spine of $\Pi_t^\#$. Thus 
by Theorem~\ref{twoz} there exists a boundary expectation for $\omega$. Let us fix a boundary 
expectation $L$ for $\omega$ obtained as in statement of Theorem~\ref{twoz}, namely as a cluster 
point of $\Pi^\#_t\circ \Lambda$. Notice that in this case, for $j=1,2$,  $L_{jj}$ is a 
boundary expectation for $\omega_j$. Since for each $j=1,2$, $\breve{\omega}_j$ is range rank one and $L_{jj}$ is a boundary expectation for $\omega_j$, we have $\range(L_{jj})=\range(\breve{\omega}_j)$, it follows that $\dim(\range(L_{jj}))=1$.  Thus, by Lemma~\ref{L-rank-one}, we have that $\dim(\range(L_{12}))=1$.

Since $L$ is a boundary expectation for $\omega$, we have that $\range(L)=\range(\breve{\omega})$, therefore there exist $\ell \in \fa(\mathfrak{H}_2, \mathfrak{H}_1)_*$ and $S\in \fb(\fk_2, \fk_1)$ such that   
$$
\breve{\omega}\begin{pmatrix} A & B \\ C & D \end{pmatrix} = 
\begin{pmatrix}\breve{\omega}_1(A) & \ell(B) S \\ \ell^*(C) S^* & \breve{\omega}_2(D) \end{pmatrix}= 
\begin{pmatrix}\rho_1(A)T_1 & \ell(B) S \\ \ell^*(C) S^* & \rho_2(D) T_2 \end{pmatrix}
$$
Therefore, $\gamma$ is given by \eqref{form-rank-one-q-corner}, i.e. it is range rank one. 
\end{proof}

In the formal definition of a
$q$-corner between two $q$-weights $\omega$ and $\eta$ we assume $\omega$
and $\eta$ are defined on different Hilbert spaces.  In the following we
find it is notationally more efficient to have the weights defined on the
same Hilbert space.  Naturally the weights need to live on orthogonal
subspaces which we define as follows.

\begin{definition}\label{definition8}  Suppose $\eta_1$ and $\eta_2$ are
$q$-weight maps defined on $\ah$ with $\fh = \fk\otimes
L^2(0,\infty )$ for $\fk$ separable.  We say $\eta_1$ and $\eta_2$ are \emph{completely orthogonal} if
there are orthogonal projections $E_1,\medspace E_2 \in \frak B
(\frak K )$ so that, if we denote $\psi_{ij}(\rho )(A) = \rho
(E_iAE_j)$ for $\rho \in \frak B(\frak K )_*$ and $A \in \frak B (\frak K )$, then
\begin{align*}
\eta_1 (\rho )(A) & = \eta_1 (\psi_{11}(\rho ))((E_1\otimes I)A(E_1\otimes I)), \\
\eta_2 (\rho )(A) &= \eta_2 (\psi_{22}(\rho ))((E_2\otimes I)A(E_2\otimes I))
\end{align*}
for $A \in \frak A (\frak H )$ and $\rho \in \frak B (\frak K )_*$.

A map $\gamma:\fb(\fk)_* \to \fa(\fh)_*$ is called an \emph{\iq $q$-corner} from $\eta_1$ and 
$\eta_2$ if for all $\rho\in \fb(\fk)_*,$ $A\in \fa(\fh)$,
$$
\gamma (\rho )(A) = \gamma (\psi_{12}(\rho ))((E_1\otimes I)A(E_2\otimes I))
$$
and furthermore the map $\fb(\fk)_* \to \fa(\fh)_*$ given by
$$
\rho \mapsto \eta_1 (\rho ) + \gamma (\rho ) + \gamma^*(\rho )
+ \eta_2 (\rho )
$$
defines a $q$-weight map.
\end{definition}

\begin{remark} \label{2by2}
We will make use of the following observation. Let $E_1, E_2$ be orthogonal projection on $\fk$ with $E_1+E_2=I$. Given a map $\phi: \ah \to \bk$ , let $\phi_b : \ah \to \fb(\fk \oplus \fk)$ be the map given by
\begin{equation}\label{matricial}
\phi_b(A) = \begin{bmatrix} E_1\phi(A)E & E_1\phi(A)E_2 \\
 E_2\phi(A)E_1 & E_2\phi(A)E_2
\end{bmatrix}
\end{equation}
Then $\phi$ is completely positive if an only if $\phi_b$ is completely positive. 
We will often refer to the map $\phi_b$ as the \emph{matricial notation} for the map $\phi$, and we will abbreviate it by identifying $\phi$ with $\phi_b$ when invoking formula \eqref{matricial}.

\end{remark}

We note that this definition is compatible with the earlier notion of $q$-corner.

\begin{prop}\label{internal-to-external}
For $j=1,2$, Let $\fk_j$ be a Hilbert space, let $\fh_j = \fk_j \otimes L^2(0,\infty)$ and let $\fk=\fk_1 \oplus \fk_2$ and $\fh=\fh_1 \oplus \fh_2$.
For each $j=1,2$, let $\omega_j$ be a $q$-weight map over $\fk_j$, let $E_j$ be the canonical projection from $\fk$ onto $\fk_j$, and let $\eta_j$ be the $q$-weight map over $\fk$ given by
\begin{equation*}
\eta_j(\rho)(A) = \omega_j(\rho_{jj}) (A_{jj}), \qquad \forall \rho \in \bk_*, \forall A \in \ah.
\end{equation*}
Then $\eta_1$ and $\eta_2$ are completely orthogonal, and moreover there exists a $q$-corner $\sigma$ from $\omega_1$ to $\omega_2$ if and only if there exists an \iq $q$-corner $\gamma$ from $\eta_1$ to $\eta_2$, and we have that
$$
\gamma(\rho)(A) = \sigma(\rho_{12})(A_{12}), \qquad \forall \rho \in \bk_*, A\in\ah.
$$
\end{prop}
\begin{proof} Notice that for all $i,j=1,2$,
\begin{equation}\label{eq-schur-internal}
\rho_{ij} = [\psi_{ij}(\rho )]_{ij} , \qquad A_{ij}  = [(E_i\otimes I)A(E_j\otimes I)]_{ij},
\end{equation}
therefore
\begin{align*}
\eta_j (\rho )(A) & = \omega_j(\rho_{jj}) (A_{jj}) = \eta_j (\psi_{jj}(\rho ))((E_j\otimes I)A(E_j\otimes I))
\end{align*}
hence $\eta_1$ and $\eta_2$ are completely orthogonal.

Suppose that $\sigma$ is a $q$-corner from $\omega_1$ to $\omega_2$. Let 
$$
\gamma(\rho)(A) = \sigma(\rho_{ij})(A_{ij}), \qquad \forall \rho \in \bk_*, \forall A \in \ah.
$$
Notice that by \eqref{eq-schur-internal}, we have that $\gamma$ satisfies the equation
$$
\gamma (\rho )(A) = \gamma (\psi_{12}(\rho ))((E_1\otimes I)A(E_2\otimes I)), \qquad \forall \rho \in \bk_*, \forall A \in \ah.
$$
Now note that if we define $\eta:\fb(\fk)_* \to \fa(\fh)_*$ by
$
\eta(\rho) =  \eta_1 (\rho ) + \gamma (\rho ) + \gamma^*(\rho )
+ \eta_2 (\rho )$, 
for all $\rho \in \bk_*$, then
$$
\eta(\rho)(A) =  \omega_1 (\rho_{11} )(A_{11}) + \sigma (\rho_{12} )(A_{12}) + \sigma^*(\rho_{21} )(A_{21}) + \omega_2 (\rho_{22} )(A_{22})
$$
for all $\rho \in \bk_*$ and $A \in \ah$. Therefore, it follows immediately from the definition of $q$-corner, that $\eta$ is a $q$-weight map over $\fk$.

Conversely, suppose that $\gamma$ is an \iq $q$-corner from $\eta_1$ to $\eta_2$. 
Let $\sigma: 
\mathfrak{B}(\mathfrak{K}_2, \mathfrak{K}_1)_*: \rightarrow \fa(\mathfrak{H}_2, \mathfrak{H}_1)_*$ be the map given by
$$
\sigma(\mu_{12})(A_{12}) = \gamma(\mu)(A)
$$
for all $\mu \in \bk_*$ and $A \in \ah$.  We show that $\sigma$ is well-defined.  First note that
for every $\ell \in \mathfrak{B}(\mathfrak{K}_2, \mathfrak{K}_1)_*$  there exists $\mu \in \bk_*$ such that $\ell=\mu_{12}$ 
and for every $X \in \fa(\mathfrak{H}_2, \mathfrak{H}_1)_*$ there exists $A \in \ah$ such that $X=A_{ij}$. Furthermore, since $\gamma$ is an \iq $q$-corner, for any other representatives $\mu'$ and $A'$ such that $\ell=\mu_{ij}'$  and $X=A_{ij}'$, we have that 
$$
\psi_{12}(\mu) = \psi_{12}(\mu'), \qquad (E_1\otimes I)A(E_2\otimes I) = (E_1\otimes I)A'(E_2\otimes I)
$$
hence
$$
\gamma(\mu')(A') = \gamma(\mu)(A),
$$
thus $\sigma$ is well-defined.

Now note that if we define once again $\eta:\fb(\fk)_* \to \fa(\fh)_*$ by
$
\eta(\rho) =  \eta_1 (\rho ) + \gamma (\rho ) + \gamma^*(\rho )
+ \eta_2 (\rho )$, 
for all $\rho \in \bk_*$, then we obtain once more
$$
\eta(\rho)(A) =  \omega_1 (\rho_{11} )(A_{11}) + \sigma (\rho_{12} )(A_{12}) + \sigma^*(\rho_{21} )(A_{21}) + \omega_2 (\rho_{22} )(A_{22})
$$
or 
$$
\eta(\rho) = \begin{bmatrix} \omega_1(\rho_{11}) & \sigma(\rho_{12}) \\ \sigma^*(\rho_{21})
& \omega_2(\rho_{22}) \end{bmatrix}.
$$
Therefore, if $\eta$ is a $q$-weight, then $\sigma$ is a $q$-corner from $\omega_1$ to $\omega_2$.
\end{proof}

The following  lemma will be a useful tool in the remainder.

\begin{lemma}\label{extractor}
Let $\fk$ be a Hilbert space, let $\sigma: \bk \to M_2(\cc)$ be a $*$-preserving linear map such that $\sigma_{12}\neq 0$, let $P_1, P_2 \in \bk$ be two non-zero orthogonal projections and let $Q \in \bk$ with norm one. Define
$\phi: \bk \to \fb(\fk \oplus \fk)$ given by 
$$
\phi(A) = \begin{bmatrix} \sigma_{11}(A) P_1 & \sigma_{12}(A) Q \\
\sigma_{21}(A) Q^* & \sigma_{22}(A) P_2 
\end{bmatrix}
$$
Then $\phi$ is (completely) positive if and only if $P_1QP_2=Q$, $QQ^* \leq P_1$, $Q^*Q \leq P_2$, and $\sigma$ is (completely) positive.
\end{lemma}
\begin{proof}
Suppose that $\phi$ is positive. Since $\sigma\neq 0$, there exists $A_0\geq 0$ such that $c_{12}=\sigma_{12}(A) \neq 0$. Notice that by positivity of $\phi$ and the fact that $P_j$ is a nonzero projection, we have that $c_{jj}=\sigma_{jj}(A) > 0$ for $j=1,2$. By positivity of $\phi$, we must have that for all $x,y\in \bk$,
$$
|c_{12}|^2  |( Qx,y)| ^2 \leq  c_{11}c_{22} (P_1y,y) (P_2 x, x)
$$
Thus we have that $\ker P_2 \subseteq \ker Q$ and $\ker P_1 \subseteq \ker Q^*$, hence $P_2Q^* = Q^*$ and $P_1Q=P_1$, or $P_1QP_2=Q$. Furthermore, 
$$
QQ^* = P_1QQ^*P_1 \leq P_1 \|Q^*\|^2 = P_1
$$
and similarly $Q^*Q\leq P_2$. Now let $x_n$ be a sequence of unit vectors in $P_2\fk$ such that $\|Qx_n\| \to 1$, and let $y_n = Qx_n$ (without loss of generality they are nonzero vectors). Let $U_n$ be the rank one partial isometry such that $U_n x_n = y_n/\|y_n\|$ and $U_n$ is zero on $\{ x_n \}^\perp$.  If $G_n=U_n\oplus U_n^*$, then
$$
\psi_n(A) = G_n^* \phi(A) G_n = \begin{bmatrix} \sigma_{11}(A) U_n^*U_n  & \sigma_{12}(A) \|y_n\| U_n \\
\sigma_{21}(A) \|y_n\| U_n & \sigma_{22}(A) U_n U_n^*
\end{bmatrix}
$$
Since $P_1$ and $P_2$ are orthogonal, it is clear that the range of $\psi_n$ is isomorphic to $M_2(\cc)$  and the map
$$
A \mapsto \begin{bmatrix} \sigma_{11}(A)  & \sigma_{12}(A) \|y_n \| \\
\sigma_{21}(A) \|y_n\| & \sigma_{22}(A)
\end{bmatrix}
$$
is (completely) positive if $\phi$ is (completely) positive. By taking limits, we obtain that $\sigma$ is (completely) positive.

Conversely suppose that  $P_1QP_2=Q$, $QQ^* \leq P_1$, $Q^*Q \leq P_2$, and $\sigma$ is (completely) positive. Then observe that 
\begin{align*}
\phi(A) & = \begin{bmatrix} \sigma_{11}(A) P_1 & \sigma_{12}(A) Q \\
\sigma_{21}(A) Q^* & \sigma_{22}(A) P_2 
\end{bmatrix} = \begin{bmatrix}  P_1 & 0 \\
0 & P_2 
\end{bmatrix}
\begin{bmatrix} \sigma_{11}(A) I & \sigma_{12}(A) Q \\
\sigma_{21}(A) Q^* & \sigma_{22}(A) P_2 
\end{bmatrix}
\begin{bmatrix}  P_1 & 0 \\
0 & P_2 
\end{bmatrix}
\end{align*}
But
$$
\begin{bmatrix} \sigma_{11}(A) I & \sigma_{12}(A) Q \\
\sigma_{21}(A) Q^* & \sigma_{22}(A) P_2 
\end{bmatrix}
=
\begin{bmatrix} \sigma_{11}(A) I & \sigma_{12}(A) Q \\
\sigma_{21}(A) Q^* & \sigma_{22}(A) Q^*Q 
\end{bmatrix}
+
\begin{bmatrix} 0 & 0 \\
0 & \sigma_{22}(A) (P_2-Q^*Q) 
\end{bmatrix}
$$
and
$$
\begin{bmatrix} \sigma_{11}(A) I & \sigma_{12}(A) Q \\
\sigma_{21}(A) Q^* & \sigma_{22}(A) Q^*Q 
\end{bmatrix} = 
\begin{bmatrix} I & 0 \\
0 & Q^*
\end{bmatrix} \cdot
\begin{bmatrix} \sigma_{11}(A) I & \sigma_{12}(A) I \\
\sigma_{21}(A) I & \sigma_{22}(A) I 
\end{bmatrix} \cdot
\begin{bmatrix} I & 0 \\
0 & Q
\end{bmatrix}.
$$

\end{proof}

\begin{lemma}  \label{ordinary} 
Let $\fk$ be a finite-dimensional Hilbert space.
Suppose $\omega$ and $\eta$ are completely
orthogonal $q$-pure range rank one $q$-weight maps over $\fk$ so
$\omega (\rho ) = \rho (T_1) \mu$ and $\eta (\rho ) = \rho (T_2)\nu$ for
$\rho \in \frak B (\frak K )_*$ and $\mu$ and $\nu$ are $q$-weights so
that $\mu (I - \Lambda (T_1)) \leq 1$ and $\nu (I - \Lambda (T_2)) \leq 1$. Suppose that $\gamma$ is a non-zero range rank one \iq corner from $\omega$
to $\eta$.  Then there exist $\tau \in \frak A (\frak H )_*$ and $Q \in \frak B (\frak K
)$ satisfying
\begin{equation}\label{q-eqs}
\|Q\| = 1, \qquad QQ^* \leq T_1\quad \text{and} \quad Q^*Q \leq T_2
\end{equation}
such that $\gamma (\rho ) = \rho(Q)\tau$ for all $\rho \in \bk_*$.  For $t
> 0$ let
$$
h(t) = \frac{(1+\mu\vert_t(\Lambda (T_1)))^{1/2}\;\; (1+\nu\vert_t(\Lambda
(T_2)))^{1/2}} {1+\tau\vert_t(\Lambda (Q))}
$$
Then $h(t)$ has a limit as $t\to 0+$, and furthermore $\vert h(t)\vert$ is a non-increasing
 function of $t$ which is bounded above. If $\kappa = \lim_{t \to 0+} |h(t)|$, then 
$1 \leq |h(t)| \leq \kappa$ for all $t>0$ and $\kappa \gamma$ is a ordinary \iq corner from
 $\omega$ to $\eta$.  

Conversely, given $Q\in\bk$ satisfying \eqref{q-eqs}, let $\gamma:\bk_* \to \ah_*$ be 
given by $\gamma (\rho ) = \rho (Q)\tau$ for $\rho\in \frak B (\frak K )_*$, and let $h(t)$ be defined as above for $t>0$. If there exists $\kappa\geq 1$ such that $|h(t)|\leq \kappa$ for all $t>0$ and
$\kappa \gamma$ is an ordinary \iq corner from $\omega$ to $\eta$, then $\gamma$ is an \iq $q$-corner from $\omega$ to $\eta$. If $\omega$ and $\eta$ are unbounded then the
\iq corner $\kappa\gamma$ from $\omega$ to $\eta$ is trivially maximal in that
$\lambda\kappa\gamma$ is not an \iq corner from $\omega$ to $\eta$ for $\lambda
> 1$.
\end{lemma}
\begin{proof}  
Suppose that $\gamma$ is a non-zero range rank one \iq $q$-corner from $\omega$ to $\eta$,
so that the map $\Theta  = \omega  + \gamma  + \gamma^* + \eta $
is a $q$-weight map from $\bk_*$ to $\ah_*$.  We write $\Theta$ in matrix form
$$
\Theta =
\begin{bmatrix}
\omega&\gamma
\\
\gamma^*&\eta
\end{bmatrix}.
$$

Since $\omega$ and $\eta$ are range rank one, there exist $T_1$ and $T_2$ non-zero positive operators so that $\omega (\rho ) = \rho (T_1) \mu$ and $\eta (\rho ) = \rho (T_2)\nu$ for
$\rho \in \frak B (\frak K )_*$ and $\mu$ and $\nu$ are $q$-weights so
that $\mu (I - \Lambda (T_1)) \leq 1$ and $\nu (I - \Lambda (T_2)) \leq 1$.
 Since they are $q$-pure, by Theorem~\ref{rank-one-q-pure} we 
have that $T_1$ and $T_2$ are projections. Furthermore, notice that 
if $\omega$ and $\eta$ are are completely orthogonal with respect to orthogonal 
projections $E_1$ and $E_2$, then we automatically have that 
$T_j \leq E_j$ for $j=1,2$. Recall that without loss of generality, we may assume that $E_1+E_2=I$.

Suppose that $\gamma$ has range rank one, i.e. there exists $Q \in \bk$ with norm one and $\tau \in \ah_*$ such that $\gamma(\rho) = \rho(Q) \tau$ for all $\rho \in \bk_*$. Since it is an \iq corner,
we have that $E_1QE_2 =  Q$. Thus,  by Remark~\ref{2by2}, $\brTheta$ is completely positive if and only if the map 
$$
A \mapsto  \begin{bmatrix} E_1\brTheta(A)E_1 & E_1\brTheta(A)E_2 \\
 E_2\brTheta(A)E_1 & E_2\brTheta(A)E_2 
\end{bmatrix} = \begin{bmatrix} \mu(A)T_1 & \tau(A)Q \\
\tau^*(A)Q^* & \nu(A) T_2
\end{bmatrix}
$$
from $\ah$ to $\fb(\fk\oplus\fk)$ is completely positive. It follows from Lemma~\ref{extractor} that $Q$ satisfies \eqref{q-eqs}.  Thus we see any candidate for a
range rank one $q$-corner must be of the form given in the statement of the
lemma.

Calculating the generalized boundary representation of $\Theta$ we find
$$
\Pi_t^\# (A) =
\begin{bmatrix}
(1+\mu{\vert_t}(\Lambda (T_1)))^{-1}\mu{\vert_t}(A)T_1&
(1+\tau{\vert_t}(\Lambda (Q)))^{-1}\tau{\vert_t}(A)Q
\\
(1+\tau^*{\vert_t}(\Lambda (Q^*)))^{-1}\tau^*{\vert_t}(A)Q^*&
(1+\nu{\vert_t}(\Lambda (T_2)))^{-1}\nu{\vert_t}(A)T_2
\end{bmatrix}.
$$
To simplify notation we make the following definitions
\begin{align*}
a(t) & = 1 + \mu{\vert_t} (\Lambda (T_1)),\\  
b(t) & = 1 + \nu{\vert_t} (\Lambda (T_2)), \\
c(t) & = 1 + \tau{\vert_t} (\Lambda (Q)).
\end{align*}
$$
h(t) = \frac {\sqrt{a(t)b(t)}} {c(t)}.
$$
By Lemma~\ref{extractor}, $\Pi_t^\#$ is completely positive if and only if
$$
A \mapsto 
\begin{bmatrix}
(1+\mu{\vert_t}(\Lambda (T_1)))^{-1}\mu{\vert_t}(A)&
(1+\tau{\vert_t}(\Lambda (Q)))^{-1}\tau{\vert_t}(A)
\\
(1+\tau^*{\vert_t}(\Lambda (Q^*)))^{-1}\tau^*{\vert_t}(A)&
(1+\nu{\vert_t}(\Lambda (T_2)))^{-1}\nu{\vert_t}(A)
\end{bmatrix}
$$
is completely positive. Since taking the Schur product of this mapping with a matrix
$$
\begin{bmatrix}
\vert x\vert^2&xy
\\
\overline{xy}&\vert y\vert^2
\end{bmatrix}
$$
where $xy \neq 0$ preserves completely positivity, it follows that
$\Pi_t^\#$ is completely positive if and only if
$$
\psi_t(A) =
\begin{bmatrix}
\mu{\vert_t}(A)&\vert h(t)\vert\tau{\vert_t}(A)
\\
\vert h(t)\vert\tau^*{\vert_t}(A)&\nu{\vert_t}(A)
\end{bmatrix}
$$
is completely positive for $t > 0$.  Note in the above expression we can
replace $\vert h(t)\vert$ by $h(t)$ since multiplying the upper off
diagonal entry by $z$ and the lower off diagonal entry by $\overline {z}$
where $\vert z\vert = 1$ preserves complete positivity.

Returning to our proof we see that if $\gamma$ is an \iq $q$-corner then  $\psi_t$ is completely positive for every $t > 0$.  We
will now show that $\vert h(t)\vert$ is a non-increasing function of $t$
and has a finite limit $\kappa \geq 1$ as $t \rightarrow 0+$.  First we show
$h(t)$ is bounded.  To accomplish this we need to find a positive element
$A \in \frak A (\frak H )$ so that $\mu (A) > 0,\medspace \nu (A) > 0$ and
$\tau (A) \neq 0$.  Since $\tau \neq 0$ we have $\tau{\vert_{t_o}}
\neq 0$ for some $t_o > 0$.  Hence there is an operator $B \in \frak B
(\frak H )$ with $\Vert B\Vert \leq 1$ so that $\tau (E(t,\infty
)BE(t,\infty )) \neq 0$.  Since $\mu$ and $\nu$ are not zero there is a
$t_1 > 0$ so that $\mu{\vert_{t_1}}  (I) > 0$ and $\nu{\vert_{t_1}}
(I) > 0$.  Let $s = \min(t_o,t_1)$ and let $A = 3E(s,\infty ) +
E(t_o,\infty )(zB + \overline {z}B^*)E(t_o,\infty )$ where $z \in \cc$
and $\vert z\vert \leq 1$.  Then we have $A \in \frak A (\frak H )$ and
$\mu (A) > 0,\medspace \nu (A) > 0$ and
$$
\tau (A) = 3\tau (E(s,\infty )) + 2Re(z\tau{\vert_{t_o}}  (B)).
$$
Since $\tau{\vert_{t_o}}  (B) \neq 0$ we can arrange it so $\tau (A) \neq 0$  with an appropriate choice of $z$ with $\vert z\vert \leq 1$.  Since $A
\geq 0$ we have $\psi _t(A) \geq 0$ for all $t > 0$ and for $0 < t \leq s$
we have $\psi_t(A)$ is constant.  Since the determinant of $\psi_t$ is
positive we have
$$
\vert h(t)\tau (A)\vert^2 \leq \mu (A)\nu (A)
$$
for $t \in (0,s]$ and since $\tau (A) \neq 0$ we have $h(t)$ is bounded for
$t \in (0,s]$ and $h$ is clearly bounded for $t > s$ so $h$ is bounded.

Let $A=T_1 + Q + Q^* + T_2$, which is a positive operator.
For $0 < t < s$,  we define $\Lambda_t^s(A) = E(t,s)\Lambda (A)E(t,s)$ and note that
$$
\mu (\Lambda_t^s(A)) = \mu{\vert_t} (\Lambda (A)) - \mu{\vert_s}
(\Lambda (A)).
$$
The same applies to $\nu$ and $\tau$.  Then we have for $0 < t < s$ the matrix
$$
\psi_t(\Lambda_t^s(A)) = 
\begin{bmatrix}
\mu{\vert_t}(\Lambda_t^s(T_1))&
\vert h(t)\vert\tau{\vert_t}(\Lambda_t^s(Q))
\\
\vert h(t)\vert\tau^*{\vert_t}(\Lambda_t^s(Q^*))&
\nu{\vert_t}(\Lambda_t^s(T_2))
\end{bmatrix}
=
\begin{bmatrix}
a(t)-a(s)&\vert h(t)\vert (c(t)-c(s))
\\
\vert h(t)\vert (\overline{c}(t)-\overline{c}(s))&b(t)-b(s)
\end{bmatrix}
$$
is a positive matrix.  Since the determinant is positive we have
$$
\vert h(t)\vert^2\vert c(t) - c(s)\vert^2 \leq (a(t)-a(s))(b(t)-b(s))
$$
for $0 < t < s$.  Recalling the definition of $h(t)$ we have
\begin{align} \label{9.1}
\vert 1 - c(s)/c(t)\vert^2 \leq (1 - a(s)/a(t))(1 - b(s)/b(t)).
\end{align}
$$\quad \text{Now for} \quad x,y \in [0,1]\quad \text{we have} \quad x - 2\sqrt
{xy} + y = (\sqrt{x} - \sqrt{y})^2 \geq 0
$$
which yields
$$
(1- \sqrt{xy})^2 = 1 - 2\sqrt{xy} + xy \geq 1 - x - y + xy = (1 - x)(1 - y).
$$
Using this inequality with $x = a(s)/a(t)$ and $y = b(s)/b(t)$ in
inequality \eqref{9.1} we find
\begin{align} \label{9.2}
\vert 1-c(s)/c(t)\vert^2 \leq (1- \sqrt{\frac{a(s)b(s)} {a(t)b(t)}})^2
\end{align}
Note that $a(t) \geq a(s)$ and $b(t) \geq b(s)$ so we have
$$
\vert 1 - c(s)/c(t)\vert \leq (1 -  \sqrt{\frac {a(s)b(s)} {a(t)b(t)}})
$$
Since $1 - \vert z\vert \leq \vert 1 - z\vert$ for $z \in \cc$ we have
$$
-  \frac {\vert c(s)\vert} {\vert c(t)\vert}  \leq  -  \frac {a(s)b(s)}
{a(t)b(t) } .
$$
Recalling the definition of $h$ this inequality states $\vert h(t)\vert
\geq \vert h(s)\vert$ and we have $\vert h(t)\vert$ is non increasing and
since $|h(t)|$ is bounded above we have $\vert h(t)\vert$ converges to its least
upper bound $\kappa$ as $t \rightarrow 0+$.  Hence, we have shown that $\vert
h(t)\vert \rightarrow \kappa$ as $t \rightarrow 0+$ and $\vert h(t)\vert \leq
\kappa$ for all $t > 0$.  Since $h(t) \rightarrow 1$ as $t \rightarrow
\infty$ we have $\kappa \geq 1$.

Now we show the function $h(t)$ has a limit as $t \rightarrow 0+$.  Since the
absolute value of $h$ has a limit and that limit is not zero to show $h$
has a limit it is enough to show the reciprocal has a limit as $t
\rightarrow 0+$.  Let $w(t) = (a(t)b(t))^{1/2}$ for $t > 0$.
Multiplying inequality \eqref{9.2} by $\vert c(t)\vert^2$ we find
$$
\vert c(t) - c(s)\vert^2 \leq \vert c(t)\vert^2 (1 - w(s)/w(t))^2
$$
and expanding both sides and canceling terms we find
$$
- 2Re(c(t)c(s)) + \vert c(s)\vert^2 \leq \vert
c(t)\vert^2w(s)/w(t)(w(s)/w(t) - 2)
$$
and dividing both sides by $w(t)w(s)$ we find
$$
\frac {-2Re(c(t)c(s))} {w(t)w(s)}  + \frac {\vert c(s)\vert^2} {w(t)w(s)}
\leq \frac {w(s)\vert c(t)\vert^2} {w(t)^3} - \frac {2\vert c(t)\vert^2}
{w(t) ^2}
$$
Now
$$
\vert 1/h(t) - 1/h(s)\vert^2 = \frac {\vert c(t)\vert^2} {w(t)^2}  + \frac
{\vert c(s)\vert^2} {w(s)^2} - \frac {2Re(c(t)c(s))} {w(t)w(s)}
$$
and combining this with the above inequality we find
$$
\frac {w(s)\vert c(t)\vert^2} {w(t)^3} - \frac {\vert c(t)\vert^2} {w(t)^2}
+  \frac {\vert c(s)\vert^2} {w(s)^2} - \frac {\vert c(s)\vert^2} {w(t)w(s)
} \geq \vert 1/h(t) - 1/h(s)\vert^2
$$
And in recalling that $c(t)/w(t) = 1/h(t)$ this inequality becomes
$$
(1 - w(s)/w(t))(\vert 1/h(s)\vert^2 - \vert 1/h(t)\vert^2) \geq \vert 1/h(t)
- 1/h(s)\vert^2
$$
Since $w(s) \leq w(t)$ and $\vert h(t)\vert \rightarrow \kappa \geq 1$ as
$t \rightarrow 0+$ it follows that the left hand side of the above
inequality tends to zero as $s$ and $t$ approach zero so it follows that
$1/h(t) - 1/h(s)$ tends to zero as $s$ and $t$ approach zero and since $1
\leq \vert h(t)\vert \leq \kappa$ it follows that $h(t) \rightarrow
z\kappa$ as $t \rightarrow 0+$ where $\vert z\vert = 1$.

We have the map
$$
\psi_t(A) =
\begin{bmatrix}
\mu{\vert_t}(A)&\vert h(t)\vert\tau{\vert_t}(A)
\\
\vert h(t)\vert\tau^*{\vert_t}(A)&\nu{\vert_t}(A)
\end{bmatrix}
$$
is a completely positive map from $\frak A (\frak H )$ into $M_2(\C)$
for each $t > 0$ and, therefore, since $\mu, \nu$ and $\tau$ are boundary weights, 
the limiting map defined for all $A \in \ah$ by
$$
\psi_o(A) =
\begin{bmatrix}
\mu (A)&\kappa\tau (A)
\\
\kappa\tau^*(A)&\nu (A)
\end{bmatrix}
$$
is completely positive. By Lemma~\ref{extractor}, we have that $\kappa\gamma$ is an ordinary \iq corner from $\omega$ to $\eta$.

Conversely, suppose $\kappa \geq 1$ and $\kappa\gamma$ is a corner from
$\omega$ to $\eta$ and
$$
\vert h(t)\vert = \frac {(1+\mu\vert_t(\Lambda (T_1)))^{\frac{1}{2}} (1+\nu\vert
_t(\Lambda (T_2)))^{\frac{1}{2}}} {\vert 1+\tau\vert_t(\Lambda (Q))\vert} \leq
\kappa
$$
for all $t > 0$. Then, in particular, for all $t>0$, the scalar 
$1+\tau\vert_t(\Lambda (Q))$ is not zero. Therefore, if $\Theta = \omega + \gamma + \gamma^* +\eta$, then its generalized boundary representation $\Pi_t^\#$ is well-defined and it is given by
$$
\Pi_t^\# (A) =
\begin{bmatrix}
(1+\mu{\vert_t}(\Lambda (T_1)))^{-1}\mu{\vert_t}(A)T_1&
(1+\tau{\vert_t}(\Lambda (Q)))^{-1}\tau{\vert_t}(A)Q
\\
(1+\tau^*{\vert_t}(\Lambda (Q^*)))^{-1}\tau^*{\vert_t}(A)Q^*&
(1+\nu{\vert_t}(\Lambda (T_2)))^{-1}\nu{\vert_t}(A)T_2
\end{bmatrix}.
$$
By an argument using Lemma~\ref{extractor} analogous to the one above, the map $\Pi_t^\#$ is completely positive if and only if the map
$$
\psi_t(A) =
\begin{bmatrix}
\mu{\vert_t}(A)&\vert h(t)\vert\tau{\vert_t}(A)
\\
\vert h(t)\vert\tau^*{\vert_t}(A)&\nu{\vert_t}(A)
\end{bmatrix}
$$
is completely positive.  But this follows immediately from the fact that $\kappa \gamma$ is an \iq corner from $\omega$ to $\eta$ and $|h(t)| \leq \kappa$ for all $t>0$.
Therefore $\gamma$ is an \iq $q$-corner from $\omega$ to $\eta $.

Finally, we show $\kappa\gamma$ is a trivially maximal corner from $\omega$
to $\eta$ if $\omega$ and $\eta$ are unbounded.  Suppose $\lambda > 1$ and
$\lambda\kappa \gamma$ is a corner from $\omega$ to $\eta$.  From the
inequality for $\kappa$ and the fact that $\mu{\vert_t} (\Lambda (T_1))$,
$\nu{\vert_t} (\Lambda (T_2))$ and, thus, $\vert\tau{\vert_t} ( \Lambda
(Q))\vert$ tend to infinity as $t \rightarrow 0+$ we have
$$
\lim_{t\rightarrow 0+}  \frac {\mu\vert_t(\Lambda (T_1))\nu\vert_t(\Lambda
(T_2))} {\vert\tau\vert_t(\Lambda (Q))\vert^2} \leq \kappa^2.
$$
Since $\lambda\kappa\gamma$ is a corner we have $\lambda^2\kappa^2\vert\tau
{\vert_t} (\Lambda (Q))\vert^2 \leq \mu{\vert_t} (\Lambda
(T_1)) \nu{\vert_t} (\Lambda (T_2))$ for $t > 0$ and, therefore,
$$
\lim_{t\rightarrow 0+}  \frac {\mu\vert_t(\Lambda (T_1))\nu\vert_t(\Lambda
(T_2))} {\vert\tau\vert_t(\Lambda (Q))\vert^2} \geq \lambda^2\kappa^2 >
\kappa ^2.
$$
This contradicts the previous limit inequality so $\kappa\gamma$ is a
trivially maximal corner from $\omega$ to $\eta.$
\end{proof}

The following generalizes Theorem 3.14 of \cite{powers-holyoke} to $q$-pure
range rank one $q$-weight maps, and it provides a useful criterion for the determining the existence of $q$-corners.

\begin{theorem}  \label{rank-one-corner-form}   
Suppose $\omega$ and $\eta$ are completely
orthogonal $q$-pure range rank one $q$-weight maps on a finite-dimensional Hilbert space
$\fk$ so $\omega (\rho )(A) = \rho (T_1)\mu (A)$ and $\eta (\rho )(A) =
\rho (T_2)\nu (A)$ for $\rho \in \frak B (\frak K )_ *$ and $A \in \frak A
(\frak H )$ where $T_1$ and $T_2$ are non-zero projections and $\mu$ and
$\nu$ are $q$-pure $q$-weights on $\frak A (\frak H )$.  Then there is a
non zero range rank one $q$-corner $\gamma$ from $\omega$ to $\eta$ if and
only if there is a unitary operator $U \in \frak B (\frak K )$ and a $z >
0$ so that $UT_1U^* = T_2$ and $\mu$ and $\nu$ can be expressed in the form
$$
\mu (A) = \sum_{k\in I} (f_k,(I - \Lambda )^{-{\tfrac{1}{2}}}A(I - \Lambda
)^{-{\tfrac{1}{2}} }f_k)
$$
$$
\nu (A) = \sum_{k\in I} (g_k,(I - \Lambda )^{-{\tfrac{1}{2}}}A(I - \Lambda
)^{-{\tfrac{1}{2}} }g_k)
$$
for $A \in \frak A (\frak H )$ and $g_k = z(U\otimes I)f_k + h_k$, with $f_k$ in the range of 
$T_1 \otimes I$ and
$h_k$ in the range of $(I - \Lambda )^{\frac{1}{2}}(T_2 \otimes I)$ for $k \in I$ and
$$
\sum_{k\in I} \Vert (I - \Lambda )^{-{\frac{1}{2}}}h_k\Vert^2 < \infty .
$$
\end{theorem}
\begin{proof}  Before we begin the proof we remark that in the sums for
$\omega$ and $\eta$ we sum over the same index set $I$.  Even though we sum
over the same index set the sums for $\omega$ and $\eta$ can have different
numbers of non zero terms since some of the $f^{\prime}s$ or $g^{\prime}s$
can be zero.

Assume the setup and notation of the theorem. Let us assume that $g_k =$ $z(U\otimes I)f_k + h_k$, with $z > 0$, $f_k$  in the range of $T_1 \otimes I$, $h_k$ in the range of $(I - \Lambda )^{1/2}(T_2 \otimes I)$ for each $k \in I$ and the sum
involving the $h_k^{\prime}s$ given above converges.  Since the sum
involving the $h_k^{\prime}s$ above converges and $0 \leq \Lambda \leq I$
we have
$$
r = \sum_{k\in I} \Vert\Lambda^{\tfrac{1}{2}} (I - \Lambda
)^{-{\tfrac{1}{2}}}h_k\Vert^2 < \infty .
$$

Let $\tau\in \ah_*$ be the boundary weight given by
$$
\tau (A) = \sum_{k\in I} (f_k,(I - \Lambda )^{-{\tfrac{1}{2}}}A(I - \Lambda
)^{-{\tfrac{1}{2}} }g_k)
$$
and let
$$
\lambda = \frac {2z} {1+z^2+r}
$$
We shall prove that the map $\gamma: \bk_* \to \ah_*$ given by
$$
\gamma (\rho )(A) = \lambda\rho (T_1U^*)\tau (A)
$$
is an internal $q$-corner from $\omega$ to $\eta$. To that end, we will use Lemma~\ref{ordinary}. Notice that the operator $Q=T_1U^*$ satisfies \eqref{q-eqs}. 
Furthermore, $\gamma$ is an ordinary \iq corner from $\omega$ to $\eta$. Indeed, since $0< \lambda\leq 1$, it suffices to prove that  
$$
h(t) = \frac{(1+\mu\vert_t(\Lambda (T_1)))^{1/2}\;\; (1+\nu\vert_t(\Lambda
(T_2)))^{1/2}} {1+\tau\vert_t(\Lambda (T_1U^*))}
$$
is well-defined, satisfies $|h(t)| \leq \lambda^{-1}$ for all $t>0$ and 
$\lambda^{-1}\gamma$ is an ordinary \iq corner from $\omega$ to $\eta$. Let $\Theta = \omega + \lambda^{-1}\gamma + \lambda^{-1}\gamma^* + \eta$, and notice that using the matricial identification from Remark~\ref{2by2} we have that
$$
\brTheta(A) = \begin{bmatrix}
\mu(A)T_1& \tau(A)T_1U^* \\
\tau(A)UT_1 & \nu(A) T_2
\end{bmatrix}.
$$
By Lemma~\ref{extractor}, $\brTheta$ is completely positive if and only if the map
$$
\psi(A) = \begin{bmatrix}
\mu(A)& \tau(A) \\
\tau(A) & \nu(A) 
\end{bmatrix}
$$
is completely positive. But this follows immediately from observing that by the definition of $\mu$, $\nu$ and $\tau$, for all $B \in \bh$, if $A= (1-\Lambda)^{1/2}B(1-\Lambda)^{1/2}$,
$$
\psi(A) = \sum_{k\in I}
\begin{bmatrix}
(f_k,B f_k)& (f_k,B g_k)
\\
(g_k,B f_k)  & (g_k , B g_k)
\end{bmatrix}
$$
hence $\psi$ is a sum of completely positive maps. Finally, we shall prove that the function 
$h(t)$ defined in Lemma~\ref{ordinary} is well-defined, bounded, and $\sup |h(t)|
\leq \lambda^{-1}$. It suffices to show that 
\begin{align}\label{determinant}
\lambda^2\; \big(1+\mu{\vert_t} (\Lambda (T_1)) \big) \;
\big(1+\nu{\vert_t} (\Lambda(T_2)) \big)
\leq \big|1 + \lambda \tau{\vert_t} (\Lambda (T_1U^*)) \big|^2
\end{align}
since by taking the limit as $t \to 0+$ we obtain
\begin{equation}\label{det2}
\lambda^2\; \big(1+\mu (\Lambda (T_1)) \big) \;
\big(1+\nu (\Lambda(T_2)) \big)
\leq \big|1 + \lambda \tau (\Lambda (T_1U^*)) \big|^2
\end{equation}
This equation implies that 
$$
0< \lambda \leq \big|1 + \lambda \tau (\Lambda (T_1U^*)) \big|
$$
hence  $h(t)$ is well-defined for all $t>0$. Furthermore, also by equation \eqref{det2}, we have that $| h(t) | \leq \lambda^{-1}$,  for all $t>0.$
We proved above that $\lambda^{-1}\gamma$ is an ordinary \iq corner, hence assuming that \eqref{determinant} holds then we have by Lemma~\ref{ordinary} that $\gamma$ is an \iq $q$-corner. 

Let us prove \eqref{determinant}. By expanding it, we obtain 
\begin{align*}
\lambda^2(1+\mu{\vert_t} (\Lambda (T_1))+\nu{\vert_t} (\Lambda
(T_2))+\mu{\vert_t} (\Lambda (T_1))\nu{\vert_t} (&\Lambda (T_2))
-\vert\tau{\vert_t} (\Lambda (T_1U^*))\vert^2)
\\
& \leq 1 + 2\lambda Re(\tau{\vert_t} (\Lambda (T_1U^*))).
\end{align*}
To give this inequality a name so we can refer to it we will call
this inequality the determinant inequality.  We will need some
notation.  Let
$$
\zeta (A) = \sum_{k\in I} (f_k,(I - \Lambda )^{-{\tfrac{1}{2}}}A(I - \Lambda )^{-
{\tfrac{1}{2}}}h_k)
$$
and
$$
\vartheta (A) = \sum_{k\in I} (h_k,(I - \Lambda )^{-{\tfrac{1}{2}}}A(I - \Lambda
)^{-{\tfrac{1}{2}}}h_k)
$$
for $A \in \frak A (\frak K )$.  Then we have
$$
\nu (A) = z^2\mu ((U^*\otimes I)A(U\otimes I))+ z\zeta ((U^*\otimes I)A)+ z
\zeta^*(A(U\otimes I)) + \vartheta (A)
$$
and
$$
\zeta (A) = z\mu (A(U\otimes I)) + \vartheta (A)
$$
for $A \in \frak A (\frak H )$.  Then we have
$$
\nu{\vert_t} (\Lambda (T_2)) = z^2\mu{\vert_t} (\Lambda (T_1))
+ z\zeta{\vert_t} (\Lambda (T_1U^*)) + z\zeta^*{\vert_t} (\Lambda
(UT_1)) + \vartheta{\vert_t} (\Lambda (T_2))
$$
and
$$
\zeta{\vert_t} (\Lambda (T_1U^*)) = z\mu{\vert_t} (\Lambda (T_1))
+ \vartheta{\vert_t} (\Lambda (T_1U^*)).
$$
To simplify the determinant inequality let
$$
a = \mu{\vert_t} (\Lambda (T_1)),\qquad b = \zeta{\vert_t}
(\Lambda (T_1U^*))\qquad \text{and} \qquad c = \vartheta{\vert_t}
(\Lambda (T_2)).
$$
Then the determinant inequality becomes
$$
\lambda^2((1+a)(1+z^2a+zb+z\overline {b}+c)-\vert za+b\vert^2) \leq 1 +
2\lambda (az+Re(b))
$$
and with a slight further simplification this becomes
$$
\lambda^2(1+a+az^2+zb+z\overline {b}+c+ac-\vert b\vert^2) \leq 1 + 2\lambda
(az+Re(b)).
$$
Note that $c = \vartheta{\vert_t} (\Lambda (T_2)) \leq \vartheta
(\Lambda (T_2)) = r$ so if we replace $c$ by $r$ the term on the left hand
side of the above inequality does not decrease so if the above inequality
is satisfied with $c$ replaced by $r$ it will be satisfied.  Then to
establish the determinant inequality it is sufficient to prove that
\begin{equation} \label{10.1}
\lambda^2(1+a+az^2+2zRe(b)+r+ar-\vert b\vert^2) \leq 1 + 2\lambda
(az+Re(b)).
\end{equation}
Since $\lambda = 2z(1 + z^2 + r)^{-1}$ we have
\begin{equation} \label{10.2}
\lambda^2(1 + z^2 + r) = 2\lambda z.
\end{equation}
Multiplying both sides of this equation by a we find
\begin{equation} \label{10.3}
\lambda^2(a + az^2 + ar) = 2\lambda az.
\end{equation}
Since $\vert\lambda (z - b) - 1\vert^2 \geq 0$ we have
\begin{equation} \label{10.4}
\lambda^2(-z^2 + 2zRe(b) - \vert b\vert^2) \leq - 2\lambda z + 2\lambda Re(b)
+ 1.
\end{equation}
If we add equations \eqref{10.2} and \eqref{10.3} to inequality \eqref{10.4} we obtain
inequality \eqref{10.1}.  Hence, $\gamma$ is a non-zero rank one \iq $q$-corner
from $\omega$ to $\eta$.

Now we give the proof in converse direction so we assume $\gamma$ is a non-zero
 range rank one $q$-corner from $\omega$ to $\eta$.  From Lemma~\ref{ordinary}, 
we know $\gamma$ is of the form $\gamma (\rho )(A) = \rho (Q)\tau
(A)$ for $\rho \in \frak B (\frak K )_*$ and $A \in \frak A (\frak H )$ and
$\Vert Q\Vert = 1,\medspace QQ^* \leq T_1$ and $Q ^*Q \leq T_2$ and if
$$
h(t) = \frac {(1+\mu\vert_t(\Lambda (T_1)))^{\tfrac{1}{2}} (1+\nu\vert_t(\Lambda
(T_2)))^{\tfrac{1}{2}}} {1+\tau\vert_t(\Lambda (Q))}
$$
then there is a number $\kappa \geq 1$ so that $\vert h(t)\vert \leq
\kappa$ for all $t > 0$  and $\vert h(t)\vert \rightarrow \kappa$ as $t
\rightarrow 0+$ and $\kappa\gamma$ is an ordinary \iq corner from $\omega$ to 
$\eta$.

We will show there is a constant $K$ and complex number $x \neq 0$
and so that
\begin{equation} \label{10.5}
\vert x\vert^2\mu{\vert_t} (\Lambda (T_1)) + \nu{\vert_t}
(\Lambda (T_2)) - 2\kappa Re(\overline {x}\tau{\vert_t}
(\Lambda (Q))) \leq K 
\end{equation}

By Remark~\ref{2by2} and Lemma~\ref{extractor}, the map from $\frak A (\frak H )$ into 
$M_2(\cc)$ given by
$$
\psi_o(A) =
\begin{bmatrix}
\mu (A)&\kappa\tau (A)
\\
\kappa\tau^*(A)&\nu (A)
\end{bmatrix}
$$
is completely positive.  Now consider the family of matrices
$$
M_t = 
\begin{bmatrix}
1+\mu{\vert_t}(\Lambda (T_1))&
\kappa+\kappa\tau{\vert_t}(\Lambda (Q))
\\
\kappa+\kappa\tau^*{\vert_t}(\Lambda (Q^*))&
1+\nu{\vert_t}(\Lambda (T_2))
\end{bmatrix} 
$$
Note that if $0 < t < s$ then
$$
M_t - M_s  = 
\begin{bmatrix}
\mu (\Lambda_t^s(T_1))&\kappa\tau (\Lambda_t^s(Q))
\\
\kappa\tau^*(\Lambda_t^s(Q^*))&\nu (\Lambda_t^s(T_2))
\end{bmatrix}= \psi_0\big(\Lambda_t^s(T_1+Q+Q^*+T_2)\big)
$$
where $\Lambda_t^s(A) = E(t,s)\Lambda (A)E(t,s)$.  Since $\psi_o$ is
completely positive and $Q$ satisfies \eqref{q-eqs}, it follows that $M_t - M_s$ is a positive matrix.
Since $\vert h(t)\vert \leq \kappa$ for all $t > 0$ we have
$$
(1 + \mu{\vert_t} (\Lambda (T_1)))(1 + \nu{\vert_t} (\Lambda (T
_2))) \leq \vert\kappa + \kappa\tau{\vert_t} (\Lambda (Q))\vert^2.
$$
This tells us the determinant of $M_t$ is less than or equal to zero so for
each $t > 0$ there is a unit vector $v(t)$ so that $M_tv(t) = \lambda
(t)v(t)$ and $\lambda (t) \leq 0$.  We will prove there is a unit vector
$v$ so that $(v,M_tv) \leq 0$ for all $t > 0$.

Since the set of unit vectors in $\cc^2$ is compact in the norm topology
there is at least one accumulation point of $v_t$ as $t \rightarrow 0+$.
Let $v$ be such an accumulation point so for each $\epsilon > 0$ there is a
$t \in (0,\epsilon )$ with $\Vert v - v_t\Vert < \epsilon$.  We show
$(v,M_tv) \leq 0$ for all $t > 0$.  Suppose $t > 0$ and $\epsilon > 0$.
Let $\epsilon_1 = \min(\epsilon ,\medspace {\tfrac{1}{2}} \epsilon/\Vert M_t\Vert
)$.  Then there is an $s \in (0,\epsilon_1)$ with $\Vert v - v_s\Vert <
\epsilon_1$.  We have
$$
(v,M_tv) = (v,M_tv) - (v_s,M_tv_s) + (v_s,M_tv_s) - (v_s,M_sv_s) + (v_s,M_sv
_s).
$$
We have
\begin{align*}
(v,M_tv) - (v_s,M_tv_s) &= ((v - v_s),M_tv) + (v_s,M_t(v - v_s))
\\
&\leq 2\Vert v - v_s\Vert\medspace\Vert M_t\Vert < 2\epsilon_1\Vert
M_t\Vert \leq \epsilon .
\end{align*}
Since $M_t$ is non increasing and $0 < s < t$ we have
$$
(v_s,M_tv_s) - (v_s,M_sv_s) \leq 0.
$$
And we have
$$
(v_s,M_sv_s) = \lambda (s) \leq 0.
$$
Combining the previous four relation we find $(v,M_tv) \leq \epsilon$ and
since $\epsilon$ is arbitrary we have $(v,M_tv) \leq 0$ for all $t$.

Notice that the non-zero vector $v$ so that $(v,M_tv) \leq 0$ for
all $t > 0$ cannot be a multiple of $(1,0)$ or $(0,1)$ because in the
first case we have $(v,M_tv) = 1 + \mu{\vert_t} (\Lambda (T_1))$ and
the second case $(v,M_tv) = 1 + \nu{\vert_t} (\Lambda (T_2))$ and
neither of these are less than or equal to zero.  So the vector $v$ must be
a multiple of a vector of the form $w = (x,-1)$ with $x \neq 0$.  Then we
have $(w,M_tw) \leq 0$ which yields the promised inequality \eqref{10.5}
$$
\vert x\vert^2\mu{\vert_t} (\Lambda (T_1)) + \nu{\vert_t}
(\Lambda (T_1)) - 2\kappa Re(\overline {x}\tau{\vert_t} (\Lambda
(Q))) \leq 2\kappa Re(x) - 1 - \vert x\vert^2 = K
$$
for all $t > 0$.  

The fact that $\psi_o$ is completely positive on $\frak A
(\frak H )$ means the mapping $\phi (A) = \psi_o((I - \Lambda )^{\frac{1}{2}}
A(I - \Lambda )^{\frac{1}{2}} )$ is a completely positive map from $\frak B
(\frak H )$ into $\frak B (\Bbb C \oplus \Bbb C )$ and every such
completely positive map is of the form
$$
\phi (A) = \sum_{i\in I} S_iAS_i^*
$$
where the $S_i$ are linear maps from $\frak H$ to $\Bbb C \oplus \Bbb C$.
Since every such map is specified by a pair of vectors $f_i,g_i \in \frak
H$ it follows that

\begin{align}\label{form-mu-nu-tau1}
\mu(A) &= \sum_{i\in I} (f_i,(I - \Lambda )^{-{\tfrac{1}{2}}}A(I - \Lambda
)^{-{\tfrac{1}{2}} }f_i)
\\\label{form-mu-nu-tau2}
\nu (A) &= \sum_{i\in I} (g_i,(I - \Lambda )^{-{\tfrac{1}{2}}}A(I - \Lambda
)^{-{\tfrac{1}{2}} }g_i)
\\\label{form-mu-nu-tau3}
\kappa\tau (A) &= \sum_{i\in I} (f_i,(I - \Lambda )^{-{\tfrac{1}{2}}}A(I -
\Lambda )^{-{\tfrac{1}{2}}}g_i)
\end{align}
for $A\in \frak A (\frak H )$.  Furthermore, since $\omega$ and $\eta$ are completely orthogonal with respect to $T_1$ and $T_2$ and $\gamma$ is an \iq corner from $\omega$ to $\eta$, by replacing $f_i$ by $(T_1 \otimes I)f_i$ and replacing $g_i$ by $(T_2 \otimes I)g_i$, we still have that \eqref{form-mu-nu-tau1}, \eqref{form-mu-nu-tau2} and \eqref{form-mu-nu-tau3} hold.

Then from inequality \eqref{10.5} we have
\begin{equation} \label{10.6}
\sum_{i\in I} \vert x\vert^2(f_i^t,(T_1\otimes I)f_i^t) + (g_i^t,(T_2\otimes
I)g_i^t) - 2Re(xf_i^t,(Q\otimes I)g_i^t) \leq K 
\end{equation}
where $f_i^t = \Lambda^{\tfrac{1}{2}} (I - \Lambda )^{-{\tfrac{1}{2}}}E(t,\infty )f_i$
and $g_i^t = \Lambda^{\tfrac{1}{2}} (I - \Lambda )^{-{\tfrac{1}{2}}}E(t,\infty )g_i$
for $t > 0$ and $i \in I$.  We can write the above sum in the
form
$$
\sum_{i\in I} \Vert x(T_1\otimes I)f_i^t - (Q\otimes I)g_i^t\Vert^2 +
(g_i^t,[(T _2-Q^*Q)\otimes I]g_i^t) \leq K
$$
From this it follows that $\nu{\vert_t} (\Lambda (T_2-Q^*Q)) \leq K$
for all $t > 0$.  Now if $T_2 \neq Q^*Q$ there is a rank one projection $e
\in \frak B (\frak K )$ and a real number $y > 0$ that $ye \leq T_2 - Q^*Q$
from which it follows that $\nu{\vert_t} (\Lambda (e)) \leq K/y$ for
all $t > 0$.   But since $\eta$ is $q$-pure it follows from 
Theorem~\ref{rank-one-q-pure}
that $\nu (\Lambda (e)) = \infty$.  Hence, we conclude that $T_2 = Q^*Q$.

The sum \eqref{10.6} can also be written in the form
\begin{equation} \label{10.7}
\sum_{i\in I} \Vert x(Q^*\otimes I)f_i^t - (T_2\otimes I)g_i^t\Vert^2 +
\vert x\vert^2(f_i^t,[(T_1-QQ^*)\otimes I]f_i^t) \leq K 
\end{equation}
From which we conclude $\mu{\vert_t} (\Lambda (T_1 - QQ^*)) \leq
K\vert x\vert^{-2}$ for all $t > 0$.  Since $\omega$ is $q$-pure we
conclude by repeating the argument above that $T_1 = QQ^*$.  Hence, $Q$ is
a partial isometry from the $\range(T_2)$ to the $\range(T_1)$.  We construct number $z$ and the unitary $U$ as promised in the statement of the theorem. (Strictly speaking, we are assuming that $\dim \fk<\infty$, hence the argument is straightforward, however we describe a proof which works even when $\dim \fk = \infty$ anticipating future considerations).  Let 
$z = \vert x\vert$ and $u = x/\vert x\vert$ so $x = zu$.
Since $Q^*Q = T_2$ and $QQ^* = T_1$ it follows that the ranges of $T_1$ and
$T_2$ have the same dimension.  If $T_1$ is of finite rank then the
$\range(T_1)^\perp$ and $\range(T_2)^\perp$ have the same dimension and if $T
_1$ is of infinite rank then since $T_1$ and $T_2$ are disjoint it follows that 
$\range(T_1)^\perp$ and $\range(T_2)^\perp$ are both of infinite dimension.
So $\range(T_1)^\perp$ and $\range(T_2)^\perp$ have the same dimension.
We define $U$ as follows.  We define $U$ on $\range(T_1)$ as $UT_1 =
uQ^*$, and since $\range(T_1) ^\perp$ and $\range(T_2)^\perp$ have the
same dimension we can define a partial isometry $S$ from $\range(I - T_1)$
onto $\range(I - T_2)$.  Then we define $U$ on $\range(I - T_1)$ as  $U(I 
- T_1) = S(I - T_1)$.  Then we have $U$ is unitary and $UT_1U^* = T_2$ and
we have from inequality \eqref{10.7} that
\begin{equation} \label{10.8}
\sum_{i\in I} \Vert z(UT_1\otimes I)f_i^t - (T_2\otimes I)g_i^t\Vert^2
\leq K
\end{equation}
for all $t > 0$.  This is the important estimate which will yield the
estimate stated in the theorem when we use the normalization properties of
$\mu$ and $\nu$.  In subsequent calculation it is useful to have the
formulae
$$
x = zu, \quad z = \vert x\vert ,\quad  T_2U = UT_1 = uQ^*\qquad \text{and} \qquad
\overline {u}Q = T_1U^* = U^*T_2
$$
Now expressing inequality \eqref{10.8} inner product form we have
\begin{align} \label{10.8'}
\sum_{i\in I} &\vert x\vert^2((I-\Lambda )^{-{\tfrac{1}{2}}}E(t,\infty )f_i,\Lambda
(T_1)(I-\Lambda )^{-{\tfrac{1}{2}}}E(t,\infty )f_i)
\\
-& \sum_{i\in I}\quad \text{2Re} (x(I-\Lambda )^{-{\tfrac{1}{2}}}E(t,\infty )f_i,
\Lambda (Q)(I-\Lambda )^{-{\tfrac{1}{2}}}E(t,\infty )g_i)
\\
&\qquad + \sum_{i\in I} ((I-\Lambda )^{-{\tfrac{1}{2}}}E(t,\infty )g_i,\Lambda (T
_2)(I-\Lambda )^{-{\tfrac{1}{2}}}E(t,\infty )g_i) \leq K
\end{align}
We also have
\begin{equation} \label{10.9}
\mu{\vert_t} (I-\Lambda (T_1)) = \sum_{i\in I} ((I-\Lambda )^{-{\tfrac{1}{2}}
}E(t,\infty )f_i,(I-\Lambda (T_1))(I-\Lambda )^{-{\tfrac{1}{2}}}E(t,\infty )f_i)
\leq 1
\end{equation}
and
\begin{equation} \label{10.10}
\nu{\vert_t} (I-\Lambda (T_2)) = \sum_{i\in I} ((I-\Lambda )^{-{\tfrac{1}{2}}
}E(t,\infty )g_i,(I-\Lambda (T_2))(I-\Lambda )^{-{\tfrac{1}{2}}}E(t,\infty )g_i)
\leq 1
\end{equation}
We will need one more inequality. Notice that if $A =T_1 + Q + Q^* + T_2$, then 
$B=(A \otimes I) - \Lambda(A)$ is positive. Therefore, from the complete positivity of $\psi_o$, and using the fact that $\mu(B\vert_t) = \mu\vert_t(I - \Lambda(T_1))$, $\nu(B\vert_t) = \nu\vert_t(I - \Lambda(T_2))$ and 
$\tau(B\vert_t) = \tau\vert_t(  (u U^* \otimes I )- \Lambda(Q) )$,
we have that 
$$
0 \leq \psi_o(B\vert_t) = \begin{bmatrix}
\mu\vert_t (I - \Lambda(T_1))& \kappa\tau\vert_t (u U^* \otimes I - \Lambda(Q) )
\\
\kappa\tau^*\vert_t(\overline{u} U \otimes I - \Lambda(Q) )&\nu\vert_t (I -\Lambda(T_2))
\end{bmatrix}
$$
whence by considering the determinant we obtain
$$
\vert\kappa\tau\vert_t (u U^*\otimes I-\Lambda (Q))\vert^2 \leq \mu\vert_t (I-
\Lambda (T_1))\nu\vert_t (I-\Lambda (T_1)) \leq 1\cdot 1 = 1.
$$
This gives us the inequality
\begin{equation}\label{10.11}
- \sum_{i\in I}\quad \text{2Re} (x(I-\Lambda )^{-{\tfrac{1}{2}}}E(t,\infty )f_i,(
u U^*\otimes I-\Lambda (Q))(I-\Lambda )^{-{\tfrac{1}{2}}}E(t,\infty )g
_i) \leq 2\vert x\vert.
\end{equation}
Combining inequalities \eqref{10.8'} through \eqref{10.11}, we obtain
\begin{align*}
\sum_{i\in I} &\vert x\vert^2((I-\Lambda )^{-{\tfrac{1}{2}}}E(t,\infty
)f_i,(I-\Lambda )^{-{\tfrac{1}{2}}}E(t,\infty )f_i)
\\
-& \sum_{i\in I}\quad \text{2Re} (x(I-\Lambda )^{-{\tfrac{1}{2}}}E(t,\infty )f_i,(
u U^*\otimes I)(I-\Lambda )^{-{\tfrac{1}{2}}}E(t,\infty )g_i)
\\
&\qquad + \sum_{i\in I} ((I-\Lambda )^{-{\tfrac{1}{2}}}E(t,\infty )g_i,(I-\Lambda
)^{-{\tfrac{1}{2}}}E(t,\infty )g_i) \leq K + 1 + 2z + z^2
\end{align*}
And the three sums collapse into a simpler sum
$$
\sum_{i\in I} \Vert (I-\Lambda )^{-{\tfrac{1}{2}}}E(t,\infty )h_i\Vert^2 \leq K +
(1+z)^2
$$
where $h_i = z(U\otimes I)f_i - g_i$ for $i \in I$.  Since the sum is
independent of $t$ we have $h_i \in \range((I-\Lambda )^{\frac{1}{2}} ) =$
Domain of $(I-\Lambda )^{-{\frac{1}{2}}}$ and
$$
\sum_{i\in I} \Vert (I-\Lambda )^{-{\tfrac{1}{2}}}h_i\Vert^2 \leq K + (1+z)^2.
$$
This completes the proof of the theorem. 
\end{proof}

In the case when $\omega$ and $\eta$ have trivial normal spines, by Proposition~\ref{internal-to-external} and Theorem~\ref{corners-are-range-rank-one}, the only non-zero $q$-corners between them have to be range rank one. If they are unital, this leads to the following immediate corollary. 

\begin{cor}\label{the-corollary}
For $j=1,2$, let $\omega_j$ be a unital $q$-pure range rank one $q$-weight map over $\fk_j$ finite-dimensional with trivial normal spine. If $\dim \fk_1 \neq \dim \fk_2$, then $\omega_1$ and $\omega_2$ induce E$_0$-semigroups that are not cocycle conjugate.
\end{cor}

\section{The range rank two case}\label{sec-rank-two}

In this section we analyze range rank two $q$-weights.  This means that $\omega$
is of the form
\begin{equation*} 
\omega (\rho )(A) = \rho (e_1)\mu_1(A) + \rho (e_2)\mu_2(A)
\end{equation*}
for $\rho \in \frak B (\frak K )_*$ and $A \in \frak A (\frak H )$. 
Typically we assume that $e_1$ and $e_2$ are positive norm one operators of the following particular form.

\begin{prop}\label{standard-form-range-rank-two}
Let $\fk$ be a separable Hilbert space and let $\omega: \bk_* \to \bh_*$ be a unital $q$-weight map with range rank 2 and trivial normal spine. Then there exist positive $\mu_1, \mu_2\in \ah_*$ and $e_1, e_2 \in \bk$ positive operators with norm one such that for all $\rho \in \bk_*$ and $A \in \ah$,
$$
\omega (\rho )(A) = \rho (e_1)\mu_1(A) + \rho (e_2)\mu_2(A).
$$ 
Furthermore, given $x_1, x_2 \in \rr$, we have that $0\leq x_1 e_1 + x_2e_2 \leq I$ if and only if $x_1, x_2 \in [0,1]$. 
\end{prop}
\begin{proof}
Since $\omega$ has a trivial normal spine, by Theorem~\ref{twoz}  there exists $L: \bk \to \bk$ a boundary expectation for $\omega$. Since $\range(L)=\range(\bromega)$, we have that the Choi-Effros algebra $\rl$ is isomorphic to $\cc \oplus \cc$. Furthermore, notice that the unit of $\rl$ is given by $L(I)$. Let $e_1 \in \rl$ be a minimal central projection (with respect to the Choi-Effros multiplication), and notice that $e_1\neq 0$ and $e_1\neq L(I)$. Let $e_2 = L(I) - e_1$, also a minimal projection. Then we have that for all $X \in \range(L)$,
$$
X = e_1 \cem X \cem e_1 + e_2 \cem X \cem e_2 = L(e_1Xe_1) + L(e_2Xe_2) = 
\nu_1(X) e_1 + \nu_2(X) e_2
$$
for some states $\nu_1, \nu_2$ on $\rl$ since $e_1$ and $e_2$ are minimal projections in $\rl$. Thus, for all $A \in \ah$,
$$
\bromega(A) = \nu_1(\bromega(A)) e_1 + \nu_2(\bromega(A)) e_2
$$
Now let $\mu_j = \nu_j \circ \bromega$ for $j=1,2$. Those are clearly positive boundary weights.

The final part of the statement follows immediately from the observation that $e_1$ and $e_2$ are orthogonal projections in the algebra $\rl$ such that $e_1+e_2=I_{\rl}=L(I)$ and 
$L(I)=I$ since $\omega$ is unital. Indeed,  $\bromega(I-\Lambda)=I$ and $L$ fixes the range of $\bromega$.
\end{proof}

We recall the following notation. If $A \in \fb(\fh)$ and $t>0$, then we define the operator of $\fb(\fh)$
$$
A\vert_t = E(t,\infty) A E(t,\infty).
$$
We emphasize that in fact, for all $t>0$, we have that $A\vert_t \in \fa(\fh)$.

\begin{theorem} \label{rank-two-big}  Let $\fk$ be a separable Hilbert space. 
Suppose $e_1$ and $e_2$ are two positive
operators in $\frak B (\frak K )$ so that if $x_1$ and $x_2$ are
real numbers then
\begin{equation}\label{nc-square}
0 \leq x_1e_1 + x_2e_2 \leq I\quad \iff \quad x_1,x_2 \in[0,1].
\end{equation}
Let $\omega$ be a  range rank two boundary weight map of the form
\begin{equation}\label{2.8.1}
\omega (\rho )(A) = \rho (e_1)\mu_1(A) + \rho (e_2)\mu_2(A)
\end{equation}
for $\rho \in \frak B (\frak K )_*$ and $A \in \frak A (\frak H )$ where
$\mu _1$ and $\mu_2$ are positive boundary weights on $\frak A (\frak H
)$.  Let $h_1$ and $h_2$ be functions
\begin{equation}\label{name-h1-h2}
h_1(t) = \frac {\mu_1(\Lambda (e_2)\vert_t)} {1+\mu_2(\Lambda (e_2)\vert_t)
}\qquad \text{and} \qquad h_2(t) = \frac {\mu_2(\Lambda (e_1)\vert_t)}
{1+\mu _1(\Lambda (e_1)\vert_t)}
\end{equation}
defined for $t > 0$.  Let
\begin{equation}\label{def-kappa}
\kappa_1 = \sup(h_1(t): t > 0)\qquad \text{and} \qquad \kappa_2 =
\sup(h_2(t): t > 0)
\end{equation}
Then $\omega$ is a $q$-weight if and only if the following conditions are
satisfied.  The numbers $\kappa_1$ and $\kappa_2$ are in the closed
interval $[0,1 ]$ and the weights $\mu_1$ and $\mu_2$ satisfy the weight
inequalities
\begin{equation}\label{kappa-ineqs}
\mu_1 \geq \kappa_1\mu_2\qquad \text{and} \qquad \mu_2 \geq \kappa_2\mu_1
\end{equation}
and the numbers $x = \mu_1(I - \Lambda (e_1+e_2))$ and $y = \mu_2(I -
\Lambda (e_1+e_2))$ are in the set $S_o$ consisting of a parallelogram in
the $(x,y)$ plane with opposite vertices $(0,0)$ and $(1,1)$ and whose
sides are parallel to the lines $x = \kappa_1y$ and $y = \kappa_2x$.  In
the event that either $\kappa _1 = 1$ or $\kappa_2 = 1$ then the set $S_o$
consists of the line segment $0 \leq x = y \leq 1$.  If both $\kappa_1 < 1$
and $\kappa_2 < 1$ then set $S_o$ consists of the points satisfying the
inequalities
$$
0 \leq  x - \kappa_1y \leq 1 - \kappa_1\qquad \text{and} \qquad 0 \leq  y -
\kappa_2x \leq 1 - \kappa_2.
$$

Furthermore, in the event $\mu_1$ and $\mu_2$ satisfy these conditions then
the functions $h_1(t)$ and $h_2(t)$ are non increasing and $h_1(t)
\rightarrow \kappa_1$ and $h_2(t) \rightarrow \kappa_2$ as $t \rightarrow
0+$ and if either $\kappa _1 = 1$ or $\kappa_2 = 1$ then $\mu_1 = \mu_2$ so
$\omega$ is a range rank one $q$-weight.
\end{theorem}
\begin{proof}  Before we begin the proof we will establish some notation and basic facts
that we will use through out the proof.  Let $e_1, e_2 \in \bk$ be positive operators 
satisfying \eqref{nc-square}, let $\mu_1$ and $\mu_2$ be positive boundary weights on 
$\ah$ and suppose that $\omega$ is a boundary weight map of the form \eqref{2.8.1}. 
Then $\omega$ is a completely positive boundary weight map and its 
dualized boundary weight map $\bromega$ is well-defined.
Let $X(t)$ for $t > 0$ denote
the matrix representing the restriction of $\bromega\vert_t\Lambda$ to the invariant subspace of
$\bk$ spanned by $e_1, e_2$:
\begin{equation}\label{what-is-x}
X(t) = \begin{bmatrix} \mu_1(\Lambda (e_1)\vert_t)&\mu_1(\Lambda
(e_2)\vert_t)
\\
\mu_2(\Lambda (e_1)\vert_t)&\mu_2(\Lambda (e_2)\vert_t)
\end{bmatrix}.
\end{equation}
We denote the entries of $X(t)$ by $x_{ij}(t)$ for $i,j =
1,2$.  Sometimes in calculations we will write simply $X$ or
$x_{ij }$ instead of $X(t)$ or $x_{ij}(t)$.  When we do this we of course
mean that all terms in the equation are to be evaluated at the same $t$. 

Notice that $I+\bromega\vert_t\Lambda$ is invertible if and only if it is injective. Indeed,
 $\bromega\vert_t\Lambda$ is finite rank, therefore $I+\bromega\vert_t\Lambda$ is a Fredholm
 operator of index zero. Therefore, if it is injective then it is automatically surjective
 therefore invertible. Let us describe a convenient necessary and sufficient condition for 
its injectivity. Since the range of $\bromega\vert_t$ is an invariant subspace under $I+
 \bromega\vert_t\Lambda$, the same must hold for $(I+ \bromega\vert_t\Lambda)^{-1}$ whenever 
the inverse exists. Observe that if $T:\bk \to \bk$ is any operator, then $\ker(I+T) 
\subseteq \range(T)$, hence we have that $I+ \bromega\vert_t\Lambda$ is injective if
 and only if its restriction to the range of $\bromega\vert_t$ is injective. Thus $I+ \bromega\vert_t\Lambda$ is invertible if and only if
 $\det(I+X(t))\neq 0$. 

Let us assume that $\det(I+X(t))\neq 0$ for all $t>0$. In this 
case, the generalized boundary representation for $\omega$  can obtained by solving the 
equation $(I+X(t))\pi_t^\#=\bromega\vert_t$, from which we obtain that
$$
\pi_t^\#(A) = \ell^{(1)}_t(A)e_1 + \ell^{(2)}_t(A) e_2, \qquad \forall A \in \bh
$$
where
\begin{equation} \label{2.8.2}
\begin{bmatrix}
\ell^{(1)}_t(A) \\
\ell^{(2)}_t(A)
\end{bmatrix}
= \frac {1} {\det(I+X(t))} \begin{bmatrix} 1+x_{22}(t)&-x_{12}(t)
\\
-x_{21}(t)&1+x_{11}(t)
\end{bmatrix}
\begin{bmatrix} \mu_1\vert_t(A)
\\
\mu_2\vert_t(A)
\end{bmatrix}
\end{equation}
for all $A \in \bh$.
Furthermore, recall that $\omega$ is $q$-weight if and only if $\pi_t^\#$ is a completely
positive contraction for all $t > 0$. But observe that by equation~\eqref{nc-square}, we will have that $\pi^\#_t$ is positive if and only if both $\ell^{(1)}_t$ and $\ell^{(2)}_t$ are positive linear functionals. However, if those linear functionals are positive, then clearly $\pi_t^\#$ is completely positive since it is the sum of completely positive maps. Furthermore, when $\pi_t^\#$ is positive, we have also by \eqref{nc-square} that $\pi_t^\#$ is a contraction if and only if the map from $\bh$ to $\cc\oplus\cc$ given by \eqref{2.8.2} is a contraction. Thus, in order to simplify matters, we will abuse notation and denote also by
$\pi_t^\#$ the map from $\bh$ to $\cc\oplus\cc$ given by 
\begin{equation}\label{coord-pi-t}
\pi_t^\#(A) = \begin{bmatrix}
\ell^{(1)}_t(A) \\
\ell^{(2)}_t(A)
\end{bmatrix}.
\end{equation}
In summary, we conclude that the true generalized boundary representation 
is a completely positive contraction if and only if the map $\pi_t^\#$ given by \eqref{coord-pi-t} is a positive contraction.   For the remainder of the
proof we will refer to $\pi_t^\#$ given by \eqref{coord-pi-t} as if it were the true
generalized boundary representation.  Although this is not strictly true
the conclusions we draw for it are still valid for the reasons we have
given.

Notice in terms of the $x_{ij}(t)$ the functions $h_1$ and $h_2$ are 
\begin{equation} \label{2.8.3}
h_1(t) = \frac {x_{12}(t)} {1+x_{22}(t)}\qquad \text{and} \qquad h_2(t) =
\frac {x_{21}(t)} {1+x_{11}(t)}
\end{equation}

Our goal is to show that the conditions stated in the theorem are necessary
and sufficient for $\omega$ to be a $q$-weight.

For the forward direction of the proof of the theorem we assume that $\omega$ is a $q$-weight map,
which is equivalent to assuming $\pi_t^\#$ given above is a 
positive contraction for all $t > 0$.  Notice that since $\omega$ is a $q$-weight and thus $\omega(\rho)(I-\Lambda)\leq \rho(I)$ for all positive $\rho \in \bk_*$, we have that 
$$
\bromega(I-\Lambda) = \mu_1(I-\Lambda) e_1 + \mu_2(I-\Lambda) e_2 \leq I
$$
hence by \eqref{nc-square}, we have that $\mu_j(I-\Lambda) \leq 1$ for $j=1,2$, hence we find ourselves precisely in the framework described in the preamble of the current proof, and we have that $\det(I+X(t))\neq0$ for all $t>0$.

 The first problem we face is that we need to know the sign
of the $\det(I + X(t))$.  As we will see this determinant is never less
than one. Recall that if $\nu$ is any boundary weight, then for all $A\in \bh$, we have that for all $t>1$, $\nu\vert_t(A) = \nu\vert_1(A\vert_t)$ and $\nu\vert_1$ is normal, hence $\nu\vert_t(A) \to 0$ as $t\to \infty$.  Therefore $x_{ij}(t) \rightarrow 0$ since as $t \rightarrow \infty$ so  $\det(I + X(t))$ approaches one in the limit.  Hence, there is a
number $t_o > 0$ so that the determinant is greater than one half for $t >
t_o$.  Then for $t > t_o$
\begin{align*}
\pi_t^\# (\Lambda (e_1)) \geq 0\qquad & \textrm{yields} \qquad x_{11}(t) + \det(X(t))
\geq 0,
\\
\pi_t^\# (\Lambda (e_2)) \geq 0\qquad & \textrm{yields} \qquad x_{22}(t) + \det(X(t))
\geq 0.
\end{align*}
Since
$$
\det(I + X(t)) = 1 + x_{11}(t) + x_{22}(t) + \det(X(t)),
$$
we find
$$
\det(I + X(t)) \geq 1 + \max(x_{11}(t),x_{22}(t)) \geq 1
$$
Hence, we have $\det(I + X(t)) \geq 1$ for $t > t_o$.  Since this estimate
is independent of $t$ and the determinant is continuous in $t$ we have the
determinant is greater than one for all $t > 0$.  Having established this
fact we will use it in subsequent computations without comment.

Now making use of the fact that $\pi_t^\#$ is a contraction we have
$$
\pi_t^\# (\Lambda (e_1+e_2)) \leq 1
$$
which yields
\begin{equation} \label{2.8.4}
x_{12}(t) \leq 1 + x_{22}(t)\qquad \text{and} \qquad x_{21}(t) \leq 1 +
x_{11 }(t)
\end{equation}
for $t > 0$.  This then shows that $h_1(t) \leq 1$ and $h_2(t) \leq 1$ for
all positive $t$ and so we have $\kappa_1 \leq 1$ and $\kappa_2 \leq 1$.

The positivity of $\pi_t^\#$ yields the inequalities
\begin{align} \label{2.8.5}
\mu_1\vert_t &\geq \frac {x_{12}(t)} {1+x_{22}(t)} \mu_2\vert_t =
h_1(t)\mu_2\vert_t
\\ \label{2.8.5'}
\mu_2\vert_t &\geq \frac {x_{21}(t)} {1+x_{11}(t)} \mu_1\vert_t = h_2(t)\mu
_1\vert_t
\end{align}
for $t > 0$.  We claim $h_1$ and $h_2$ are non increasing functions of $t$.
Suppose $0 < t < s$.  Applying the above inequalities to the positive
elements $\Lambda (e_1)\vert_t - \Lambda (e_1){\vert_ s}$  and
$\Lambda (e_2)\vert_t - \Lambda (e_2){\vert_s}$  we find

\begin{align*}
x_{11}(t) - x_{11}(s) &\geq \frac {x_{12}(t)} {1+x_{22}(t)} (x_{21}(t) - x_
{21}(s))
\\
x_{12}(t) - x_{12}(s) &\geq \frac {x_{12}(t)} {1+x_{22}(t)} (x_{22}(t) - x_
{22}(s))
\\
x_{21}(t) - x_{21}(s) &\geq \frac {x_{21}(t)} {1+x_{11}(t)} (x_{11}(t) - x_
{11}(s))
\\
x_{22}(t) - x_{22}(s) &\geq \frac {x_{21}(t)} {1+x_{11}(t)} (x_{12}(t) - x_
{12}(s))
\end{align*}

Multiplying by the denominators in the middle two inequalities we find
\begin{align*}
x_{12}(t) + x_{12}(t)x_{22}(s) &\geq x_{12}(s) + x_{22}(t)x_{12}(s)
\\
x_{21}(t) + x_{21}(t)x_{11}(s) &\geq x_{21}(s) + x_{11}(t)x_{21}(s)
\end{align*}
which yields
$$
h_1(t) \geq h_1(s)\qquad \text{and} \qquad h_2(t) \geq h_2(s)
$$
for $0 < t < s$.  Hence, $h_1$ and $h_2$ are non increasing functions.
Since $h_1(t) \leq 1$ and $h_2(t) \leq 1$ for all $t > 0$ these functions
have limits as $t \rightarrow 0+$.  Hence, we have
$$
h_1(t) \rightarrow \kappa_1 \leq 1\qquad \text{and} \qquad h_2(t)
\rightarrow \kappa_2 \leq 1
$$
From inequalities \eqref{2.8.5} and \eqref{2.8.5'}, we see that
the inequalities \eqref{kappa-ineqs} hold.

Now we tackle the normalization conditions on the point in the $xy$-plane
given by
$$
x_o = \mu_1(I - \Lambda (e_1+e_2))\qquad \text{and} \qquad y_o = \mu_1(I -
\Lambda (e_1+e_2)).
$$

Since, $\pi_t^\#$ is a completely positive contraction we have $\pi_t^\# (I)$ is a positive contraction
for all $t > 0$ where $\pi_t^\# (I)$ is given by
$$
\pi^\#_t (I) = \frac {1} {\det(I+X(t))} \begin{bmatrix} 1+x_{22}(t)&-x_{12}(t)
\\
-x_{21}(t)&1+x_{11}(t)
\end{bmatrix}
\begin{bmatrix} \mu_1(I\vert_t)
\\
\mu_2(I\vert_t)
\end{bmatrix}
$$
We see $\pi_t^\#$ is a contraction if and only if the first and second
entries of $\pi_t^\#$ are contractions by \eqref{nc-square}.  Beginning with the top entry we
have
$$
(1+x_{22}(t))\mu_1(I\vert_t) - x_{12}(t)\mu_2(I\vert_t) \leq
(1+x_{11}(t))(1 +x_{22}(t)) - x_{12}(t)x_{21}(t))
$$
for $t > 0$.  Now we have $I\vert_t = (I - \Lambda (e_1+e_2))\vert_t +
\Lambda (e_1+e_2)\vert_t$ and putting this in the above inequality we have
\begin{align*}
((1+x_{22})\mu_1 - x_{12}\mu_2)((I-\Lambda (e_1+e_2))\vert_t) +&
(1+x_{22})(x _{11}+x_{12}) - x_{12}(x_{21}+x_{22})
\\
&\leq (1+x_{11})(1+x_{22}) - x_{12}x_{21}
\end{align*}
Canceling terms we see this inequality is equivalent to
$$
((1+x_{22}(t))\mu_1 - x_{12}(t)\mu_2)((I-\Lambda (e_1+e_2))\vert_t) \leq 1
+ x_{22}(t) - x_{12}(t)
$$
Now dividing by $(1 + x_{22}(t))$ we find
\begin{equation} \label{2.8.6}
(\mu_1 - h_1(t)\mu_2)((I-\Lambda (e_1+e_2))\vert_t) \leq 1 - h_1(t)
\end{equation}

By symmetry we see the condition that the bottom term in $\pi_t^\#$ be a
contraction is
\begin{equation} \label{2.8.7}
(\mu_2 - h_2(t)\mu_1)((I-\Lambda (e_1+e_2))\vert_t) \leq 1 - h_2(t)
\end{equation}

Now we show that if $h_1(t) = 1$ or $h_2(t) = 1$ for some $t > 0$
then $\mu_1 = \mu_2$.  Suppose for example $h_2(t_o) = 1$ for $t_o > 0$.
It follows that $\kappa_2 = 1$ and, therefore, $\mu_2 \geq \mu_1$ by \eqref{kappa-ineqs}.  Since
$h_1$ is non increasing we have $h_1(t) = 1$ for all $t \in (0,t_o]$.
Hence, by \eqref{2.8.7} we have
$$
(\mu_2 - \mu_1)((I-\Lambda (e_1+e_2))\vert_t) = 0
$$
for all $t \in (0,t_o)$.  But this means $(\mu_2 - \mu_1)(I-\Lambda
(e_1+e_2)) = 0$ and therefore,  $0 \leq (\mu_2 - \mu_1)(I-\Lambda ) \leq
(\mu_2 - \mu_1)(I-\Lambda (e_1+e_2)) = 0$.  Hence, $\mu_2 - \mu_1$ is a
positive weight which gives zero when applied to $I - \Lambda$ so $\mu_2 -
\mu_1 = 0$.  In the case where $\mu_1 = \mu_2$  we have $x_{11}(t) =
x_{21}(t)$ and $x_{12}(t) = x_{22}(t)$ and $\pi_t^\# (I)$ is
$$
\pi^\#_t (I) = \frac {1} {1+\mu (\Lambda (e_1+e_2)\vert_t)} \begin{bmatrix} \mu
(I\vert_t)
\\
\mu (I\vert_t)
\end{bmatrix}
$$
for $t > 0$ where $\mu = \mu_1 = \mu_2$.  Then $\pi_t^\#$ is a contraction
if and only if
$$
\mu ((I - \Lambda (e_1+e_2))\vert_t) \leq 1
$$
for all $t > 0$ and this is equivalent to $\mu (I - \Lambda (e_1+e_2)) \leq
1$.  Hence, we see that if $h_2(t) = 1$ for some $t > 0$ then $\mu_1(I -
\Lambda (e_1+e _2))$ and $\mu_2(I - \Lambda (e_1+e_2))$ satisfy the
normalization condition given in the statement of the theorem.  The same
argument with the indices 1 and 2 interchanged shows that if $h_1(t) = 1$
for some $t > 0$ then $\mu_1 = \mu_2$ and the normalization conditions on
$\mu_1$ and $\mu_2$ stated in the theorem are satisfied.

Now that we have taken care of the case where either $h_1(t) = 1$ or
$h_2(t) = 1$ for some $t > 0$ we now assume that both $h_1(t) < 1$ and
$h_2(t) < 1$ for all $t > 0$.  We will show that if $\kappa_1 = 1$ then the
normalization conditions of the theorem are satisfied.  Suppose $\kappa_1 =
1$.  Then inequality \eqref{2.8.6} becomes
$$
(\mu_1 - \mu_2)((I-\Lambda (e_1+e_2))\vert_t) + (1 -
h_1(t))\mu_2((I-\Lambda (e_1+e_2))\vert_t) \leq 1 - h_1(t)
$$
Since $\mu_1 - \mu_2 \geq 0$ and $h_1(t) < 1$ for all $t > 0$ we have
$$
\mu_2((I-\Lambda (e_1+e_2))\vert_t) \leq 1
$$
for all $t > 0$ and we conclude $\mu_2(I-\Lambda (e_1+e_2)) \leq 1$.  Then
we conclude from the above inequalities that
$$
(\mu_1 - \mu_2)((I-\Lambda (e_1+e_2))\vert_t) \leq (1 - h_1(t))(1 -
\mu_2((I -\Lambda (e_1+e_2))\vert_t))
$$
Since the right hand side of the above inequality converges to zero as $t
\rightarrow 0+$  we conclude that $0 \leq (\mu_1 - \mu_2)((I-\Lambda (e
_1+e_2))) \leq 0$ so $\mu_1 = \mu_2$ and as we have seen this then leads to
the conclusion that $\mu_1$ and $\mu_2$ satisfy the normalization condition
in the statement of the theorem. The same argument with the indices 1 and 2 
interchanged shows that if $\kappa_2 = 1$ then $\mu_2 = \mu_1$ and the 
normalization conditions are satisfied.  

Thus to complete the proof of the forward direction of the theorem, 
all that remains is the case where
both $\kappa_1 < 1$ and $\kappa_2 < 1$.  We assume this to be the case.
Then inequality \eqref{2.8.6}  and \eqref{kappa-ineqs} imply
that for $t > 0$,
$$
0 \leq (\mu_1 - \kappa_1\mu_2)((I-\Lambda (e_1+e_2))\vert_t) \leq 1 -
h_1(t),
$$
from which we conclude that
$$
0 \leq (\mu_1 - \kappa_1\mu_2)((I-\Lambda (e_1+e_2))) \leq 1 - \kappa_1.
$$
The same argument with the indices 1 and 2 interchanged shows
$$
0 \leq (\mu_2 - \kappa_2\mu_1)((I-\Lambda (e_1+e_2))) \leq 1 - \kappa_2.
$$
These inequalities show that $\mu_1((I-\Lambda (e_1+e_2)))$ and
$\mu_2((I-\Lambda (e_1+e_2)))$ lie in the parallelogram given in the
statement of the theorem.  Hence, we have shown if $\omega$ is a $q$-weight
of the form \eqref{2.8.1} then $\mu_1$ and $\mu _2$ satisfy the conditions given in
the statement of the theorem.

Now we prove the backward direction of the theorem. Let us assume that $\omega$ is of the form \eqref{2.8.1} and $\mu_1$ and $\mu_2$ satisfy
the conditions given in the statement of the theorem. Notice that we automatically have that
$\mu_j(I-\Lambda) \leq \mu_j(I - \Lambda(e_1+e_2)) \leq 1$, since this is one of the coordinates of a point $(x,y)$ in the parallelogram $S_o$. Thus we find ourselves again in the framework discussed in the preamble of the current proof.

First we note that
if $\kappa_1 = 1$ or $\kappa_2 = 1$ then $\mu_1 = \mu_2$.  Indeed, suppose
$\kappa_1 = 1$.  Then we have $\mu_1 - \mu_2 \geq 0$ and $(\mu_1 - \mu_2)(I
- \Lambda)\leq (\mu_1 - \mu_2)(I
- \Lambda (e_1+e_2)) = 0$  from which we conclude that $\mu_1 = \mu_2$.  As
we have seen from Theorem \ref{rank-one-q-weight}, if
$$
0 \leq \mu_1(I - \Lambda (e_1+e_2)) = \mu_2(I - \Lambda (e_1+e_2)) \leq 1
$$
then $\omega$ is a $q$-weight.  The same argument with the indices 1 and 2
interchanged shows that if $\kappa_2 = 1$ then $\omega$ is a range rank one
$q$-weight.

Then to complete the proof of the theorem we may assume both $\kappa_1 < 1$
and $\kappa_2 < 1$.  Let $X(t)$ be the matrix given by \eqref{what-is-x}.  Since $x_{12}(t) \leq \kappa_1(1 + x_{22}(t))$ and
$x_{21}(t) \leq \kappa_2(1 + x_{11}(t))$ we have
\begin{align*}
\det(I + X(t)) &= (1 + x_{11}(t))(1 + x_{22}(t)) - x_{12}(t)x_{21}(t)
\\
&\geq (1 + x_{11}(t))(1 + x_{22}(t))(1 - \kappa_1\kappa_2) \geq (1 - \kappa
_1\kappa_2)
\end{align*}
so we can conclude the determinant is positive for all $t > 0$, hence the generalized boundary representation of $\omega$ is well-defined as discussed in the preamble.  From
equation \eqref{2.8.2} we conclude the generalized boundary representation is
completely positive if and only if
\begin{equation} \label{some}
\mu_1\vert_t \geq h_1(t)\mu_2\vert_t\qquad \text{and} \qquad \mu_2\vert_t \geq
h_2(t)\mu_1\vert_t
\end{equation}
for all $t > 0$.  The conditions on $\mu_1$ and $\mu_2$ are $\mu_1 \geq
\kappa _1\mu_2,\medspace \mu_2 \geq \kappa_2\mu_1$, $h_1(t) \leq \kappa_1
\leq 1$ and $h_2(t) \leq \kappa_2 \leq 1$ and these conditions ensure that
\eqref{some} holds.  Hence, we conclude that the generalized boundary
representation is completely positive.

We recall that the proof that $h_1(t)$ and $h_2(t)$ are non increasing
functions of $t$ we only needed that these functions were bounded by one
and the positivity of $\pi_t^\#$.  Since these conditions are satisfied we
know that these functions are non increasing and, therefore, $h_1(t)
\rightarrow \kappa_1$ and $h_2(t) \rightarrow \kappa_2$ as $t \rightarrow
0+$.  All that remains to show is that $\pi_t^\#$ is contractive and for
this we only need show the first and second components of $\pi_t^\# (I)$ do
not exceed one and as we have already calculated this condition will be met
if and only if equations \eqref{2.8.6} and \eqref{2.8.7} hold
for all $t > 0$.  Now by the assumptions on $\mu_1$ and $\mu_2$ and the
fact that we are in the case $\kappa_1 < 1$ and $\kappa_2 < 1$ all the
terms above have limits and in the limit of $t \rightarrow 0+$ these
inequalities become
$$
0 \leq (\mu_1 - \kappa_1\mu_2)((I-\Lambda (e_1+e_2))) \leq 1 - \kappa_1
$$
and
$$
0 \leq (\mu_2 - \kappa_2\mu_1)((I-\Lambda (e_1+e_2))) \leq 1 - \kappa_2.
$$
We know these inequalities are satisfied because these inequalities are
precisely the inequalities that describe the parallelogram $S_o$ and by the
conditions of the theorem the point $x = \mu_1(I-\Lambda (e_1+e_2))$ and $y
= \mu_2(I-\Lambda (e_1+e_2))$ lies in $S_o$.  At first it may seem that
knowing the inequalities are satisfied in the limit will not be of much
help until we recall we have other inequalities at our disposal.  Since
$\mu_1 \geq \kappa_1\mu_2$ and $\mu_2 \geq \kappa_2\mu_1$ we have
$$
(\mu_1 - \kappa_1\mu_2)((I-\Lambda (e_1+e_2) - (I-\Lambda
(e_1+e_2)))\vert_t) \geq 0
$$
$$
(\mu_2 - \kappa_2\mu_1)((I-\Lambda (e_1+e_2) - (I-\Lambda
(e_1+e_2)))\vert_t) \geq 0
$$
$$
(\mu_1 - \kappa_1\mu_2)((I-\Lambda (e_1+e_2))\vert_t) \geq 0
$$
$$
(\mu_2 - \kappa_2\mu_1)((I-\Lambda (e_1+e_2))\vert_t) \geq 0
$$
and let us not forget the inequalities
$$
h_1(t) \leq \kappa_1 < 1\qquad \text{and} \qquad h_2(t) \leq \kappa_2 < 1
$$
for all $t > 0$.   To deal with these equations we will need to simplify
notation.  To this end let $x(t) = \mu_1((I-\Lambda (e_1+e_2))\vert _t)$
and $y(t) = \mu_2((I-\Lambda (e_1+e_2))\vert_t)$.  Then $x(t)$ and $y(t)$
are non increasing functions of $t$ and the inequalities we have have shown
to be true are
\begin{align}
\label{the-n-1}
0 \leq  x(0) - &\kappa_1y(0) \leq 1 - \kappa_1
\\
\label{the-n-2}
0 \leq  y(0) - &\kappa_2x(0) \leq 1 - \kappa_2
\\
\label{the-n-3}
x(0) - \kappa_1y(0) &\geq x(t) - \kappa_1y(t)
\\
\label{the-n-4}
y(0) - \kappa_2x(0) &\geq y(t) - \kappa_2x(t)
\\
\label{the-n-5}
x(t) \geq& \kappa_1y(t)
\\
\label{the-n-6}
y(t) \geq& \kappa_2y(t)
\end{align}
Now note that by equations \eqref{the-n-1}, \eqref{the-n-3} and \eqref{the-n-5} we have that
$$
0 \leq x(t) - \kappa_1y(t) \leq 1 -\kappa_1.
$$
Adding $(\kappa_1 - h_1(t))y(t)$ to both sides and using the fact that $y(t) \leq 1$ for all $t$, we obtain
\begin{equation}\label{useful-h1}
0\leq x(t) - h_1(t) y(t) \leq 1 -h_1(t).
\end{equation}
By using \eqref{the-n-2}, \eqref{the-n-4} and \eqref{the-n-6} in the analogous way, we obtain:
$$
0\leq y(t) - h_2(t) x(t) \leq 1 -h_2(t).
$$
These are precisely the desired inequalities.

 \end{proof}

\begin{theorem}  \label{2-not-qpure}  Let $\fk$ be a separable Hilbert space. Suppose $e_1$ and $e_2$ are two positive
operators in $\frak B (\frak K )$ so that $0 \leq x_1e_1 + x_2e_2
\leq I$ if and only if both $x_1$ and $x_2$ lie in the closed interval
$[0,1]$.  Let $\omega$ be a $q$-weight map of the form 
\begin{equation*}
\omega (\rho )(A) = \rho (e_1)\mu_1(A) + \rho (e_2)\mu_2(A)
\end{equation*}
for $\rho \in \frak B (\frak K )_*$ and $A \in \frak A (\frak H )$ where
$\mu _1$ and $\mu_2$ are positive boundary weights on $\frak A (\frak H )$
satisfying the conditions given in Theorem \ref{rank-two-big} which ensure $\omega$ is
a $q$-weight map.  If $\omega$ is $q$-pure then $\mu _1 = \mu_2$
so $\omega$ has range rank one.
\end{theorem}
\begin{proof}  Assume the hypothesis of the theorem and suppose $\mu_1 \neq \mu_2$.
Then the constants $\kappa_1$ and $\kappa_2$ defined in the previous
theorem satisfy $\kappa _1 < 1$ and $\kappa_2 < 1$.  Let us consider the boundary weight maps given by
$$
\eta (\rho ) = \rho (e_1)(\mu_1 - \kappa_1\mu_2)\qquad \text{and} \qquad
\nu (\rho ) = \rho (e_2)(\mu_2 - \kappa_2\mu_1)
$$
for $\rho \in \frak B (\frak K )_*$.  Let us show first that $\eta$ is a $q$-weight, as the proof for $\nu$ is analogous. By Theorem~\ref{rank-two-big}, $\mu_1 - \kappa_1\mu_2$ is a positive boundary weight. Hence, by Theorem~\ref{rank-one-q-weight}, since $e_1$ is a positive norm one operator, in order to obtain that $\eta$ is a $q$-weight map, it suffices to show that 
$$
(\mu_1 - \kappa_1\mu_2)(I-\Lambda(e_1)) \leq 1.
$$
Using the notation for the matrix $X(t)$ as in \eqref{what-is-x}, and denoting 
$x(t)=\mu_1((I-\Lambda(e_1+e_2))\vert_t)$, $y(t)=\mu_2((I-\Lambda(e_1+e_2))\vert_t)$
we have that
\begin{align*}
(\mu_1 - \kappa_1\mu_2)((I-\Lambda(e_1))\vert_t) & = (\mu_1 - \kappa_1\mu_2)\big((I-\Lambda(e_1+e_2) + \Lambda(e_2))\vert_t\big) \\
& = x(t) - \kappa_1 y(t) + x_{12}(t) -\kappa_1 x_{22}(t) \\
\text{by \eqref{name-h1-h2}}\;\; & = x(t) - \kappa_1 y(t) + h_1(t) (1+x_{22}(t)) - \kappa_1 x_{22}(t) \\
\text{by \eqref{useful-h1}}\;\;& \leq 1- \kappa_1 + h_1(t) \big(1+x_{22}(t)\big) - \kappa_1 x_{22}(t) \\
& = 1 - (\kappa_1 - h_1(t))\big(1 + x_{22}(t)\big) \\
& \leq 1
\end{align*}
Thus we have the desired inequality by taking the limit. By Theorem~\ref{rank-one-q-weight}, the generalized boundary representation $\xi^\#$ for $\eta$ is given by
$$
\xi_t^\# (A) = \frac{(\mu_1 -
\kappa _1\mu_2)(A\vert_t)}{1+x_{11}(t)-\kappa_1x_{21}(t)} \; e_1
$$
Now let $\pi^\#$ be the generalized boundary representation for $\omega$. Then, by Theorem~\ref{boundary-representation-subordinates}, $\eta \leq_q \omega$ if and only if for all $t>0$, $\pi_t^\#-\xi_t^\#$ is completely positive. Since $e_1$ and $e_2$ satisfy \eqref{nc-square}, in order to prove that $\pi_t^\#-\xi_t^\#$ is completely positive, it suffices to prove that 
$$
\ell^{(1)}_t - \frac{(\mu_1 -
\kappa _1\mu_2)\vert_t}{1+x_{11}(t)-\kappa_1x_{21}(t)}
$$
is positive (where $\ell^{(1)}_t$ is defined in \eqref{2.8.2}). In other words, we need to check that
\begin{align*}
((1+x_{11}-\kappa_1x_{21})(1+x_{22})&-\det(I+X))\mu_1\vert_t
\\
&+ (\kappa_1\det(I+X)-x_{12}(1+x_{11}-\kappa_1x_{21}))\mu_2\vert_t \geq 0.
\end{align*}
With some simplification this yields
$$
x_{21}(x_{12}-\kappa_1(1+x_{22}))\mu_1\vert_t + (\kappa_1(1+x_{11}+x_{22}+x
_{11}x_{22})-x_{12}(1+x_{11}))\mu_2\vert_t \geq 0.
$$
Dividing by $(1+x_{22})(1+x_{11})$ we find this is equivalent to
$$
h_2(t)(h_1(t)-\kappa_1)\mu_1\vert_t + (\kappa_1 - h_1(t))\mu_2\vert_t \geq 0.
$$
And this inequality can be written in the form
$$
(\kappa_1 -  h_1(t))(\mu_2 - h_2(t)\mu_1)\vert_t \geq 0.
$$
Now notice that this inequality holds because $\kappa_j \geq h_j(t)$ for $j=1,2$ and $\mu_2 \geq \kappa_2 \mu_1$ by \eqref{def-kappa} and \eqref{kappa-ineqs}. Therefore $\eta\leq_q\omega$, and by an analogous argument with indices 1 and 2 exchanged we have that $\nu\leq_q\omega$.
 It is immediately apparent that neither
$\eta \geq_q \nu$ nor $\nu \geq_q \eta$ so the subordinates of $\eta$ are
not well ordered.  Hence, $\omega$ is not $q$-pure.  \end{proof}

We remark that we do not know that all range rank two $q$-weights have the form assumed in the theorem, 
namely in terms of positive weights $\mu_1, \mu_2$ and positive operators $e_1, e_2$ such that $0 \leq x_1e_1 + x_2e_2 \leq I$
 if and only if $x_1, x_2 \in [0,1]$. However, at least in the interesting case when the $q$-weight is unital and has trivial normal spine, in other words it gives rise to an 
E$_0$-semigroup of type II$_0$, this is guaranteed by 
Proposition~\ref{standard-form-range-rank-two}. Thus we immediately have the following 
result as an application of Theorem~\ref{2-not-qpure}. 

\begin{cor}\label{sec6-main}
Let $\omega$ be a unital $q$-weight map over a separable Hilbert space $\fk$
with range rank two. If its induced E$_0$-semigroup  has type II$_0$ then it is not $q$-pure.
\end{cor}

\begin{theorem} \label{1or4}
If $\omega$ is a $q$-weight map over $\cc^2$ with trivial normal spine, 
then $\omega$ has range rank 1, 2, or 4.
Furthermore, if $\omega$ is unital and $q$-pure, then it has range rank 1 or 4.  
\end{theorem}
\begin{proof}
Suppose that $\dim \range(\breve{\omega}) >2$, and let $L$ be a boundary expectation
for $\omega$, so $\dim \range(L)=\dim \range(\bromega)>2$
by Theorem~\ref{twoz}.  Recall that the range $\mathfrak{R}_L$ of $L$
is a $C^*$-algebra under the norm and involution inherited from $M_2(\C)$ but with 
Choi-Effros multiplication given by \eqref{multiplication}.
Let $\mathfrak{B}$
be the $C^*$-subalgebra of $M_2(\C)$ generated by the range of $L$ in the usual sense. 
Linearity of $L$ and the definition of the multiplication operation in \eqref{multiplication} imply that the restriction of $L$ to $\mathfrak{B}$ is a 
$*$-homomorphism onto $\mathfrak{R}_L$.  

Since $\dim \range(L)>2$ and no $C^*$-subalgebra of 
$M_2(\C)$ has dimension 3, it follows
that $\mathfrak{B}=M_2(\C)$, hence $L$ is a $*$-homomorphism from $M_2(\C)$ onto $\mathfrak{R}_L$.
Notice that  $M_2(\C)$ is simple, hence $\rl$ is either $\{0\}$ or $M_2(\C)$.  Since $\bromega$ is not the zero map, we have $\rl=M_2(\cc)$.  Therefore, if $\omega$ is a $q$-weight map over $\cc^2$ whose normal spine is zero, then $\omega$ has range rank 1, 2, or 4.
To finish the proof, we note that by Corollary~\ref{sec6-main}, there is no $q$-pure unital $q$-weight map with range rank 2 over $\cc^2$
that has trivial normal spine.
\end{proof}

\providecommand{\bysame}{\leavevmode\hbox to3em{\hrulefill}\thinspace}

\end{document}